\documentclass[12pt]{article}
\usepackage{amsthm,amsmath,amssymb}
\usepackage{natbib}
\usepackage{multirow}
\usepackage[pdftex]{graphicx}
\usepackage{subfigure}
\usepackage{float}
\usepackage{makecell}
\usepackage{booktabs}
\usepackage{array}
\usepackage{fullpage}
\usepackage{url}
\usepackage{algorithm}
\usepackage{algorithmic}
\usepackage{bm}
\usepackage{smile}
\usepackage{lipsum}
\usepackage{mathrsfs}
\usepackage{dsfont}
\usepackage{titling}
\usepackage{epstopdf}
\usepackage{caption}
\usepackage{mathtools}

\usepackage{natbib}
\usepackage{multirow}

\usepackage[usenames,dvipsnames,svgnames,table]{xcolor}
\usepackage[colorlinks=true,
            linkcolor=blue,
            urlcolor=blue,
            citecolor=blue]{hyperref}

\newcommand*{\Sb}{\boldsymbol{S}}

\newcommand*{\supp}{\mathrm{supp}}

\begin{document}

\title{\LARGE A Unified Theory of Confidence Regions and Testing for High Dimensional Estimating Equations}

\author{Matey Neykov\thanks{Department of Operations Research and Financial Engineering, Princeton University, Princeton, NJ 08544; e-mail: {\tt\{mneykov,yning,hanliu\}@princeton.edu}} \and Yang Ning\protect\footnotemark[1] \and Jun S. Liu\thanks{Department of Statistics, Harvard University, Cambridge, MA 02138, USA; e-mail: {\tt jliu@stat.harvard.edu}} \and Han Liu\protect\footnotemark[1]}

\date{}

\maketitle

\vspace{-0.5in}

\begin{abstract}
	We propose a new inferential framework for constructing confidence regions  and testing hypotheses in statistical models specified by a system of high dimensional estimating equations. We construct an influence function by projecting the fitted estimating equations to a sparse direction obtained by solving a large-scale linear program. Our main theoretical contribution is to establish a unified Z-estimation theory of confidence regions for high dimensional problems. 
	Different from existing methods, all of which require the specification of the likelihood or pseudo-likelihood, our framework is likelihood-free. As a result, our approach provides valid inference for a broad class of high dimensional constrained estimating equation problems, which are not covered by existing methods. 
	Such examples include, noisy compressed sensing, instrumental variable regression, undirected graphical models, discriminant analysis and vector autoregressive models. We present detailed theoretical results for all these examples. Finally, we conduct thorough numerical simulations, and a real dataset analysis to back up the developed theoretical results.
\end{abstract}


\section{Introduction}

Let us observe a sample of $n$, $q$-dimensional random vectors $\{\bZ_i\}_{i = 1}^n$. Denote with $\Zb$ the $n \times q$ data matrix obtained by stacking all the vectors $\bZ_i$. 
Let $\tb(\Zb, \bbeta): \RR^{n\times q} \times \RR^d \mapsto \RR^d$ be an estimating equation \citep{godambe1991estimating} for  a $d$-dimensional unknown parameter $\bbeta$, and further let $E_{\tb}(\bbeta) = \lim_{n \rightarrow \infty} \EE \tb(\Zb, \bbeta)$ denote the limiting expected value of the estimating equation $\tb(\Zb, \bbeta)$ as $n \rightarrow \infty$. 
As an example given $n$ i.i.d. observations $\bZ_i$ and an equation $\hb$, this reduces to the classical Z-estimation setup $\tb(\Zb, \bbeta) = n^{-1} \sum_{ i = 1}^n \hb(\bZ_i, \bbeta)$ and $E_{\tb}(\bbeta) = \EE \hb(\Zb, \bbeta)$.
For the purpose of parameter estimation, it is usually assumed that the estimating equation is unbiased in the sense that the true value $\bbeta^*$ is the unique solution to $E_{\tb}(\bbeta) = 0$. When the dimension $d$ is fixed and much smaller than the sample size $n$, inference on $\bbeta^*$ can be obtained by solving the estimating equation $\tb(\Zb, \bbeta)=0$, and the asymptotic properties follow from the classical Z-estimation theory \citep{van2000asymptotic}.  
However, when $d > n$, directly solving $\tb(\Zb, \bbeta)=0$ is an ill-posed problem. To avoid this problem, a popular approach is to impose the sparsity assumption on  $\bbeta^*$, which motivates constrained Z-estimators in the following generic form \citep{cai2014geometrizing}:
\begin{align}
	\hat \bbeta = \argmin\|\bbeta\|_1, \mbox{ subject to } \|\tb(\Zb, \bbeta)\|_{\infty} \leq \lambda, \label{mainoptimization}
\end{align}
where $\lambda$ is a regularization parameter. 

Assume that we can partition $\bbeta$ as $(\theta, \bgamma^T)^T$, where $\theta$ is a univariate parameter of interest and $\bgamma$ is a $(d-1)$-dimensional nuisance parameter. Similarly, we denote $\hat \bbeta = (\hat \theta, \hat \bgamma^T)^T$ and $\bbeta^* = (\theta^*, \bgamma^{*T})^T$. The goal of this paper is to develop a general estimating equation based framework to obtain valid confidence regions for $\theta^*$ under the regime that $d$ is much larger than $n$. %As an example if we observe $n$ i.i.d. data points $\bZ_i$ and an estimating equation  $\hb$, one can take $\tb(\Zb, \bbeta) = n^{-1} \sum_{ i = 1}^n \hb(\bZ_i, \bbeta)$ and $E_{\tb}(\bbeta) = \EE \hb(\Zb, \bbeta)$, where the expectation is taken under the true distribution $\bZ \sim \PP_{\bbeta^*}$. 
The proposed framework has a large number of applications. For instance, given a convex and smooth loss function (or negative log-likelihood) $\ell: \RR^q\times \RR^d \mapsto \RR$, with i.i.d. data $\bZ_i$, the inference on $\bbeta$ can be conducted based on the estimating equation (or score function) $\tb(\Zb, \bbeta)= n^{-1}\sum_{i=1}^n\frac{\partial \ell(\bZ_i, \bbeta)}{\partial \bbeta}$. Hence, inference on many high dimensional problems with specifications of the loss function or the likelihood can be addressed through our framework. More importantly, the estimating equation method has an advantage over likelihood methods in that it usually only requires the specification of a few moment conditions rather than the entire probability distribution \citep{godambe1991estimating}. To see the advantage of our framework,  we consider the following examples, which are naturally handled by estimating equations.

%In the following, we compactly review a few examples, whose natural formulation coincides with (\ref{mainoptimization}).%Such problems include multiple examples which are often used by both theoreticians and practitioners and we compactly review a few of them below. 

\subsection{Examples}\label{special:cases:EE}

%We now give several concrete examples which serve to better illustrate the abstract concepts above.

\vspace{0.3cm}
\noindent\textbf{Linear Regression via Dantzig Selector \citep{candes2007dantzig}.} Assume that a linear model (also referred to as noisy compressed sensing) is specified by the following moment condition $\EE(Y\mid\bX)= \bX^T\bbeta^*$. Let $\Xb \in \RR^{n \times d}$ be the design matrix stacking the i.i.d. covariates $\{\bX_i\}_{i = 1}^n$ and $\bY\in \RR^n$ be the response vector with independent entries $Y_i$.  Given the moment condition, we can easily construct the estimating equation as $\tb((\bY,\Xb), \bbeta) = n^{-1}\Xb^T(\Xb \bbeta - \bY)$ and $E_{\tb}(\bbeta) = \EE \tb((\bY,\Xb), \bbeta)$. In addition, $E_{\tb}(\bbeta)$ has the true value $\bbeta^*$ as its unique root, provided that the second moment matrix $\bSigma_{\bX} := n^{-1} \EE \Xb^T \Xb$ is positive definite. In the high dimensional setting, \cite{candes2007dantzig} estimated $\bbeta$ by the following Dantzig selector, 
$$
\hat \bbeta = \argmin \|\bbeta\|_1, \mbox{ such that } \|n^{-1}\Xb^T(\Xb\bbeta - \bY)\|_{\infty} \leq \lambda.
$$

\vspace{0.3cm}
\noindent\textbf{Instrumental Variables Regression (IVR).} Similarly to the previous example, consider the linear model $Y = \bX^T\bbeta^* + \varepsilon$. In many economic applications, it is not always reasonable to believe that the error and the design variables are uncorrelated, i.e., $\EE[\bX \varepsilon] = 0$, which is a key condition ensuring the unbiasedness of the estimating equation and consequently the consistency of the Dantzig selector $\hat \bbeta$. In such cases one may use a set of \textit{instrumental variables} $\bW \in \RR^d$ which are correlated with $\bX$ but satisfy $\EE[\bW \varepsilon] = 0$ and $\EE[\varepsilon^2 \mid \bW] = \sigma^2$. Let $\Xb, \Wb \in \RR^{n \times d}$ be the design matrix and instrumental variable matrix stacking the i.i.d. covariates $\{\bX_i\}_{i = 1}^n$ and instrumental variables $\{\bW_i\}_{i = 1}^n$ respectively, and $\bY\in \RR^n$ be the response vector with independent entries $Y_i$.  Using the instrumental variables, one can construct unbiased estimating equations $\tb((\bY,\Xb, \Wb), \bbeta) = n^{-1}\Wb^T(\Xb \bbeta - \bY)$ and $E_{\tb}(\bbeta) = \EE \tb((\bY,\Xb,\Wb), \bbeta)$. In addition, $E_{\tb}(\bbeta)$ has $\bbeta^*$ as its unique root, provided that the second moment matrix $\bSigma_{\bW\bX} := n^{-1} \EE \Wb^T \Xb$ is of full rank. Inspired by \cite{gautier2011high}, we consider the following estimator $\hat \bbeta$:
$$
\hat \bbeta = \argmin \|\bbeta\|_1, \mbox{ such that } \|n^{-1}\Wb^T(\Xb\bbeta - \bY)\|_{\infty} \leq \lambda.
$$
%The instrumental variables regression can of course be viewed as a generalization of the linear regression considered above, upon substituting $\bW \equiv \bX$. We defer the development of the inference theory for IVR to Appendix \ref{dantzigselector:IVR}.

\vspace{0.3cm}
\noindent\textbf{Graphical Models via CLIME/SKEPTIC \citep{cai2011constrained,liu2012transelliptical}.} Let $\bX_1,...,\bX_n$ be i.i.d. copies of $\bX\in\RR^d$ with $\EE(\bX)=0$ and $\Cov(\bX)=\bSigma_{\bX}$. It is well known that in the case when $\bX$ are Gaussian, the precision matrix $\bOmega^*=(\bSigma_{\bX})^{-1}$ induces a graph, encoding conditional independencies of the variables $\bX$. More generally, this observation can be extended to transelliptical distributions \citep{liu2012transelliptical}. %To estimate the parameter $\Omega^*_{1m}$ for some $m \in \{1,\ldots, d\}$, \cite{cai2011constrained} suggested solving the following optimization problem: %When $\bX$ are coming form a Gaussian distribution, the confidence intervals for $\Omega^*_{1m}$ provide uncertainty assessment on whether $X_1$ is independent of $X_m$ given the rest of the variables. 

%There are a number of recent works considering the inferential problems for  Gaussian graphical models \citep{jankova2014confidence,chen2015asymptotically,ren2015asymptotic,liu2013} and Gaussian copula graphical models \citep{gu2015local,barber2015rocket}. Our framework differs from these existing procedures in the following two aspects. First, our method is based on the estimating equations rather than the likelihood and (node-wise) pseudo-likelihood. Second, we only require each component of $\bX$ is sub-Gaussian, whereas the majority of the existing methods require the data to be sampled from Gaussian or Gaussian copula distributions. 

Let $\bSigma_n = n^{-1} \sum_{i = 1}\bX_i^{\otimes 2}$ be the sample covariance of $\bX_1,...,\bX_n$. Based on the second moment condition $\bSigma_{\bX}\bOmega^*=\Ib_d$, \cite{cai2011constrained} proposed the CLIME estimator of $\bOmega^*$:
\begin{align} \label{generalCLIMEopt}
\hat{\bOmega}=\argmin\|\bOmega\|_1,~~\textrm{subject to}~~ \|\bSigma_n\bOmega-\Ib_d\|_{\max}\leq\lambda.
\end{align}
In this case we have $\tb(\Xb, \bOmega) = \bSigma_n \bOmega - \Ib_d$, and $E_{\tb}(\bOmega) = \bSigma_{\bX} \bOmega - \Ib_d$. Under the more general setting of transelliptical graphical models, \cite{liu2012transelliptical} substituted the sample covariance $\bSigma_n$ with a non-parametric estimate based on Kendall's tau (see Remark \ref{trans:ell:graph:model}). Doing so breaks down the i.i.d. decomposition of the estimating equation described above, but continues to belong to our formulation (\ref{mainoptimization}). 

\vspace{0.3cm}
\noindent\textbf{Discriminant Analysis \citep{cai2011direct}.} Let $\bX$ and $\bY$ be $d$-dimensional random vectors, coming from two populations with different means $\bmu_1=\EE(\bX)$, $\bmu_2=\EE(\bY)$, and a common covariance matrix $\bSigma=\Cov(\bX)=\Cov(\bY)$. Given some training samples, we are interested in classifying a new observation $\bO$ into population $1$ or population $2$. It is well known (e.g. see \cite{mardia1979multivariate} Theorem 11.2.1) that, under certain conditions, the Bayes classification rule takes the form:
$$
\psi(\bO) = I((\bO - \bmu)^T\bOmega \bdelta > 0),
$$
where $I(\cdot)$ is an indicator function, $\bmu = (\bmu_1 + \bmu_2)/2$, $\bdelta = (\bmu_1 - \bmu_2)$ and $\bOmega = \bSigma^{-1}$. Specifically, the observation $\bO$ is classified into population 1 if and only if $\psi(\bO) = 1$.

To implement $\psi(\bO)$ in practice, one has to estimate the unknown parameters $\bmu_1, \bmu_2$ and $\bOmega$. Assume we observe $n_1$ and $n_2$ training samples from population 1 and population 2 denoted by $\bX_1,\ldots, \bX_{n_1}\in\RR^d$ and $\bY_1,\ldots, \bY_{n_2}\in\RR^d$. We  assume that 
\begin{align}
\bX_i = \bmu_1 + \bU_i, i = 1,\ldots, n_1, ~~~\textrm{and}~~~\bY_i = \bmu_2 + \bU_{i + n_1}, i = 1,\ldots,n_2,\label{sub-gauss:assump:LDP} 
\end{align}
where $\bU_i$ are i.i.d. copies of $\bU =  (U_1, \ldots, U_d)^T$, which satisfies $\EE(\bU)=0$ and $\Cov(\bU)=\bSigma$. Define the sample means as $\bar \bX = \frac{1}{n_1} \sum_{i = 1}^{n_1}  \bX_{i}$ and $\bar \bY = \frac{1}{n_2} \sum_{i = 1}^{n_2}  \bY_{i}$, and the sample covariances as $\hat \bSigma_{\bX} = \frac{1}{n_1} \sum_{i = 1}^{n_1}  (\bX_{i} - \bar \bX)^{\otimes 2}$ and $\hat \bSigma_{\bY} = \frac{1}{n_2} \sum_{i = 1}^{n_2}  (\bY_{i} - \bar \bY)^{\otimes 2}$. Furthermore let $\hat \bSigma_n = \frac{n_1}{n}\hat  \bSigma_{\bX} + \frac{n_2}{n}\hat \bSigma_{\bY}$ be the weighted average of $\hat  \bSigma_{\bX}$ and $\hat \bSigma_{\bY}$. 

In the high dimensional setting with $d \gg n$, we cannot directly estimate $\bOmega$ by $\hat \bSigma_n^{-1}$, since the sample covariance is not invertible. Noting that the classification rule solely depends on $\bbeta^* = \bOmega \bdelta$, \cite{cai2011direct} proposed a direct approach to estimate $\bbeta^*$, rather than estimating $\bOmega$ and $\bdelta$ separately. Their estimated classification rule is as follows:
\begin{align}
\hat \psi(\bO) & =  I((\bO - (\bar \bX + \bar \bY)/2)^T \hat \bbeta  > 0), ~~ \mbox{ where } \nonumber \\
\hat \bbeta & = \argmin\|\bbeta\|_1,~~~\textrm{subject to}~~~ \|\hat \bSigma_n \bbeta - (\bar \bX - \bar \bY )\|_{\infty} \leq \lambda. \label{sparseLDAeq}
\end{align}
Clearly the latter formulation constitutes a high dimensional estimating equation as in (\ref{mainoptimization}), with $\tb((\{\bX_i\}_{i = 1}^{n_1}, \{\bY_i\}_{i = 1}^{n_2}), \bbeta) = \hat \bSigma_n \bbeta - (\bar \bX - \bar \bY)$ and $E_\tb(\bbeta) = \bSigma \bbeta - (\bmu_1 - \bmu_2)$.

\vspace{0.3cm}
\noindent\textbf{Vector Autoregressive Models \citep{han2014direct}.} Let $\{\bX_t\}_{t = -\infty}^{\infty}$ be a stationary sequence of mean 0 random vectors in $\RR^d$ with covariance matrix $\bSigma$. The sequence $\{\bX_t\}_{t = -\infty}^{\infty}$ is said to follow a lag-1 autoregressive model if
$$
\bX_t = \Ab^T \bX_{t - 1} + \bW_t, ~~~ t \in \ZZ:=\{...,-1,0,1,....\},
$$
where $\Ab$ is a $d\times d$ transition matrix, and the noise vectors $\bW_t$ are i.i.d. with $\bW_t \sim N(0, \bPsi)$ and independent of the history $\{\bX_s\}_{s < t}$. Under the additional assumption that $\det(\Ib_d - \Ab^T z) \neq 0$ for all $z \in \mathcal{\CC}$ with $|z| \leq 1$, it can be shown that $\bPsi$ can be selected so that the process is stationary, i.e. for all $t$: $\bX_t \sim N(0,\bSigma)$. Let $\bSigma_i := \Cov(\bX_0, \bX_i)$, where $\bSigma_0 := \bSigma$. A simple calculation under the lag-1 autoregressive model leads to the following Yule-Walker Equation:
$\bSigma_i := \bSigma_0 \Ab^i$, for any $i\in\NN$.
A special case of the above equation with $i=1$ yields that:
\begin{equation}\label{eqA}
\Ab = \bSigma_0^{-1} \bSigma_1.
\end{equation}
Assume that the data $(\bX_1,...,\bX_T)$ follow the lag-1 autoregressive model. By the equation (\ref{eqA}), \cite{han2014direct} proposed the following estimator of $\Ab$ in the high dimensional setting:
\begin{align}
\hat \Ab = \argmin_{\Mb \in \RR^{d \times d}} \sum_{1\leq j,k\leq d} |M_{jk}|, \mbox{ subject to } \|\Sb_0 \Mb - \Sb_1\|_{\max} \leq \lambda, \label{optproblemvecauto}
\end{align}
where $\lambda > 0$ is a tunning parameter, $\Sb_0 = 1/T\sum_{t = 1}^T \bX_t^{\otimes 2}$ and $\Sb_1 = \frac{1}{T - 1} \sum_{t = 1}^{T-1} \bX_t \bX_{t+1}^T$  are estimators of $\bSigma_0$ and $\bSigma_1$ respectively, and $T$ is the number of observations. In this case we have that $\tb(\{\bX_t\}_{t = 1}^T, \Mb) = \Sb_0 \Mb - \Sb_1$, and $E_\tb(\Mb) = \bSigma_0 \Mb - \bSigma_1$.

\subsection{Related Methods}\label{related:methods:sec}

Having explored a few  examples falling into the estimating equation framework (\ref{mainoptimization}), we move on to outline some related works on high dimensional inference.
%In a parallel line of work, where a likelihood function is available, one could opt for estimating the sparse parameter $\bbeta$ through a penalized likelihood. The archetypical example of such approach in the linear model is the LASSO \citep{tibshirani1996regression}. The statistical consistency and variable selection properties of the LASSO have also been successfully studied in the literature; see \cite{bickel2009simultaneous,bunea2007sparsity,van2008high, meinshausen2009lasso, meinshausen2006high,zhao2006model,wainwright2009sharp}, among others. The properties of the Dantzig selector are established by  
Recently, significant progress has been made towards understanding the post-regularization inference for the LASSO estimator in the linear and generalized linear models. For instance, \cite{lockhart2014significance, taylor2014post, lee2013exact} suggested conditional tests based on covariates which have been selected by the LASSO. We stress the fact that this type of tests are of fundamentally different nature compared to our work.  
Another important class of methods is based on the bias correction of $L_1$ or nonconvex regularized estimators.
In particular, \cite{belloni2012sparse, belloni2013honest,zhang2014confidence, javanmard2013confidence, van2013asymptotically} proposed a double selection estimator, low dimension projection estimator, debiasing and desparsifying correction methods, respectively, for constructing confidence intervals in high dimensional models with the $L_1$ penalty. Recently, \cite{sparc2014ning,ning2014general} proposed a decorrelated score test in a likelihood based framework. The difference between our method and this class of methods will be discussed in more details in the next section.
A different score related approach is considered by \cite{voorman2014inference}, which is testing a null hypothesis depending on the tuning parameter, and hence differs from our work. 
For the nonconvex penalty, under the oracle properties, the asymptotic normality property of the estimators is established by \cite{fan2011nonconcave}, which requires strong conditions, such as the minimal signal condition. In contrast, our work does not rely on oracle properties or variable selection consistency. 
P-values and confidence intervals based on sample splitting and subsampling are suggested by \cite{meinshausen2009p, meinshausen2010stability, shah2013variable,wasserman2009high}. However, the sample splitting procedures may lead to certain efficiency loss. In a recent paper by \cite{lu2015confidence}, the authors developed a new inferential method based on a variational inequality technique for the LASSO procedure which provably produces valid confidence regions. In contrast to our work their method needs the dimension $d$ to be fixed, and it may not be applicable to the inference problem based on the formulation (\ref{mainoptimization}).

\subsection{Contributions}

%Performing inference for the CLIME estimator has implications in graphical modeling. If the data is Gaussian, then such hypothesis testing is equivalent to edge testing in the graph structure. More generally our framework can be used to perform inference for Transelliptical graphical models, suggested by \cite{liu2012transelliptical}. To the best of our knowledge, non of the aforementioned algorithms has been equipped with inferential procedures. 

Our first contribution is to propose a new procedure for high dimensional inference in the estimating equation framework. In order  to construct confidence regions, our method is to project the general estimating equation onto a certain sparse direction, which can be easily estimated by solving a large-scale linear program. Thus, the proposed inferential procedure is a general methodology and can be directly applied to many inference problems, including all aforementioned examples. Although some of the works discussed in Section \ref{related:methods:sec} also used a projection idea, our method is  different from previous works in that it directly targets the influence function of the estimating equation. For instance, in \cite{sparc2014ning,ning2014general}, the decorrelated score function is defined as $n^{-1} \sum_{i = 1}^n [\partial \ell(\bZ_i, \bbeta)/\partial \theta-\wb^T\partial \ell(\bZ_i, \bbeta)/\partial \bgamma]$, where $\ell(\bZ_i, \bbeta)$ is the log-likelihood for data $\bZ_i$, and $\wb^T\partial \ell(\bZ_i, \bbeta)/\partial \bgamma$ is the sparse projection of the $\theta$-score function $\partial \ell(\bZ_i, \bbeta)/\partial \theta$ to the $(d-1)$-dimensional nuisance score space $\textrm{span}\{\partial \ell(\bZ_i, \bbeta)/\partial \bgamma\}$. While the score functions can be treated as a special case of estimating equation, such a construction cannot be directly extended to general estimating equations. The reason is that it remains unclear how to disentangle the estimating equation for the parameter of interest and the nuisance estimating equation space, and therefore the projection conducted by \cite{sparc2014ning,ning2014general} is not well defined. To address this challenge, motivated from the classical influence function representation, we propose a different projection approach, which directly estimates the influence function of the equation. %where there exists a one-to-one correspondence between equations and parameters, i.e. when a loss function $\ell$ is present the equation corresponding to the parameter $\theta$ is $n^{-1} \sum_{i = 1}^n \frac{\partial \ell(\bZ_i, \bbeta)}{\partial \theta}$. In contrast, in the estimating equation formulation (\ref{mainoptimization}) the notion of parameter equation correspondence is lost when the function $\ell$ is unknown or non-existent. Yet another difference with likelihood-based inference is the formulation of the optimization (\ref{mainoptimization}), which although related to regularized loss functions, produces an estimate with different overall properties. 

Our second contribution is to establish a unified Z-estimation theory of confidence intervals. %Under the null hypothesis, we establish uniform weak convergence of the test statistic to a normal distribution over a sufficiently sparse parameter set.  Moreover, we study the local power of our proposed test statistic and demonstrate that the same transition phenomenon as in the low dimensional case occurs. 
In particular, we construct a Z-estimator $\tilde \theta$ that is consistent and asymptotically normal, and its asymptotic variance can be consistently estimated. Furthermore,  %In contrast to the one-step estimators proposed by \cite{sparc2014ning,ning2014general,van2013asymptotically}, our Z-estimator may not have a closed form such that its technical development requires   empirical process theory.
the pointwise asymptotic normality results can be strengthened by showing that $\tilde \theta$ is uniformly asymptotically normal for $\bbeta^*$ belonging to a certain parameter space (deferred to the Supplementary Material). Moreover, owing to the flexibility of the estimating equations framework, we are able to push the theory through for non i.i.d. data, relaxing the assumptions made in most existing work. In terms of relative efficiency, when the estimating equation corresponds to the score function, our estimator $\tilde \theta$ is semiparametrically efficient. The theoretical properties of hypothesis tests have also been established, but for space limitations the proofs will be omitted and can be provided by the authors upon request. %and for simplicity we defer them to the Supplementary Material. 

%and the test statistic is asymptotically optimal in the class of unbiased estimating equations.

Our third contribution is to apply the proposed framework to establish theoretical results for the previous motivating examples including the noisy compressed sensing with moment condition, instrumental variable regression, graphical models, transelliptical graphical models, linear discriminant analysis and vector autoregressive models. To the best of our knowledge, many of the aforementioned problems (e.g. instrumental variables regression, linear discriminant analysis and vector autoregressive models) have not been equipped with any inferential procedures. 

Finally, we further emphasize the difference between our method and the class of methods based on the bias correction of regularized estimators. 
Compared to these methods in \cite{zhang2014confidence, javanmard2013confidence, van2013asymptotically,ning2014general}, our framework has the following three advantages. First, all of the above propositions rely on the existence of a likelihood, or more generally a loss function. In contrast, our framework directly handles the estimating equations and is likelihood-free, enabling us to perform inference in many examples (e.g., the motivating examples discussed in Section \ref{special:cases:EE}) where the likelihood or the loss function is unavailable or difficult to formulate. For instance in the instrumental variable regression it is not clear how to devise a loss function, while the problem naturally falls into the realm of estimating equations. This leads to different methodological development from the previous work, which will be explained later in details. Second, some of the existing work such as \cite{zhang2014confidence, javanmard2013confidence, van2013asymptotically} is only tailored for the linear and generalized linear models. In contrast, our framework covers  much broader class of statistical models specified by estimating equations, such as linear discriminant analysis and vector autoregressive models whose inferential properties have not been studied before. Third, our framework can handle dependent data, which is more flexible than these existing works all of which require the data to be independent.

\subsection{Organization of the Paper}

The paper is organized as follows. In Section \ref{LZEHDSsection}, we  propose our generic inferential procedure for high dimensional estimating equations, and layout the foundations of the general theoretical framework. In Section \ref{secexample}, we apply the general theory to study the motivating examples including the Dantzig Selector, instrumental variables regression, graphical models, discriminant analysis and  autoregressive models. Numerical studies and a real data analysis are presented in Section \ref{numericalStudySec}, and a discussion is provided in Section \ref{discussionSec}.

\subsection{Notation} The following notations are used throughout the paper. For a vector $\vb = (v_1, \ldots, v_d)^T\in \RR^d$, let $\|\vb\|_q = (\sum_{i = 1}^d v_i^q)^{1/q},  1 \leq q < \infty$, $\|\vb\|_0 = | \supp(\vb)|$, where $\supp(\vb) = \{j: v_j \neq 0\}$, and $|A|$ denotes the cardinality of a set $A$. Furthermore let $\|\vb\|_{\infty} = \max_{i} |v_i|$ and $\vb^{\otimes 2} = \vb \vb^T$. For a matrix $\Mb$ denote with $\Mb_{*j}$ and $\Mb_{j*}$ the $j$\textsuperscript{th} column and row of $\Mb$ correspondingly. Moreover, let $\|\Mb\|_{\max} = \max_{ij} |M_{ij}|$, $\|\Mb\|_p = \max_{\|\vb\|_p = 1} \|\Mb \vb\|_p$ for $p \geq 1$. If $\Mb$ is positive semidefinite let  $\lambda_{\max}(\Mb)$ and $\lambda_{\min}(\Mb)$ denote the largest and smallest eigenvalues correspondingly. For a set $S \subset \{1,\ldots d\}$ let $\vb_S = \{v_j : j \in S\}$ and $S^c$ be the complement of $S$.  We denote with $\phi, \Phi, \overline \Phi$ the pdf, cdf and tail probability of a standard normal random variable correspondingly. Furthermore, we will use $\rightsquigarrow$ to denote weak convergence.

For a random variable $X$, we define its $\psi_{\ell}$ norm for any $\ell \geq 1$ as:
\begin{align} \label{psi1norm}
	\|X\|_{\psi_\ell} = \sup_{p \geq 1} p^{-1/\ell} (\EE|X|^p)^{1/p}.
\end{align} 
In the present paper we mainly use the $\psi_1$ and $\psi_2$ norms. Random variables with bounded $\psi_1$ and $\psi_2$ norms are called \textit{sub-exponential} and \textit{sub-Gaussian} correspondingly \citep{vershynin2010introduction} . It can be shown that  a random variable is sub-exponential if there exists a constant $K_1 > 0$ such that $\PP(|\bX| > t) \leq \exp(1 - t/K_1)$ for all $t \geq 0$. Similarly, a random variable is sub-Gaussian, if there exists a $K_2 > 0$ such that $\PP(|\bX| > t) \leq \exp(1 - t^2/K^2_2)$ for all $t \geq 0$. Finally, for two sequences of positive numbers $\{a_n\}$ and $\{b_n\}$ we will write $a_n \asymp b_n$ if there exist positive constants $c, C > 0$ such that $\limsup_n a_n/b_n \leq C$ and $\liminf_n a_n/b_n \geq c$. 
%Finally we recall the definition of \textit{restricted eigenvalue} (RE) assumption \citep{bickel2009simultaneous}.
%
%\begin{definition}[RE]\label{RE} We say that the symmetric positive semi-definite matrix $\Mb_{k\times k}$ possesses the restricted eigenvalue property if:
%	$$\operatorname{RE}_\Mb(s, \xi) = \min_{S \subset \{1,\ldots,k\}, |S| \leq s} \min \left\{\frac{\ub^T \Mb \ub}{\|\ub_{S}\|^2_2} : \ub \in \mathbb{R}^d\setminus\{0\}, \|\ub_{S^c}\|_1 \leq \xi \|\ub_{S}\|_1 \right\} > 0.$$
%\end{definition}

\section{High Dimensional Estimating Equations} \label{LZEHDSsection}

%While in formulation (\ref{mainoptimization}) it is assumed that the data is i.i.d., in this section and throughout the paper we will discuss estimating equations under a more general setting. Assume that we observe an $n \times q$ data matrix $\Zb$. Let $\tb(\Zb, \bbeta): \RR^{n\times q} \times \RR^d \mapsto \RR^d$ be an estimating equation, and let $E_{\tb}(\bbeta)$ be a deterministic functional of $\tb$, which represents a local (in $\bbeta$) concentrating value of the estimating equation as $n \rightarrow \infty$. We require that $\bbeta^*$ is the unique solution to $E_{\tb}(\bbeta) = 0$. As an example if we observe $n$ i.i.d. data points $\bZ_i$, one can take $\tb(\Zb, \bbeta) = n^{-1} \sum_{ i = 1}^n \hb(\bZ_i, \bbeta)$. In this case $E_{\tb}(\bbeta) = \EE \hb(\bZ, \bbeta)$. 

In this section we present the intuition behind the construction of our projection, and formulate the main results of our theory. Recall that $\bbeta = (\theta, \bgamma^T)^T \in \RR^{d}$, where $\theta$ is a univariate parameter of interest and $\bgamma$ is a $(d-1)$-dimensional nuisance parameter. We are interested in constructing a confidence interval for $\theta$. In fact our results can be extended in a simple manner to cases with $\btheta$ being a finite and fixed-dimensional vector, but we do not pursue this development in the present manuscript. Throughout the paper, we assume without loss of generality that $\theta$ is the first component of $\bbeta$. 
 
In the conventional framework, where the dimension $d$ is fixed and less than the sample size $n$, one can estimate the $d$-dimensional parameter $\bbeta$ by the Z-estimator, which is the root (assumed to exist) of the following system of $d$ equations \citep{godambe1991estimating},
\begin{align} \label{EE}
	\tb(\Zb, \bbeta) = 0.
\end{align}
Under certain regularity conditions, the Z-estimator is consistent, and one has the following influence function expansion of the parameter $\hat \theta$, where $\hat \bbeta = (\hat \theta, \hat \bgamma^T)^T$ is the solution to (\ref{EE}) \citep{newey1994large,van2000asymptotic}:
\begin{align}\label{low:d:intuition}
\sqrt{n}(\hat \theta - \theta^*) = -\sqrt{n} [E_{\Tb}(\bbeta^*)]^{-1}_{1*} \tb(\Zb, \bbeta^*) + o_p(1).
\end{align}
%In the preceding display $E_{\Tb}(\bbeta)$ is a $d\times d$ functional, which is a concentration value of $\Tb(\Zb, \bbeta)$ where $\Tb(\Zb, \bbeta) := \frac{\partial}{\partial \bbeta}\tb(\Zb, \bbeta)$. 
In the preceding display, we assume $E_{\Tb}(\bbeta) := \lim_{n \rightarrow \infty} \EE \Tb(\Zb, \bbeta)$ is invertible, where $\Tb(\Zb, \bbeta) := \frac{\partial}{\partial \bbeta}\tb(\Zb, \bbeta)$.
%and we have assumed it to be (locally) invertible in (\ref{low:d:intuition}). In fact equation (\ref{low:d:intuition}) can be interpreted as a manifestation of the implicit function theorem, and 
It is noteworthy to observe that in contrast to the Hessian matrix of the log-likelihood (or more generally any smooth loss function), the Jacobian matrix $\Tb(\Zb, \bbeta)$ need not be symmetric in general (refer to the IVR model for an example). Under further conditions, the right hand side of (\ref{low:d:intuition}) converges to a normal distribution, hence guaranteeing the asymptotic normality of the estimator $\hat \theta$.

In the case when $d > n$, the estimating equation (\ref{EE}) is ill-posed as one has more parameters than samples, resulting in multiple solutions for $\bbeta$. To deal with such situations, under the sparsity assumption on $\bbeta^*$, we solve the constrained optimization program (\ref{mainoptimization}). Due to the constraint in (\ref{mainoptimization}), the limiting distribution of the estimator $\hat\bbeta$, and $\hat \theta$ in particular, becomes intractable as expansion (\ref{low:d:intuition}) is no longer valid. Hence, instead of focusing on the left hand side of (\ref{low:d:intuition}), in order to construct a theoretically tractable estimator of $\theta$ we consider a direct approach by estimating the influence function on the right hand side. Emulating expression (\ref{low:d:intuition}), we propose the following projected estimating equation along the direction $\hat\vb$,
$$
\hat S(\bbeta)= \hat\vb^T\tb(\Zb, \bbeta),
$$
where $\hat \vb$ is defined as the solution to the optimization problem:
\begin{align} \label{vdef}
	\hat\vb=\argmin \|\vb\|_1~~\textrm{such that}~~ \Big\|\vb^T \Tb(\Zb,\hat \bbeta) -\eb_1\Big\|_\infty\leq \lambda'.
\end{align}
In (\ref{vdef}) $\lambda'$ is an additional tuning parameter, and $\eb_1$ is a $d$-dimensional {row} vector $(1,0,\ldots,0)$, where the position of $1$ corresponds to that of $\theta$ among $\bbeta$. It is easily seen that $\hat \vb^T$ is a natural estimator of $\vb^{*T} := [E_{\Tb}(\bbeta^*)]^{-1}_{1*}$ in the high dimensional setting, which is an essential term in the right hand of (\ref{low:d:intuition}). Thus, $\hat S(\bbeta)$ can be viewed as an estimate of the influence function for estimating $\theta$ in high dimensions. To better understand our method, consider the linear model example. In this case, we have
$$
\hat S(\bbeta)=n^{-1}\hat\vb^T\Xb^T(\Xb\bbeta - \bY),
$$
where 
\begin{align*}
\hat\vb=\argmin \|\vb\|_1,~~\textrm{such that}~~ \|\vb^T\bSigma_{n} -\eb_1\|_{\infty}\leq \lambda'.
\end{align*}
We can see that $\hat\vb$ corresponds to the first column of the CLIME estimator for the inverse covariance matrix of $\bX_i$. 

We emphasize that the construction of $\hat\vb$ does not depend on knowing which is the estimating equation for $\theta$ and which is the nuisance estimating equation space, and thus the projection is different from the decorrelated score method in \cite{ning2014general}. In fact, the lack of a valid loss (or likelihood) function corresponding to the general estimating equations is the main difficulty for applying the existing likelihood based inference methods. 

%In addition, the proposed method is also different from \cite{van2013asymptotically}. Specifically, our method does not require to invert the Karush-Kuhn-Tucker conditions of the optimization problem (\ref{mainoptimization}), which is used by \cite{van2013asymptotically} to study the Lasso type estimator.  Another difference is that, we use a direct approach to estimate $\vb^*$,  

%The reason we construct the projected estimating equation $\hat S(\bbeta)$ is that the population direction  $\vb^{*T} := [E_{\Tb}(\bbeta^*)]^{-1}_{1*}$ can be well approximated by $\hat \vb$ since $E_{\Tb}(\bbeta)$ is the concentrating value of $\Tb(\Zb, \bbeta)$. By choosing a consistently estimated direction $\hat \vb$, the $(d-1)$-dimensional estimator $\hat{\bgamma}$ (recall $\hat \bbeta = (\hat \theta, \hat \bgamma^T)^T$) of the nuisance parameter is asymptotically ignorable along this direction. The new estimating equation $\hat S(\bbeta)$ is similar to the classical first order local ancillary estimating equation \citep{mcleish1988theory}, but encourages the sparsity structure of the high dimensional vector $\hat\vb$ as indicated by (\ref{low:d:intuition}). 

Recall that $\hat \bbeta = (\hat \theta, \hat \bgamma^T)^T$. By plugging in the estimator $\hat{\bgamma}$, we obtain the projected  estimating equation $\hat S(\theta,\hat{\bgamma})$ for the parameter of interest $\theta$. Similar to the classical estimating equation approach, we propose to estimate $\theta$ by a Z-estimator $\tilde{\theta}$, which is the root of $\hat S(\theta,\hat{\bgamma})=0$. In the following section we lay out the foundations of a unified theory guaranteeing that the estimator $\tilde \theta$ is asymptotically normal.

\subsection{A General Theoretical Framework} \label{masterthm:sec}

In this section we provide generic sufficient conditions which guarantee the existence and asymptotic normality of $\tilde \theta$, which is the root of $\hat S(\theta,\hat{\bgamma})=0$. Due to space limitations, we only present results on the confidence intervals, and the results on uniformly valid confidence intervals are deferred to the Supplementary Material. The results and proofs on hypothesis testing can be obtained from the authors upon request.

We assume that $\tb(\Zb, \bbeta)$ is twice differentiable in $\bbeta$. Recall that we further require $\bbeta^*$ to be the unique solution to $E_{\tb}(\bbeta) = 0$, where $E_{\tb}(\bbeta) = \lim_{n \rightarrow \infty} \EE \tb(\Zb, \bbeta)$. %A typical example of the population equation $E_{\tb}(\bbeta)$, which will be used throughout the paper is $E_{\tb}(\bbeta) = \lim_{n \rightarrow \infty} \EE \tb(\Zb, \bbeta)$. 
For any $\bbeta$ we let $S(\bbeta) := \vb^{*T}\tb(\Zb, \bbeta)$, where $\vb^{*T} := [E_{\Tb}(\bbeta^*)]^{-1}_{1*}$. %and recall that $\hat S(\bbeta) = \hat \vb^T \tb(\Zb, \bbeta)$ with $\hat \vb$ being some estimate of $\vb^*$. 
Let $\PP_{\bbeta}$ be the probability measure under the parameter $\bbeta$. We use the shorthand notation $\PP^* = \PP_{\bbeta^*}$, to indicate the measure under the true parameter $\bbeta^*$. %Throughout Section \ref{masterthm:sec} let $\hat \bbeta$ and $\hat \vb$ be estimates of the parameters $\vb^*$ and $\bbeta^*$, coming from any estimation algorithm as long as assumptions required below are satisfied. 
For any vector $\bbeta = (\theta, \bgamma^T)^T$, we use the following shorthand notation $\bbeta_{\check \theta} = (\check \theta, \bgamma^T)^T$ to indicate that $\theta$ is replaced by $\check \theta$. Before we proceed to define our abstract assumptions and present the results, we first motivate them and give an informal description below.

\subsubsection{Motivation and Informal Description}

Throughout this section we build our theory based on the premisses that the estimators $\hat \bbeta$ and $\hat \vb$ can be shown to be $L_1$ consistent, i.e. $\|\hat \bbeta - \bbeta^*\|_1 = o_p(1)$ and  $\|\hat \vb - \vb^*\|_1 = o_p(1)$. This is expected to hold for estimators solving programs (\ref{mainoptimization}) and (\ref{vdef}) owing to the fact that both programs aim to minimize the $L_1$ norm of the parameters. The $L_1$ consistency (see (\ref{consistencyassumpweakcn})) is central in what follows. Under this presumption, the key idea in our theory is the successful control of the deviations of the ``plug-in'' equation $\hat S(\theta, \hat \bgamma) = \hat S(\hat \bbeta_{\theta})$ about the equation $S(\theta, \bgamma^*) = S(\bbeta^*_{\theta})$ (recall $S(\bbeta^*_{\theta}) := \vb^{*T}\tb(\Zb, \bbeta^*_{\theta})$), i.e. we aim to establish $\hat S(\hat \bbeta_{\theta}) = S(\bbeta^*_{\theta}) + o_p(1)$. By the mean value theorem
\begin{align}\label{mean:val:thm:ass}
\hat S(\hat \bbeta_{\theta}) & = S(\bbeta^*_{\theta}) + \hat \vb^{T} \Tb(\Zb, \tilde \bbeta_\nu) (\hat \bbeta_{\theta} - \bbeta^{*}_{\theta}) + (\hat \vb - \vb^*)^T \tb(\Zb, \bbeta^*_{\theta}),
\end{align}
where $\tilde \bbeta_\nu$ is a point on the line segment joining $\hat \bbeta_{\theta}$ with $\bbeta^*_{\theta}$. Owing to the $L_1$ consistency of $\hat \vb$ and $\hat \bbeta$, (\ref{mean:val:thm:ass}) can indeed be rewritten in the form  $\hat S(\hat \bbeta_{\theta}) = S(\bbeta^*_{\theta}) + o_p(1)$, provided that $\|\hat \vb^{T} \Tb(\Zb, \tilde \bbeta_\nu)\|_{\infty} = O_p(1)$ and $\|\tb(\Zb, \bbeta^*_{\theta})\|_{\infty} = O_p(1)$. A sufficient and also sensible condition for these bounds, is to desire $\|\tb(\Zb, \bbeta^*_{\theta}) - E_{\tb}(\bbeta^*_{\theta})\|_{\infty} = o_p(1)$, and $\|\hat \vb^{T} \Tb(\Zb, \tilde \bbeta_\nu) - \vb^{*T} E_{\Tb} (\bbeta^*_\theta)\|_{\infty} = o_p(1)$, where $E_{\tb}(\bbeta^*_{\theta})$ and $E_{\Tb} (\bbeta^*_\theta)$ are the limiting expected values of $\tb(\Zb, \bbeta^*_{\theta})$ and $\hat \vb^{T} \Tb(\Zb, \tilde \bbeta_\nu)$, respectively. 
It is therefore rational to believe that the latter $L_{\infty}$-norms converge to $0$; see Assumption \ref{noiseassumpCI}. Furthermore, to show $\sqrt{n}$ consistency of the equations one needs to require an additional scaling condition on the latter convergence rates; see (\ref{assumpone}).  
%$$
%\hat S(\theta, \hat \bgamma) = \vb^{*T}E_{\tb}(\Zb, \bbeta^*_\theta) + (\vb^{*T}{\tb}(\Zb, \bbeta^*_\theta) - \vb^{*T}E_{\tb}(\Zb, \bbeta^*_\theta)) + \hat \vb^T \Tb(\Zb, \tilde \bbeta_{\nu}) (\hat \bbeta_{\theta} - \bbeta^*_{\theta}) + (\hat \vb^T - \vb^*) \tb(\Zb, \bbeta^*_{\theta}),
%$$
%We proceed to lay out several assumptions ensuring the consistency and asymptotic normality of the Z-estimator $\tilde \theta$ proposed in the end of Section \ref{LZEHDSsection}. We begin with establishing the consistency of the estimator $\tilde \theta$. We require the following assumption.

\subsubsection{Main Results}
We now formalize our intuition above by requiring the following assumption.

\begin{assumption}[Concentration]\label{noiseassumpCI} There exists a neighborhood $\mathcal{N}_{\theta^*}$ of $\theta^*$, such that for all $\theta \in \mathcal{N}_{\theta^*}$:
\begin{align}
\lim_{n \rightarrow \infty} \PP^*(\| \tb(\Zb, \bbeta_\theta^*) - E_{\tb}(\bbeta_\theta^*) \|_{\infty} \leq r_{1}(n, \theta)) & = 1, \label{betastartassumpCI}\\
\lim_{n \rightarrow \infty} \PP^*(| \vb^{*T} \tb(\Zb, \bbeta_\theta^*) - \vb^{*T} E_{\tb}(\bbeta_\theta^*)| \leq r_{2}(n, \theta)) & = 1, \label{betastartassumpCIvstar}\\
\lim_{n \rightarrow \infty} \PP^*\bigg( \sup_{\nu \in [0,1]} \| \hat \vb^{T} \Tb(\Zb, \tilde \bbeta_\nu) -  \vb^{*T} E_{\Tb} (\bbeta^*_\theta) \|_{\infty} \leq  r_3(n, \theta) \bigg)& = 1, \label{lambdaprimeasumpCI} 
\end{align}
where $\tilde \bbeta_\nu = \nu \hat \bbeta_{\theta} + (1 - \nu)\bbeta_{\theta}^*$, $\sup_{\theta \in \mathcal{N}_{\theta^*}} \max(r_{1}(n, \theta), r_{2}(n, \theta), r_{3}(n, \theta)) = o(1)$, and the following condition holds:
$$
\sup_{\theta \in \mathcal{N}_{\theta^*}}  \|E_{\tb}(\bbeta^*_\theta)\|_{\infty} < \infty
, ~~~~ \sup_{\theta \in \mathcal{N}_{\theta^*}} \| \vb^{*T}\left[E_{\Tb} (\bbeta^*_\theta) \right]_{-1}\|_{\infty} < \infty,
$$
where $[\Ab]_{-1}$ represents a submatrix of $\Ab$ with the $1$\textsuperscript{st} column removed.
\end{assumption}

Condition (\ref{betastartassumpCI}) means that the equation $\tb(\Zb, \bbeta_\theta^*)$ concentrates on its limiting value $E_{\tb}(\bbeta_\theta^*)$ for any $\theta$ in a small neighborhood of $\theta^*$. Similarly, condition (\ref{betastartassumpCIvstar}) implies that the projection of the estimating equation on $\vb^{*}$ also concentrates on its limiting value locally around $\theta^*$, and is automatically implied by (\ref{betastartassumpCI}) when $\|\vb^*\|_1 = O(1)$. Finally, condition (\ref{lambdaprimeasumpCI}) means that the projection of the Jacobian matrix $\Tb(\Zb, \tilde \bbeta_\nu)$ on $\hat \vb$ concentrates on its limiting value $\vb^{*T} E_{\Tb} (\bbeta^*_\theta)$ in a neighborhood of $\theta^*$. These conditions are mild, and can be validated for all examples we consider.

\begin{assumption}[$L_1$ Consistency]\label{ass:consistencyassumpweakcn} Let the estimators $\hat \bbeta$ and $\hat \vb$ satisfy
\begin{align}
\lim_{n \rightarrow \infty} \PP^*(\|\hat \bbeta - \bbeta^*\|_1  \leq r_{4}(n)) = 1, ~~~~ \lim_{n \rightarrow \infty} \PP^*(\|\hat \vb - \vb^*\|_1 \leq r_{5}(n)) = 1, \label{consistencyassumpweakcn}
\end{align}
where $\max(r_4(n), r_5(n)) = o(1)$. 
\end{assumption}

As mentioned previously, (\ref{consistencyassumpweakcn}) is expected to hold due to the formulations of  (\ref{mainoptimization}) and (\ref{vdef}). In particular, (\ref{consistencyassumpweakcn})  has been verified for all examples we consider. Assumptions \ref{noiseassumpCI} and \ref{ass:consistencyassumpweakcn} suffice to show the following consistency result.

\begin{theorem}[Consistency] \label{consistency:sol} Let the (stochastic) map $\theta \mapsto \hat S(\hat \bbeta_\theta)$ be continuous with a single root $\tilde \theta$ or non-decreasing. % and $\hat \vb$ and $\hat \bbeta$ are consistent in the sense that:
%\begin{align}
%\lim_{n \rightarrow \infty} \PP^*(\|\hat \bbeta - \bbeta^*\|_1  \leq r_{4}(n)) = 1, ~~~~ \lim_{n \rightarrow \infty} \PP^*(\|\hat \vb - \vb^*\|_1 \leq r_{5}(n)) = 1, \label{consistencyassumpweakcn}
%\end{align}
%where $\max(r_4(n), r_5(n)) = o(1)$. 
Furthermore, suppose that for any $\epsilon > 0$,  
\begin{align}\label{opposing:signs}
\vb^{*T} [E_{\tb}( \bbeta^*_{\theta^* - \epsilon})]  \vb^{*T} [E_{\tb}( \bbeta^*_{\theta^* + \epsilon})] < 0.
\end{align}
 Under Assumptions \ref{noiseassumpCI} and \ref{ass:consistencyassumpweakcn} we have that
$$
\lim_{n \rightarrow \infty}\PP^*(|\tilde \theta - \theta^*| > \epsilon) = 0.
$$
\end{theorem}
Condition (\ref{opposing:signs}) implies that the scalars $\vb^{*T} [E_{\tb}( \bbeta^*_{\theta^* - \epsilon})] $ and $\vb^{*T} [E_{\tb}( \bbeta^*_{\theta^* + \epsilon})]$ have opposite signs for all $\epsilon > 0$, which in turn guarantees that $\theta^*$ is a unique root of the map $\vb^{*T} [E_{\tb}( \bbeta^*_{\theta})]$. Hence under (\ref{opposing:signs}) the population equation $\hat S (\hat \bbeta_{\theta})$ is unbiased. The condition (\ref{opposing:signs}) holds for numerous examples and is also commonly used in the classical asymptotic theory; see Section 5 of \cite{van2000asymptotic}. In fact, the conclusion of Theorem \ref{consistency:sol} remains valid if one solves the equation approximately in the sense that $\hat S(\hat \bbeta_{\tilde \theta}) = o_p(1)$.
%If $\theta \mapsto \hat S(\hat \bbeta_{\theta})$ is not monotonic, one can instead impose the following uniform convergence and identifiability conditions:
%\begin{align}\label{equniform}
%\sup_{\theta \in \Theta} |\hat S(\hat \bbeta_{\theta}) - \vb^{*T} E_{\tb}(\bbeta_\theta^*)| & \rightarrow_p 0,~~
%\inf_{|\theta - \theta^*| > \epsilon} |\vb^{*T} E_{\tb}(\bbeta_\theta^*)|  > 0,
%\end{align}
%where $\Theta$ is the parameter space for $\theta$. Under (\ref{equniform}), the proof of consistency of $\tilde\theta$ follows from Theorem 5.9 of \cite{van2000asymptotic}. Generally, we need the following two steps to verify the uniform convergence in (\ref{equniform}). First, we apply empirical process theory to establish $\sup_{\theta \in \Theta} |S(\bbeta^*_{\theta}) - \vb^{*T} E_{\tb}(\bbeta_\theta^*)|=o_p(1)$. Second, using similar concentration properties to the ones in Assumption \ref{noiseassumpCI} we can establish $\sup_{\theta \in \Theta}|\hat S(\hat \bbeta_{\theta})-S(\bbeta^*_{\theta})|=o_p(1)$. This provides the consistency of $\tilde{\theta}$. 
To establish the asymptotic normality of $\tilde \theta$, we require the following assumptions.
 %We refer to \cite{van2000asymptotic} for details.
%rewrite the formation of this propo. comment some empirical process approach. more general formulation in van der vaart book. For instance, we can approximately solve $\hat S(\hat \bbeta_{\tilde \theta}) = O_p(n^{-1/2})$} 

%The proof is provided in Appendix \ref{ext:main:proofsofgeneraltheory}.
%\begin{assumption}[Pointwise LLN; Noise Stability]\label{PLLcond} There exists  an open neighborhood of $\theta^*$ --- $\mathcal{N}$, such that:
%\begin{align*}
%& n^{-1/2} S(\bbeta^*_{\theta}) \rightsquigarrow \EE[\vb^{*T} \hb(\bZ, \bbeta^*_{\theta})],\\
%& \lim_{n \rightarrow \infty} \PP^*\bigg( \sup_{\theta \in \mathcal{N}} \bigg\| \frac{1}{n} \sum_{i = 1}^n \hat \vb^{T} \left[\Hb(\bZ_i, \hat \bbeta_{\theta}) \right]_{-1} \bigg\|_{\infty} \leq  r_5(n)  \bigg)  = 1, \\
%& 
%\end{align*}
%for any $\theta \in \mathcal{N}$.
%\end{assumption}
%
%Under Assumption \ref{PLLcond} the following Proposition is easily verified:
%
%\begin{proposition} \label{weak:LLN:prop} Suppose that conditions \ref{consistencyassumpweakcn} and \ref{PLLcond} hold. Then for any $\theta \in \mathcal{N}$:
%$$
%n^{-1/2} \hat S(\hat \bbeta_{\theta}) \rightsquigarrow \EE[\vb^{*T} \hb(\bZ, \bbeta^*_{\theta})].
%$$
%\end{proposition} 
%The proof of this proposition is very similar to that of Theorem \ref{mastertheorem} and is omitted.%We include a proof in Appendix \ref{proofsofgeneraltheory} for completeness.

\begin{assumption}[CLT]\label{CLTcond} Assume that for $\sigma^{2} =  \vb^{*T} \bSigma \vb^*$, it holds that
\begin{align*}
\sigma^{-1} n^{1/2} S( \bbeta^*) \rightsquigarrow N(0,1),
\end{align*}
where $\bSigma = \lim_{n\rightarrow \infty} n \Cov \tb(\Zb, \bbeta^*)$, and assume that $\sigma^2 \geq C > 0$ for some constant $C$.
\end{assumption}

%\begin{assumption}\label{stabilityassump}  There exists a constant $\gamma > 0$ such that:
%%\begin{align}
%%\sup_{\max(\|\vb - \vb^*\|_1, \|\bbeta - \bbeta^*\|_1) < \gamma} \bigg| \vb^{T} \frac{\partial}{\partial \theta} \left[\Tb(X, \bbeta)\right]_{*1}\bigg| \leq \psi(X), \label{stabtwo}
%%\end{align}
%\begin{align}
%\sup_{\max(\|\vb - \vb^*\|_1, \|\bbeta - \bbeta^*\|_1) < \gamma} \bigg| \vb^{T} \left[\Tb(X, \bbeta)\right]_{*1} - \vb^{*T} \left[\Tb(X, \bbeta^*)\right]_{*1}\bigg| \leq \psi(X), \label{stabtwo}
%\end{align}
%where $\psi : \RR^d \mapsto \RR$ is an integrable function such that $\EE^*\psi(\Zb) < \infty$.
%\end{assumption}

%\begin{remark} 
Assumption \ref{CLTcond} ensures that the right hand side of expansion (\ref{mean:val:thm:ass}) converges to a normal distribution when scaled appropriately. This CLT condition is mild, and can typically be verified via Lyapunov or Lindeberg conditions in cases when $\tb(\Zb, \bbeta)$ can be decomposed into a sum of i.i.d. observations. For some types of dependent data, Assumption \ref{CLTcond} holds by applying the martingale central limit theorem (e.g. autoregressive models). Thus, one of the advantages of our framework is that we can handle dependent data, which are not covered by the existing methods. We show such an example in Section \ref{SVAsec}. 

\begin{assumption}[Bounded Jacobian Derivative]\label{bounded:Jac:ass} Suppose there exists a constant $\gamma > 0$ such that:
\begin{align}
\bigg|  \vb^T \frac{\partial}{\partial \theta} \left[\Tb(\Zb, (\theta, \bgamma^T)^T)\right]_{*1}\bigg| \leq \psi(\Zb) , \label{stabtwo}
\end{align}
for any $\vb$ and $\bbeta$ satisfying $\|\vb - \vb^*\|_1<\gamma$ and $\|\bbeta - \bbeta^*\|_1< \gamma$,  where $\psi : \RR^{n\times q} \mapsto \RR$ is an integrable function with $\EE^*\psi(\Zb) < \infty$.
\end{assumption}
Inequality (\ref{stabtwo}) is a technical condition ensuring that $\vb^T\frac{\partial}{\partial \theta} \left[\Tb(\Zb, \bbeta)\right]_{*1}$  is bounded by an integrable function in a small neighborhood and hence does not behave too erratically, so that the dominated convergence theorem can be applied. This is a standard condition,  which is also assumed in Theorem 5.41 of \cite{van2000asymptotic} to establish the asymptotic normality of Z-estimator in the classical low dimensional regime. It is easily seen that this condition is mild and holds for linear estimating equations. 

\begin{assumption}[Scaling]\label{clt:scaling:ass} Assume the convergence rates  in Assumptions \ref{noiseassumpCI} and \ref{ass:consistencyassumpweakcn} satisfy
\begin{align}
n^{1/2}(r_{4}(n)r_{3}(n, \theta^*) + r_{5}(n)r_{1}(n, \theta^*) ) & = o(1). \label{assumpone}
\end{align}
\end{assumption}
Assumption (\ref{assumpone}) ensures that the small terms on the right hand side of (\ref{mean:val:thm:ass}) continue to be negligible, even after normalizing by $\sqrt{n}$. This assumption requires that the consistency and concentration rates in (\ref{consistencyassumpweakcn}) and Assumption \ref{consistency:sol} are sufficiently sharp if one hopes to establish the asymptotic normality of $\tilde\theta$. This condition can be verified in all of our examples. We are now in a position to state the main result of this section. 

\begin{theorem}[Asymptotic Normality] \label{weakconvver2:CI} Assume the conditions from Theorem \ref{consistency:sol} and Assumptions \ref{CLTcond}, \ref{bounded:Jac:ass} and \ref{clt:scaling:ass} hold. %In addition assume:
%\begin{align}
%n^{1/2}(r_{4}(n)r_{3}(n, \theta^*) + r_{5}(n)r_{1}(n, \theta^*) ) & = o(1). \label{assumpone}
%\end{align}
If $\hat \sigma^2$ is a consistent estimator of $\sigma^2$, then for any $t\in\RR$, we have:%$\hat U_n = \frac{n^{1/2}}{\hat \sigma} \hat S(\hat \bbeta_{\tilde \theta})$:
$$
\lim_{n \rightarrow \infty} |\PP^*(\hat U_n \leq t) - \Phi(t)| = 0, ~~\textrm{where}~~\hat U_n = \frac{n^{1/2}}{\hat \sigma} (\tilde \theta - \theta^*). 
$$
\end{theorem}

Some generic sufficient conditions for the consistency of  $\hat\sigma$ are shown in Proposition \ref{arxivsupp:weakconv} in Section \ref{arxivsupp:main:var:consistency} of the Supplementary Material. In our examples, we will develop consistent estimates of the variance $\sigma^2$ case by case. Given a consistent estimator $\hat \sigma^2$, Theorem \ref{weakconvver2:CI} implies that we can construct a $(1-\alpha)$\% confidence interval of $\theta^*$ in the following way%a similar spirit to (\ref{eqvarconfid}).
\begin{equation}\label{eqvarconfid}
\lim_{n\rightarrow\infty}\PP^*\Big(\theta^*\in\big[\tilde \theta - \Phi^{-1}(1 - \alpha/2) \hat \sigma/\sqrt{n}, \tilde \theta + \Phi^{-1}(1 - \alpha/2) \hat \sigma/\sqrt{n}\big]\Big)=1-\alpha.
\end{equation}

We conclude this section by noting a property of our estimator $\tilde \theta$ in cases when the estimating equation comes from a log-likelihood, i.e. $\tb(\Zb, \bbeta) = n^{-1}\sum_{i = 1}^n \hb(\bZ_i, \bbeta)$ with $\hb(\bZ_i, \bbeta)$ being the gradient of the log-likelihood for $\bZ_i$. Denote $\Hb(\bZ, \bbeta) = \frac{\partial}{\partial \bbeta} \hb(\bZ,\bbeta)$. According to the information identity $- \EE \Hb(\bZ, \bbeta^*) = \Cov \hb(\bZ, \bbeta^*)$, we have $\vb^{*T} \bSigma \vb^* = (\bSigma^{-1})_{11}$. In this case, the Z-estimator $\tilde \theta$ is efficient \citep{van2000asymptotic}, because the variance $ (\bSigma^{-1})_{11}$ coincides with the inverse of the information bound for $\theta$.

\section{Implications of the General Theoretical Framework}\label{secexample}

In this section, we apply the general theory of Section \ref{masterthm:sec} to the motivating examples we listed in the introduction.

\subsection{Linear model and Instrumental Variables Regression} \label{dantzigselectorsection}
%In this section we consider an application of the general theory to the linear model with conditional moment specifications. Assume that the linear model is specified by the following moment condition $\EE(Y\mid\bX)= \bX^T\bbeta^*$.
%Decomposing $\bbeta := (\theta, \bgamma)$, our goal is to construct confidence intervals for the parameter $\theta$. Let $\Xb \in \RR^{n \times d}$ be the design matrix stacking the i.i.d. covariates $\{\bX_i\}_{i = 1}^n$ and $\bY\in \RR^n$ be the response vector with i.i.d. entries $Y_i$. In terms of the notation of Section \ref{masterthm:sec}, we can identify $\tb((\bY,\Xb), \bbeta) = n^{-1}\Xb^T(\Xb \bbeta - \bY)$ and $E_{\tb}(\bbeta) = \EE \tb((\bY,\Xb), \bbeta)$. In addition, $E_{\tb}(\bbeta)$ has $\bbeta^*$ as its unique root, provided that the second moment matrix $\bSigma_{\bX} := n^{-1} \EE \Xb^T \Xb$ is positive definite. Denote by $\bSigma_n = n^{-1} \Xb^T \Xb$ the empirical estimator of $\bSigma_{\bX}$. Given the moment condition,  \cite{candes2007dantzig} proposed the following Dantzig selector for  $\bbeta$, 
%$$
%\hat \bbeta = \argmin \|\bbeta\|_1, \mbox{ such that } \|n^{-1}\Xb^T(\Xb\bbeta - \bY)\|_{\infty} \leq \lambda.
%$$
In this section we consider the linear model via Dantzig Selector and the instrumental variables regression. As seen in the introduction, the instrumental variables regression can be viewed as a generalization of the linear regression, by substituting $\bW \equiv \bX$. For simplicity, we only present the results for the linear regression and defer the development of the inference theory for instrumental variables regression to Appendix \ref{dantzigselector:IVR} of Supplementary Material.

Recall that $\bbeta := (\theta, \bgamma)$, and let $\bSigma_n = n^{-1} \Xb^T \Xb$ be the empirical estimator of $\bSigma_{\bX}$. Our goal is to construct confidence intervals for the parameter $\theta$. In the linear regression case, we can easily show that $\hat S(\bbeta)$ reduces to
$$
\hat S(\bbeta)=n^{-1}\hat\vb^T\Xb^T(\Xb\bbeta - \bY),
$$
where 
\begin{align}
\hat\vb=\argmin \|\vb\|_1,~~\textrm{subject to}~~ \|\vb^T\bSigma_{n} -\eb_1\|_{\infty}\leq \lambda', \label{lambda:prime:dantz}
\end{align}
is an estimator of $\vb^* = \bSigma_{\bX}^{-1} \eb_1^T$. We impose the following assumption.

\begin{assumption}\label{subGausserr:design:bounded:mom:matr} Assume that the error $\varepsilon:=Y-\bX^T\bbeta^*$ and the predictor $\bX$ are both coordinate-wise sub-Gaussian, i.e.,
$$\|\varepsilon\|_{\psi_2} := K < \infty,  ~~~ \sup_{j \in \{1,\ldots,d\}} \|X_j\|_{\psi_2} := K_{\bX} < \infty,$$ 
for some fixed constants $K, K_{\bX} > 0$. Furthermore assume that the variance $\Var(\varepsilon) \geq C_{\varepsilon} > 0$, the random variables $\varepsilon$ and $\bX$ are independent, and the second moment matrix $\bSigma_{\bX}$ satisfies $\lambda_{\min}(\bSigma_{\bX}) \geq \delta > 0,$ where $\delta$ is some fixed constant.
\end{assumption}

To construct confidence intervals for $\theta$, we consider $\hat U_n = \hat \Delta^{-1} n^{1/2}(\tilde{\theta} - \theta^*)$, where $\tilde \theta$ is defined as the solution to $\hat S(\theta, \hat \bgamma) = 0$, and  
\begin{align}\label{delta:hat:dantzig}
\hat \Delta := \hat \vb^T \bSigma_n \hat \vb n^{-1} \sum_{i = 1}^n (Y_i - \bX_i^T \hat \bbeta)^2,
\end{align}
is an estimator of the asymptotic variance $\Delta := \vb^{*T}\bSigma_{\bX}\vb^* \Var(\varepsilon)$. 
In high dimensional models, it is often reasonable to assume that the vectors $\vb^*$ and $\bbeta^*$ are sparse. Let $s$ and $s_{\vb}$ denote the sparsity of $\bbeta^*$ and $\vb^*$ correspondingly, i.e., $\|\bbeta^*\|_0 = s$ and $\|\vb^*\|_0 = s_{\vb}$.
The next corollary of the general Theorem \ref{weakconvver2:CI} shows the asymptotic normality of $\hat U_n$ in linear models. To simplify the presentation of our result we will assume that $\|\vb^*\|_1$ is bounded, although this is not needed in our proofs.

\begin{corollary} \label{normdantzigfinal} 

Assume that Condition \ref{subGausserr:design:bounded:mom:matr} holds, and 
\begin{align*}
\max(s_{\vb}, s)\log d/\sqrt{n} = o(1), ~~~ \sqrt{\log d/n} = o(1).
\end{align*}
Then with $\lambda \asymp \sqrt{\log d/n}$ and $\lambda' \asymp \sqrt{\log d/n}$, $\hat U_n$ satisfies for any $t\in\RR$:
$$
\lim_{n \rightarrow \infty} |\PP^*(\hat U_n \leq t) - \Phi(t))| = 0.
$$
\end{corollary}
The proof of Corollary \ref{normdantzigfinal} can be found in Appendix \ref{arxivsupp:main:DantzigProofs} of the Supplementary Material.
%\begin{remark} 
%We note that the tuning parameter $\lambda'$ in (\ref{lambda:prime:dantz}) is of order $\lambda' \asymp \|\vb^*\|_1\sqrt{\log d/n}$, which depends on $\|\vb^*\|_1$. This is due to the technical proof on controlling the convergence rate of $\|\vb^{*T}\bSigma_{n} -\eb_1\|_{\infty}$. The same phenomenon is also observed in the graphical model estimation problem considered by  \cite{cai2011constrained}. If $\|\vb^*\|_1$ is bounded, then the condition in Corollary \ref{normdantzigfinal} reduces to $\max(s_{\vb}, s)\log d/\sqrt{n} = o(1)$, which agrees with the existing conditions in  \cite{zhang2014confidence,van2013asymptotically}. In fact, under the additional assumption $s_{\vb}^3/n = o(1)$, we can show that  $\hat U_n$ is uniformly asymptotically normal; see Remark \ref{arxivsupp:main:unif:CI:danztig} of the Supplementary Material.
The conditions in Corollary \ref{normdantzigfinal} agree with the existing conditions in  \cite{zhang2014confidence,van2013asymptotically}. In fact, under the additional assumption $s_{\vb}^3/n = o(1)$, we can show that  $\hat U_n$ is uniformly asymptotically normal; see Remark \ref{arxivsupp:main:unif:CI:danztig} of the Supplementary Material. Finally, we comment that a similar asymptotic normality result under the instrumental variables regression is shown in Corollary \ref{normdantzigfinal:IVR} of the Supplementary Material.

\subsection{Graphical Models} \label{edgetestingsec}

We begin with introducing several assumptions which we need throughout the development. First let $\bSigma_{\bX}$ satisfy $\lambda_{\min}(\bSigma_{\bX}) \geq \delta > 0$, where $\delta$ is some fixed constant. Similarly to Section \ref{dantzigselectorsection} we assume that  $\bX$ is coordinate-wise sub-Gaussian, i.e.,
\begin{align}K_{\bX} := \max_{j \in \{1,\ldots,d\}} \|X_j\|_{\psi_2} < \infty,\label{subGausserr:CLIME}\end{align}
for some fixed constant $K_{\bX} > 0$. Our goal is to construct confidence intervals for a component of $\bOmega^*$, where $\bOmega^*=(\bSigma_{\bX})^{-1}$. Without loss of generality, we focus on the parameter $\Omega^*_{1m}$ for some $m \in \{1,\ldots, d\}$. When $\bX$ are coming from a Gaussian distribution, the confidence intervals for $\Omega^*_{1m}$ provide uncertainty assessment on whether $X_1$ is independent of $X_m$ given the rest of the variables. 
%In this section we consider applications of our general theory to graphical models with moment specifications. Assume that $\bX_1,...,\bX_n$ are i.i.d. copies of $\bX\in\RR^d$ with $\EE(\bX)=0$ and $\Cov(\bX)=\bSigma_{\bX}$. Let $\bSigma_{\bX}$ satisfy $\lambda_{\min}(\bSigma_{\bX}) \geq \delta > 0$, where $\delta$ is some fixed constant. Similarly to Section \ref{dantzigselectorsection} we assume that  $\bX$ is coordinate-wise sub-Gaussian, i.e.,
%\begin{align}\sup_{j \in \{1,\ldots,d\}} \|X_j\|_{\psi_2} := K_{\bX} < \infty,\label{subGausserr:CLIME}\end{align}
%for some fixed constant $K_{\bX} > 0$. Our goal is to construct confidence intervals for a component of $\bOmega^*$, where $\bOmega^*=(\bSigma_{\bX})^{-1}$. Without loss of generality, we focus on the parameter $\Omega^*_{1m}$ for some $m \in \{1,\ldots, d\}$. When $\bX$ are coming form a Gaussian distribution, the confidence intervals for $\Omega^*_{1m}$ provide uncertainty assessment on whether $X_1$ is independent of $X_m$ given the rest of the variables. 

There are a number of recent works considering the inferential problems for  Gaussian graphical models \citep{jankova2014confidence,chen2015asymptotically,ren2015asymptotic,liu2013} and Gaussian copula graphical models \citep{gu2015local,barber2015rocket}. Our framework differs from these existing procedures in the following two aspects. First, our method is based on the estimating equations rather than the likelihood and (node-wise) pseudo-likelihood. Second, we only require each component of $\bX$ is sub-Gaussian, whereas the majority of the existing methods require the data to be sampled from Gaussian or Gaussian copula distributions. 

%Let $\bSigma_n = n^{-1} \sum_{i = 1}\bX_i^{\otimes 2}$ be the sample covariance of $\bX_1,...,\bX_n$. Based on the second moment condition $\bSigma_{\bX}\bOmega^*=\Ib_d$, \cite{cai2011constrained} proposed the CLIME estimator of $\bOmega^*$:
%\begin{align} \label{generalCLIMEopt}
%\hat{\bOmega}=\argmin\|\bOmega\|_1,~~\textrm{s.t.}~~ \|\bSigma_n\bOmega-\Ib_d\|_{\max}\leq\lambda.
%\end{align}
Let $\bbeta^*:=\bOmega^*_{* m}$, be the $m$\textsuperscript{th} column of $\bOmega^*$. Then, the CLIME estimator of $\bbeta^*$ given by (\ref{generalCLIMEopt}) reduces to 
$$
\hat{\bbeta}=\argmin\|\bbeta\|_1,~~\textrm{subject to}~~ \|\bSigma_n\bbeta-\eb_m^T\|_{\infty}\leq\lambda,
$$
where $\eb_m^T$ is a unit column vector with $1$ in the $m$\textsuperscript{th} position and $0$ otherwise. Phrasing this problem in the terminology of Section \ref{masterthm:sec}, we can construct  $d$ estimating equations:  $\tb(\Xb, \bbeta) = \bSigma_n \bbeta - \eb_m^T$. Let us decompose the vector $\bbeta$ as $\bbeta := (\theta, \bgamma^T)^T$. Then the projected estimating equation for $\theta$ is given by
$$
\hat S(\bbeta)=\hat{\vb}^T(\bSigma_n\bbeta-\eb_m^T),
$$
where
\begin{align}\label{lambda:prime:clime}
\hat{\vb}=\argmin\|\vb\|_1,~~\textrm{such that}~~ \|\vb^T\bSigma_n-\eb_1\|_{\infty}\leq\lambda'.
\end{align}
Here $\hat \vb$ is an estimate of $\vb^* := (\bSigma_{\bX})^{-1}_{*1} = \bOmega^*_{*1}$. Notice that, due to the symmetry of $\hat{\bbeta}$ and $\hat{\vb}$, if we take $\lambda=\lambda'$, it suffices to simply solve the CLIME optimization (\ref{generalCLIMEopt}) once in order to evaluate $\hat S(\hat\bbeta)$, as $\hat \bbeta = \hat \bOmega_{*m}$ and $\hat \vb = \hat \bOmega_{*1}$. This pleasant consequence for CLIME shows that in this special case the number of tuning parameters in the generic procedure described in Section \ref{LZEHDSsection} can be reduced to $1$, and hence the computation is simplified.
%it suffices to fit CLIME procedure only once in order to conduct inference.
%\end{remark}

The solution $\tilde \theta$ to the equation $\hat S(\theta, \hat \bgamma) = 0$ has the following closed form expression:
\begin{align} \label{onestepCLIME}
\tilde \theta = \hat \theta - \frac{\hat \vb^T (\bSigma_n\hat \bbeta - \eb_m^T)}{\hat \vb^T \bSigma_{n,*1}}.
\end{align}
To establish the asymptotic normality of $\tilde \theta$, we impose the following assumption.
%\begin{assumption} 
%\end{assumption}
\begin{assumption} \label{varXXTassump} There exists a constant $\alpha_{\min} > 0$ such that:
 \begin{align*}
\Delta  \geq \alpha_{\min}  \|\bbeta^*\|_2^{2}\|{\vb}^{*}\|_2^{2}, ~~\textrm{where}~~\Delta = \Var(\vb^{*T} \bX^{\otimes 2} \bbeta^{*}). 
 \end{align*}
\end{assumption}
%\begin{remark} \label{isserlisthms} 
We note that Assumption \ref{varXXTassump} is natural. For example when $\bX \sim N(0,\bSigma_{\bX})$, Isserlis' theorem yields that for any two vectors $\bxi$ and $\btheta$,
\begin{align*}
\Var(\bxi^T \bX^{\otimes 2}\btheta) &=  (\bxi^T\bSigma_{\bX}\bxi)(\btheta^T\bSigma_{\bX}\btheta) + (\bxi^T \bSigma_{\bX} \btheta)^2 \geq  \lambda^2_{\min}(\bSigma_{\bX}) \|\bxi\|_2^2 \|\btheta\|_2^2,
\end{align*}
which clearly implies Assumption \ref{varXXTassump}, if $\lambda^2_{\min}(\bSigma_{\bX})$ is lower bounded by a constant. 
%\end{remark}

Denote $\|\bbeta^*\|_0 = s$ and $\|\vb^*\|_0 = s_{\vb}$. To simplify the presentation of our result we will assume that $\|\vb^*\|_1$ and $\|\bbeta^*\|_1$ are bounded quantities, although this is not needed in our proofs. The following corollary yields the asymptotic normality of $\hat U_n = \hat \Delta^{-1/2} n^{-1/2}(\tilde \theta - \theta^*)$, where $\hat \Delta := n^{-1}\sum_{i = 1}^n (\hat \vb^T (\bX_i^{\otimes 2} - \bSigma_n)\hat \bbeta)^2$ is an estimator of $\Delta$. 
\begin{corollary} \label{IFCLIME} Let Assumption \ref{varXXTassump} and (\ref{subGausserr:CLIME}) hold. Furthermore, assume that
\begin{align}\label{CLIMEratecond}
\max(s^2_{\vb}, s^2) & \log d\log(nd)/n = o(1), ~~~~ \exists ~ k > 2: (s_{\vb}s)^{k}/n^{k - 1} = o(1),
\end{align}
and $\Var((\vb^{*T}\bX^{\otimes 2}\bbeta^{*})^2) = o({n})$, $\EE( \vb^{*T} \bX^{\otimes 2} \bbeta^* )^2 = O(1)$. Let the tuning parameters be $\lambda \asymp \sqrt{\log d/n}$ and $\lambda' \asymp \sqrt{\log d/n}$. Then for all $t\in\RR$: %If a consistent estimate of $\Delta$ --- $\hat \Delta$ is available then for all $t$:
$$
\lim_{n \rightarrow \infty} |\PP^*(\hat U_n \leq t) - \Phi(t)| = 0.
$$
\end{corollary}
The proof of Corollary \ref{IFCLIME} can be found in Appendix \ref{arxivsupp:main:edgetestproofsSEC} of the Supplementary Material. In addition, we provide a stronger result on uniform confidence intervals for $\theta$ in Corollary \ref{arxivsupp:main:unifconvCLIME} of the Supplementary Material. %Once again, similarly to \cite{cai2011constrained}, the tuning parameters are of orders $\lambda \asymp \|\bbeta^*\|_1\sqrt{\log d/n}$ and $ \lambda' \asymp \|\vb^*\|_1\sqrt{\log d/n}$. If $\|\bbeta^*\|_1$ and $\|\vb^*\|_1$ are bounded, then the first part of condition (\ref{CLIMEratecond}) reduces to $\max(s_{\vb}, s)\log d\log(nd)/n = o(1)$, which agrees with \cite{ren2015asymptotic, liu2013}. In addition, when the data are known to be Gaussian one could use the alternative estimator $\tilde{\Delta}:=\hat v_1\hat \beta_m + \hat v_m\hat \beta_1$ of $\Delta$, which can also be shown to be consistent under the assumption $\max(s_{\vb}, s)\sqrt{\log d/n} = o(1)$. The second part of condition (\ref{CLIMEratecond}) is mild, since it is only slightly stronger than $n^{-1}s_{\vb} s = o(1)$. Unlike \cite{jankova2014confidence}, we do not assume irrepresentable conditions. 
Once again, the first part of condition (\ref{CLIMEratecond}) agrees with \cite{ren2015asymptotic, liu2013}. In addition, when the data are known to be Gaussian one could use the alternative estimator $\tilde{\Delta}:=\hat v_1\hat \beta_m + \hat v_m\hat \beta_1$ of $\Delta$, which can also be shown to be consistent under the assumption $\max(s_{\vb}, s)\sqrt{\log d/n} = o(1)$. The second part of condition (\ref{CLIMEratecond}) is mild, since it is only slightly stronger than $n^{-1}s_{\vb} s = o(1)$. Unlike \cite{jankova2014confidence}, we do not assume irrepresentable conditions. 

\begin{remark}(Transelliptical Graphical Models) \label{trans:ell:graph:model}
Our estimating equation based methods for constructing confidence intervals can be extended to transelliptical graphical models \citep{liu2012transelliptical}. The key idea is to replace the same covariance matrix $\bSigma_n$ in (\ref{generalCLIMEopt}) and (\ref{lambda:prime:clime})  by 
$$
\hat S^{\tau}_{jk} = \begin{cases}\sin\left(\frac{\pi}{2} \hat \tau_{jk}\right), &j \neq k; \\ 1, &j = k, \end{cases}
$$
where
$$
\hat \tau_{jk} = \frac{2}{n(n-1)} \sum_{1 \leq i < i' \leq n} \sign \left((X_{ij} - X_{i'j})(X_{ik} - X_{i'k})\right).
$$
Similar to Corollary \ref{IFCLIME}, the asymptotic normality of the estimator $\tilde{\theta}$ is established. The details are shown in Appendix \ref{arxivsupp:main:SKEPTICCLIME} of the Supplementary Material.
\end{remark}

\subsection{Sparse Linear Discriminant Analysis} \label{sparseLDA}

In this section, we consider an application of the general theory to the sparse linear discriminant analysis problem.
The consistency and rates of convergence of the classification rule $\hat \psi(\bO)$ (\ref{sparseLDAeq}) have been established by \cite{cai2011direct} in the high dimensional setting. In the following, we apply the theory of Section \ref{masterthm:sec} to construct confidence intervals for $\theta$, where $\theta$ is the first component of $\bbeta$, i.e., $\bbeta = (\theta, \bgamma)$. Note that if $\theta=0$, then it implies that the first feature of $\bO$ is not needed in the Bayes rule $\psi(\bO)$. Hence, our procedure can be used to assess whether a certain feature is significant in the classification. 

By the identity $\bbeta^* = \bOmega \bdelta$, we can construct the  $d$-dimensional estimating equations  $\tb((\Xb, \Yb), \bbeta) = \hat\bSigma_n \bbeta - (\bar \bX - \bar \bY )$. Then the projected estimating equation for $\theta$ is given by
$$
\hat S(\bbeta)=\hat{\vb}^T(\hat \bSigma_n\bbeta - (\bar \bX - \bar \bY)),
$$
where
$$
\hat{\vb}=\argmin\|\vb\|_1,~~\textrm{such that}~~ \|\vb^T\hat \bSigma_n-\eb_1\|_{\infty}\leq\lambda',
$$
is an estimator of $\vb^{*} = (\bSigma^{-1})_{*1}$. Solving the equation $\hat S(\theta, \hat \bgamma) = 0$ gives us the Z-estimator $\tilde \theta$.
%$$
%\tilde \theta = \hat \theta - \frac{\hat \vb^T(\hat \bSigma_n \hat \bbeta - (\bar \bX - \bar \bY))}{\hat \vb^T(\bSigma_n)_{*1}}.
%$$
To establish the asymptotic normality of $\tilde \theta$, we impose the following assumption.
%Armed with this notation, we proceed to formulate an asymptotic normality result, allowi to construct confidence intervals for $\theta$ in sparse LDA.

%We can see that due to the special structure of the Sparse LDA estimator, the estimating equation is not sum of i.i.d. terms {\color{red}why?}. However, the theory which we developed in Section \ref{masterthm:sec} allows us to handle such cases. To this end we formulate an additional distributional assumption on $\bU$:
\begin{assumption} \label{var:bound:LDA} Assume that $\bU$ satisfies the following moment assumption:
\begin{align*}
\Var(\vb^{*T} \bU^{\otimes 2} \bbeta^*) \geq V_{\min} \|\vb^{*}\|_2^2 \|\bbeta^*\|_2^2,
\end{align*}
where $V_{\min}$ is a positive constant. In addition let $K_{\bU} = \max_{j \in \{1,\ldots, d\}} \|U_j\|_{\psi_2} < \infty$.
\end{assumption}
%\begin{remark} It can be seen (see Remark \ref{ext:rem:LDP} in the Supplementary Material) that a sufficient condition for (\ref{varboundbelow}) to hold, is Assumption \ref{varXXTassump} to hold for $\bU$. In other words, it suffices that there exists a $V'_{\min}$ such that: {\color{red} can we use the following assumption instead of 6.2?}
%$$
%\Var(\vb^{*T} \bU^{\otimes 2} \bbeta^*) \geq V'_{\min} \|\vb^{*}\|_2^2 \|\bbeta^*\|_2^2.
%$$
%As we saw in Remark \ref{isserlisthms}, this implies that (\ref{varboundbelow}) holds for a multivariate normal distribution. 
%\end{remark}
As seen in the comments on Assumption \ref{varXXTassump}, we can similarly show that Assumption \ref{var:bound:LDA} holds if $\bU\sim N(0,\bSigma)$ and $\lambda^2_{\min}(\bSigma)$ is lower bounded by a positive constant.
We define $V_1 := \Var(\vb^{*T}\bU^{\otimes 2}\bbeta^* + \alpha^{-1} \vb^{*T}\bU)$, $V_2 := \Var(\vb^{*T}\bU^{\otimes 2}\bbeta^* - (1 - \alpha)^{-1} \vb^{*T}\bU)$, where $\frac{n_1}{n}=\alpha+o(1)$ for some $0 < \alpha < 1$. Denote 
\begin{align}
\Delta := \alpha V_1 + (1-\alpha)V_2, \label{deltasparseLDAdef}
\end{align} 
and $\hat U_n := \hat \Delta^{-1/2}n^{1/2} (\tilde \theta - \theta^*)$, where $\hat \Delta$ is some consistent estimator of $\Delta$. The explicit form of $\hat \Delta$ is complicated, and we defer its expression to Appendix \ref{arxivsupp:main:proofsforLDPapp} of the Supplementary Material. 
Denote $\|\bbeta^*\|_0 = s$ and $\|\vb^*\|_0 = s_{\vb}$. Once again for simplicity of the presentation we assume that $\|\vb^*\|_1$ and $\|\bbeta^*\|_1$ are bounded. We obtain the following asymptotic normality result.

\begin{corollary} \label{IFsparseLDA} Assume that $\lambda_{\min}(\bSigma) > \delta$ for some constant $\delta>  0$, and let Assumption \ref{var:bound:LDA} hold. If 
\begin{equation}\label{eqconLDA}
\max(s_{\vb},s)\log d/\sqrt{n} = o(1), ~~~~ \exists~ k > 2: (s_{\vb}s)^{k}/n^{k - 1} = o(1),
\end{equation}
holds and $\lambda \asymp \sqrt{\log d/n}$ and $\lambda' \asymp \sqrt{\log d/n}$, then for each $t\in \RR$:
$$
\lim_{n \rightarrow \infty} |\PP^*(\hat U_n < t) - \Phi(t)| = 0.
$$
\end{corollary}

The second part of (\ref{eqconLDA}) is similar to that in Corollary \ref{IFCLIME},  which is used to establish the Lyapunov's condition for central limit theorem.
The proof of Corollary \ref{IFsparseLDA} can be found in Appendix \ref{arxivsupp:main:proofsforLDPapp} of the Supplementary Material. %In addition in the same Appendix, in Remark \ref{consistent:est:sLDA} we define a consistent estimator for $\Delta$.

\subsection{Stationary Vector Autoregressive Models}\label{SVAsec}

In this section we develop inferential methods for the lag-1 vector autoregressive models considered in the introduction. Let $\bbeta^* = \Ab_{*m}$ be the $m$\textsuperscript{th} column of $\Ab$. If one is only interested in  estimating $\bbeta^*$, instead of solving the whole problem one can only solve the corresponding subproblem:
\begin{align}
\hat \bbeta = \argmin_{\bbeta \in \RR^{d}} \|\bbeta\|_1, \mbox{ subject to } \|\Sb_0 \bbeta - \Sb_{1,*m}\|_{\infty} \leq \lambda. \label{optproblemvecautosubprob}
\end{align}

\cite{han2014direct} showed that procedure (\ref{optproblemvecauto}) consistently estimates $\Ab$ under certain sparsity assumptions. In the following, we apply our method to  construct confidence intervals for $\theta$, where $\theta$ is the first component of $\bbeta$, i.e., $\bbeta = (\theta, \bgamma^T)^T$. By the equation (\ref{eqA}), we can construct the  $d$-dimensional estimating equations  $\tb(\Xb, \bbeta) = \Sb_0 \bbeta - \Sb_{1,*m}$. Then the projected estimating equation for $\theta$ is given by
$$
\hat S(\bbeta) = \hat \vb^T(\Sb_0 \bbeta - \Sb_{1,*m}),
$$
where
$$
\hat \vb = \min_{\vb \in \RR^d} \|\vb\|_1, \mbox{ subject to } \|\vb^T \Sb_0 - \eb_1\|_{\infty} \leq \lambda',
$$
is an estimator of $\vb^{*T} = (\bSigma_{0}^{-1})_{1*}$. Define $\tilde{\theta}$ to be the solution to $\hat S(\theta, \hat \bgamma) = 0$.
%$$
%\tilde \theta = \hat \theta - \frac{\hat \vb^T(\Sb_0 \hat \bbeta - \Sb_{1,*m})}{\hat \vb^T\Sb_{0,*1}}.
%$$
Note that in this framework the estimating equation  $\tb(\Xb, \bbeta)=\Sb_0 \bbeta - \Sb_{1,*m}$ decomposes into a sum of dependent random variables. To handle this challenge, our main technical tool is the martingale central limit theorem and concentration inequalities for dependent random variables. To establish the asymptotic normality of $\tilde \theta$, we define the following classes of matrices:
\begin{align*}
\mathcal{M}(s) &:= \left\{\Mb \in \RR^{d \times d}: \max_{1 \leq j \leq d} \|\Mb_{*j}\|_0 \leq s, \|\Mb\|_1 \leq M, \|\Mb\|_2 \leq 1 - \epsilon \right\}, \\ \cL &:= \left\{\Mb \in \RR^{d \times d}: \|\Mb^{-1}\|_1 \leq M, \|\Mb\|_2 \leq M\right\},
\end{align*}
where $M$ and $1 > \epsilon > 0$ are some fixed constants. In the following, we can show that $T^{1/2}(\tilde \theta - \theta^*)$ converges weakly to  $N(0,\Delta)$, where
\begin{align}
	\Delta := \Psi_{mm} \vb^{*T}\bSigma_0\vb^*. \label{Delta:SVA}
\end{align}
Let $\hat \Delta = (\Sb_{0,mm} - \hat \bbeta^{T} \Sb_0 \hat \bbeta)(\hat \vb^{T}\Sb_0 \hat \vb)$ be an estimator of the asymptotic variance $\Delta$, and define $\hat U_n := \hat \Delta^{-1/2}T^{1/2}(\tilde \theta - \theta^*)$. 
%Define $\hat \Delta = (\Sb_{0,mm} - \hat \bbeta^{T} \Sb_0 \hat \bbeta)(\hat \vb^{T}\Sb_0 \hat \vb)$ as an estimator of the asymptotic variance $\Delta$. 
We have the following asymptotic normality result.

%We then proceed to formulate an influence function expansion.
\begin{corollary}\label{IFexpansionSVA} Let $\bSigma_0 \in \cL, \Ab \in \cM(s)$,  $\min_j \Psi_{jj} \geq C > 0$ and $\|\vb^*\|_0 = s_{\vb}$. Then there exist $\lambda \asymp \sqrt{\log d/T}$ and $\lambda' \asymp \sqrt{\log d/T}$ such that if
$\max(s_{\vb}, s)\log d  = o(\sqrt{T}),$
we have for all $t\in\RR$:
$$
\lim_{T \rightarrow \infty}|\PP^*(\hat U_n \leq t) - \Phi(t)| \rightarrow 0.
$$
\end{corollary}

Similarly to \cite{han2014direct}, we assume that the matrix $\Ab$ belongs to $\cM(s)$, for the estimation purpose.
 The proof of Corollary \ref{IFexpansionSVA} is given in Appendix \ref{SVAproofs} of the Supplementary material. In this section, we only discussed the lag-1 autoregressive model. As mentioned in \cite{han2014direct}, lag-$p$ models can be accommodated in the current lag-1 model framework. Thus, similar methods can be applied to construct confidence intervals under the lag-$p$ model. %This example shows that the Z-estimator $\tilde{\theta}$ is asymptotically normal under the autoregressive model. 

\section{Numerical Results}\label{numericalStudySec}

In this section we present numerical results to support our theoretical claims. Numerical studies on hypothesis testing are available from the authors upon request.

\subsection{Linear Model}

In this section, we compare our estimating equation (EE) based procedure with two existing methods: the desparsity \citep{van2013asymptotically} and the debias \citep{javanmard2013confidence} methods in linear models. Note that in their methods the LASSO estimator is used as an initial estimator. 

Our simulation setup is as follows. We first generate $n = 150$ observations $\bX \sim N(0, \bSigma_{\bX})$, where $\bSigma_{\bX}$ is a Toeplitz matrix with $\bSigma_{\bX,ij} = \rho^{|i-j|}, i,j = 1,\ldots, d$. We consider three scenarios for the correlation parameter $\rho =0.25, 0.4, 0.6$ and three  possible values of the dimension $d = 100, 200, 500$. We generate $\bbeta^*$ under two settings. In the first setting, $\bbeta^*$ is held fixed, i.e.,  $\bbeta^* = (1,1,1, 0,\ldots,0)^T$, and in the second setting we take $\bbeta^* = (U_1, U_2, U_3, 0,\ldots,0)^T$, where $U_i$ follows a uniform distribution on the interval $[0,2]$ for $i = 1,2,3$. The former setting is labeled as ``Dirac'' and the latter as ``Uniform'' in Table \ref{cov:table} below. Both settings have three nonzero values, i.e., $\|\bbeta^*\|_0=3$. The outcome is generated by $Y = \bX^T\bbeta^* + \varepsilon$, where $\varepsilon \sim N(0,1)$. The simulations are repeated 500 times. The tuning parameter $\lambda$ is selected by a 10-fold cross validation. The parameter $\lambda'$ is manually set to $\frac{1}{2}\sqrt{\log d/n}$. We discover that the result is robust with respect to the choice of $\lambda'$. Based on the selected $\lambda$ and $\lambda'$, we construct the confidence intervals for the first component of $\bbeta$.
\begin{table}[H] 
\caption{The empirical coverage probability of 95\% confidence intervals constructed by our estimating equation (EE) based method, desparsity  and debias methods under the linear model. The average length of confidence intervals is shown in parenthesis.} \label{cov:table}
\label{sizetable}
\begin{center}
\begin{tabular}{cccccccc}
\hline
\multicolumn {2}{c}{} & \multicolumn{3}{c}{Uniform} & \multicolumn{3}{c}{Dirac}  \\%& \multicolumn{2}{c|}{linear}\\
\cline{3-8}
%  \hline
$d$ & method & $\rho =$ 0.25 &  $\rho =$ 0.4 &  $\rho =$ 0.6 & $\rho =$ 0.25 & $\rho =$ 0.4 & $\rho =$ 0.6\\%& n=500 & n=1000\\
 \hline
\multirow{3}{*}{100} & EE & 0.94 (0.3) & 0.95 (0.4) & 0.95 (0.4)  & 0.95 (0.3) & 0.96 (0.4) & 0.95 (0.4)\\
& desparisty & 0.96 (0.4) & 0.96 (0.4) & 0.95 (0.4) & 0.96 (0.4) & 0.95 (0.4) & 0.94 (0.4)\\
& debias & 0.95 (0.3) & 0.95 (0.4) & 0.94 (0.4) & 0.95 (0.4) & 0.94 (0.4) & 0.95 (0.4)\\
\hline
\multirow{3}{*}{200} & EE & 0.95 (0.3) & 0.96 (0.4) & 0.95 (0.5) & 0.95 (0.4) & 0.94 (0.4) & 0.95 (0.5)\\
 & desparisty & 0.94 (0.3) & 0.95 (0.4) & 0.95 (0.4) & 0.95 (0.4) & 0.96 (0.4) & 0.96 (0.5)\\
& debias & 0.95 (0.4) & 0.95 (0.4) & 0.95 (0.4) & 0.96 (0.4) & 0.96 (0.4) & 0.95 (0.5)\\
\hline
\multirow{3}{*}{500} & EE &0.96 (0.4) & 0.96 (0.4) & 0.95 (0.4) & 0.95 (0.4) & 0.96 (0.4) &0.95 (0.4)\\
 & desparisty & 0.96 (0.4) & 0.95 (0.4) & 0.95 (0.4) & 0.96 (0.4)  & 0.96 (0.5) & 0.96 (0.5) \\
& debias & 0.95 (0.4) & 0.95 (0.4) & 0.94 (0.5)  & 0.95 (0.5) & 0.95 (0.5) & 0.94 (0.5)\\
\hline
\end{tabular}\label{sizeDantzig}
\end{center}
\end{table}

In Table \ref{cov:table}, we summarize the empirical coverage probability of 95\% confidence intervals and their average lengths of our estimating equation (EE) based method, desparsity  and debias methods. We find that the empirical coverage probability of our method is very close to the desired nominal level. In particular, our method tends to have shorter confidence intervals than the existing two methods, when the dimension is large (e.g. $d=500$).   

%As we can see from the power plots, our proposed test statistic performs similarly to the ones proposed in the two other papers in the linear model. This is to be expected as all of the three statistics are asymptotically optimal and hence the powers should be equivalent.

% THIS IS THE PROOF IN CASE WE NEED IT
%\begin{proof}
%It is easy to see that from (\ref{onestepsparseLDA}), we have:
%$$
%n^{1/2}(\tilde \theta - \theta^*) = - n^{1/2} \frac{\hat \vb^T (\hat \bSigma_n\tilde \bbeta - (\bar X - \bar Y))}{\hat \vb^T (\hat \bSigma_n)_{\cdot 1} }
%$$
%where $\tilde \bbeta = \hat \bbeta - (\hat \theta - \theta^*)\eb_m^T$. Noting that $|\hat \vb^T (\hat \bSigma_n)_{\cdot 1}- 1| \leq \lambda'$, and $\lambda' \asymp \|\vb^*\|_1 \sqrt{\log d/n} = o(1)$, and thus by Corollary \ref{corsparseLDAnorm}:
%$$
%n^{1/2}(\tilde \theta - \theta^*) = - n^{1/2}\hat \vb^T (\hat \bSigma_n\tilde \bbeta -(\bar X - \bar Y)) + o_p(1) \rightsquigarrow N(0,\Delta)
%$$
%\end{proof}

\subsection{Graphical Models} 

In this section we compare our estimating equation (EE) based procedure to the desparsity method proposed by \cite{jankova2014confidence} based on the graphical LASSO. We consider two scenarios. In the first scenario, our data generating process is similar to \cite{jankova2014confidence}. Specifically, we consider a tridiagonal precision matrix $\bOmega$ with $\Omega_{ii} = 1, i = 1,\ldots,d$ and $\Omega_{i,i+1} = \Omega_{i+1,i} = \rho \in \{0.3, 0.4\}$ for $i = 1,\ldots,d-1$. Then we generate data from the Gaussian graphical model $\bX\sim N(0,\bOmega^{-1})$. We have three settings for $d = 60, 70, 80$, and we fix the sample size at $n = 250$, which is comparable to \cite{jankova2014confidence}. In the second scenario, we generate data from the transelliptical graphical model. 
Specifically, the latent generalized concentration matrix $\bOmega$ is generated in the same way as in the previous scenario, and then is normalized so that $\bSigma = \bOmega^{-1}$, satisfies $\diag(\bSigma) = 1$. Next, a normally distributed random vector $\bZ$ is generated through $\bZ \sim N(0, \bSigma)$, and is transformed to a new random vector $\bX=(X_1,...,X_d)$, where 
$$
X_j = \frac{f(Z_j)}{\sqrt{\int f^2(t) \phi\left(t\right) dt}},
$$
and $f(t) := \sign(t)|t|^{\alpha}$ is a symmetric power transformation with $\alpha=5$ and $\phi(t)$ is the pdf of a standard normal distribution. Then $\bX$ follows from the transelliptical graphical model with the latent generalized concentration matrix $\bOmega$. Similarly, we consider $d = 60, 70, 80$, and fix the sample size at $n = 250$. The simulations are repeated 500 times. The tuning parameters $\lambda = \lambda'$ are set equal to $0.5\sqrt{\log d/n}$. In the following, we construct confidence intervals for the parameter $\Omega_{12}$.  

\begin{table}[H] 
\caption{The empirical coverage probability of 95\% confidence intervals constructed by our estimating equation (EE) based method and the desparsity  method under the Gaussian graphical model and transelliptical graphical model. The average length of confidence intervals is shown in parenthesis.}
\label{sizetable}
\begin{center}
\begin{tabular}{cccccc}
\hline
&  & \multicolumn{2}{c}{Gaussian} & \multicolumn{2}{c}{Transelliptical}\\ 
\cline{3-6}
$d$ & method & $\rho = 0.3$  & $\rho = 0.4$ & $\rho = 0.3$  & $\rho = 0.4$\\
\hline
\multirow{2}{*}{60} & EE & 0.95 (0.3) & 0.94 (0.2) & 0.93 (0.3) & 0.94 (0.3) \\
& desparsity & 0.95 (0.3) & 0.95 (0.3)& 0.80 (0.3) & 0.44 (0.3)\\
\hline
\multirow{2}{*}{70} & EE & 0.95 (0.3) & 0.94 (0.2) & 0.92 (0.3) & 0.94 (0.3) \\
& desparsity & 0.95 (0.3) & 0.96 (0.3) & 0.74 (0.3) & 0.47 (0.3) \\
\hline
\multirow{2}{*}{80} & EE & 0.95 (0.3) & 0.95 (0.2) & 0.93 (0.3) & 0.94 (0.4) \\
 & desparsity & 0.94 (0.3) & 0.94 (0.3) & 0.70 (0.3) & 0.44 (0.3) \\
  \hline
\end{tabular}\label{CLIME}
\end{center}
\end{table}

In Table \ref{sizetable}, we present the empirical coverage probability of 95\% confidence intervals and their average lengths of our estimating equation (EE) based method, and the desparsity method. As expected, under the Gaussian graphical model, the confidence intervals of both methods have accurate empirical coverage probability and similar lengths. However, the desparsity method which imposes the Gaussian assumption shows significant under-coverage  for the transelliptical graphical model. In contrast, the proposed method preserves the nominal coverage probability, 
which demonstrates the numerical advantage of our method.

\subsection{Real Data Analysis}

In this section, we construct confidence regions for the gene network from the atlas of gene expression in the mouse aging project dataset \citep{zahn2007agemap}. The same dataset has been previously analyzed in \cite{ning2013high}, where the authors focus on a subset of $d = 37$ genes belonging to the mouse vascular endothelial growth factor signaling pathway in $8$ tissues. The number of replicates within each tissue is $n = 40$. 

Our analysis proceeds conditionally on each of the $8$ tissue types -- Adrenal (A), Cerebrum (C), Hippocampus (H), Kidney (K), Lung (L), Muscle (M), Spinal (S), Thymus (T). Namely, for each type of tissue, we construct the confidence intervals of each edges in the gene network by using our method and the procedure proposed by \cite{jankova2014confidence}. In particular, our inference is based on the approach developed in Section \ref{edgetestingsec} with the sample covariance matrix replaced by the rank covariance matrix defined in Remark \ref{trans:ell:graph:model};  see also section \ref{arxivsupp:main:SKEPTICCLIME} in the Supplementary Material for details. The tuning parameter $\lambda$ is determined by the $5$-fold cross-validation, under the Gaussian likelihood function, for a grid of values in the interval $[0.3, 0.8]$,  which is selected based on the fact that $\sqrt{\log d/n} \approx 0.3$. The tuning parameter $\lambda'$ is set to be the same as $\lambda$. The tuning parameter in \cite{jankova2014confidence} is selected by the same cross-validation method. 
%As described in our inference method, the first step is to obtain an initial estimator of the conditional independence structure of the genetic network. To achieve this we use two popular approaches based on the SKEPTIC \citep{liu2012transelliptical} and the Graphical LASSO \citep{friedman2008sparse} procedures respectively. The tuning parameters for the two procedures are determined by the $5$-fold cross-validation for a grid of values in the interval $[0.3, 0.8]$ which is selected based on the fact that $\sqrt{\log d/n} \approx 0.3$, using the Gaussian likelihood function. As a second step we compare inference algorithms based on these two fits, by comparing the procedure we develop (see section \ref{arxivsupp:main:SKEPTICCLIME} in the Supplementary Material) vs the procedure of \cite{jankova2014confidence}. The tuning parameters for the second step are selected to be the same as the ones in the first step. 

To perform the comparison we consider $2$ sets of genes which have been shown to be associated by biologists. The first set of genes -- Pla2g6, Ptk2 and Plcg2, comes from the group of PLC-$\gamma$ genes in the PKC-dependent pathway, and is crucial for ERK phosphorylation and proliferation \citep{holmes2007vascular}. The second set of genes is comprised of Mapk13, Mapk14 and Mapkapk2, which are related to the migration of endothelial cells. 
Instead of plotting confidence intervals for all the edges in the gene network, in Figure \ref{setofgenesone} we only plot confidence intervals for the $3$ edges connecting genes Pla2g6, Ptk2 and Plcg2, and genes Mapk13, Mapk14 and Mapkapk2, within each of the $8$ tissues. As we see from the plot, while most of the point estimates of our method and \cite{jankova2014confidence} are close, their variances differ drastically. The main reason is that in this dataset the gene expression values are highly non-Gaussian; see \cite{ning2013high} for demonstration. Thus, the inference procedure based on the Gaussian assumption \citep{jankova2014confidence} seems to provide inaccurate results with very wide confidence intervals. In contrast, the proposed method which relaxes the Gaussian assumption, produces confidence intervals with shorter length. In fact, most of the $95\%$ confidence intervals by the proposed method do not cover $0$, which concludes that these genes are statistically dependent. 
This result is consistent with the biological findings that genes Pla2g6, Ptk2 and Plcg2, and genes Mapk13, Mapk14 and Mapkapk2 are associated.

%It is evident from the figure that the confidence intervals which we develop reflect more accurately the true interactions of the $3$ genes of the PKC-dependent pathway within each of the $8$ tissues. 

%For Gaussian data the debiasing procedure proposed by \cite{jankova2014confidence} would be asymptotically efficient. However, any real data is unlikely to be truly Gaussian. In such cases agnostic procedures as the SKEPTIC could be much more appropriate, more efficient and posses sharper convergence rates. In the mouse aging project dataset, Figure \ref{setofgenesone} shows that our confidence intervals reflect better the underlying biological truth as compared to the intervals produced by the procedure suggested by \cite{jankova2014confidence}. %We believe that this is precisely the case in the mouse aging project dataset, as from Figure \ref{setofgenesone} it is evident that the variances given by our procedure, are much smaller than those produced by \cite{jankova2014confidence}, even though the majority of the point estimates are similar.

\begin{figure}[ht] %[H]
\minipage{0.32\textwidth}
  \centering
  \includegraphics[width=\linewidth]{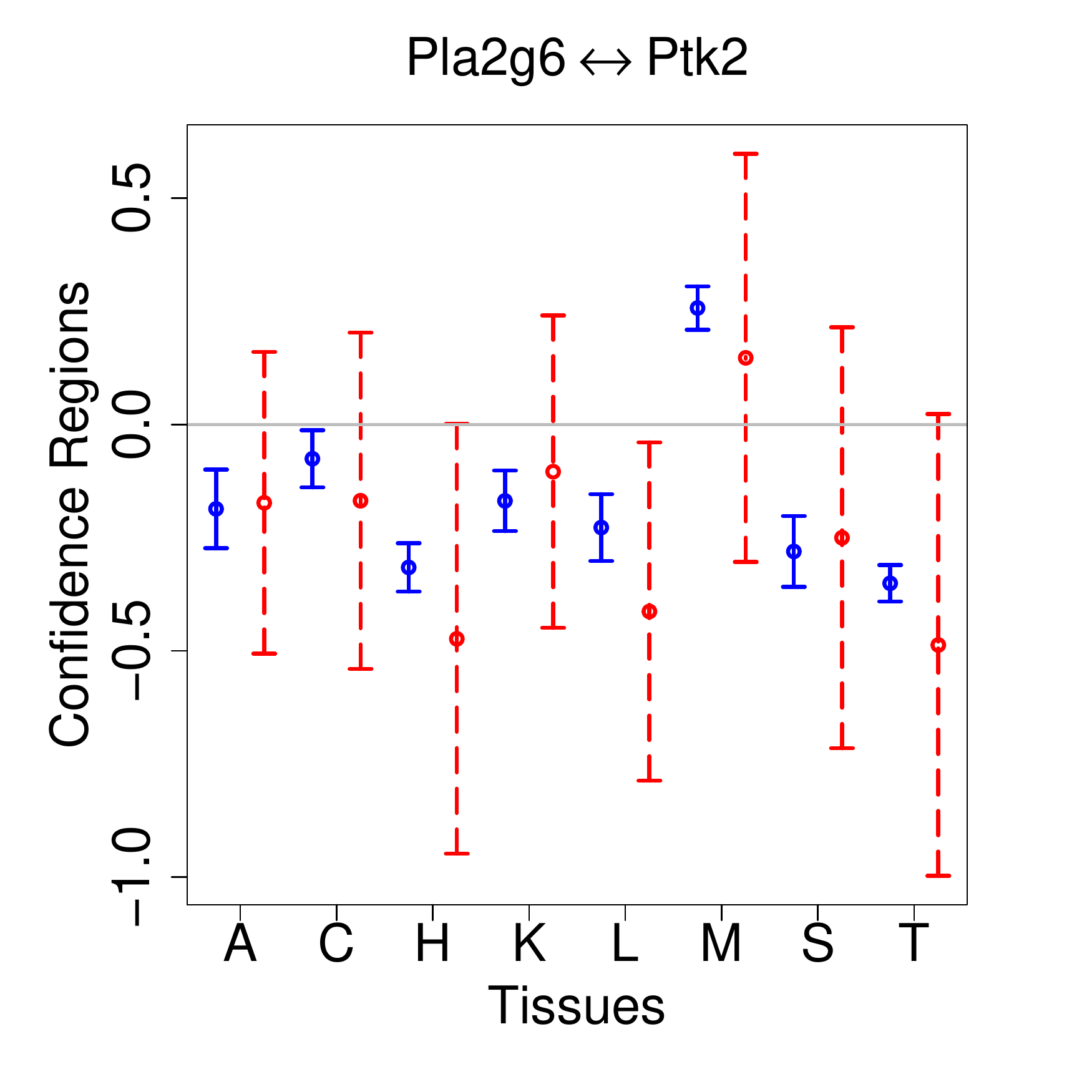}
%  \caption*{$\rho = 0.25, d = 100$}
  \label{arxivsupp:fig:sfig100025}
\endminipage  \hfill
\minipage{0.32\textwidth}
  \centering
  \includegraphics[width=\linewidth]{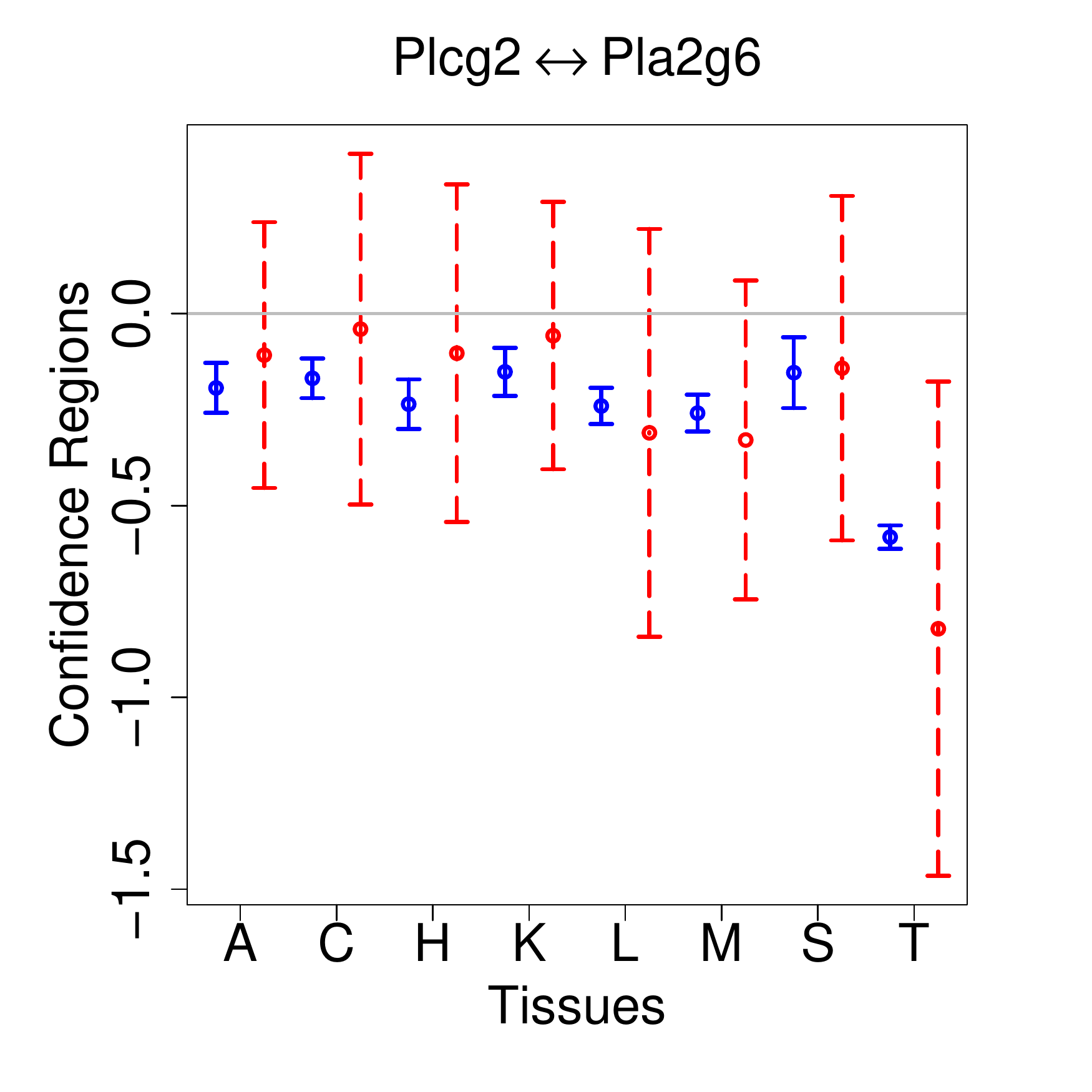}
%  \caption*{$\rho = 0.4, d = 100$}
  \label{arxivsupp:fig:sfig10004}
\endminipage\hfill
 \minipage{0.32\textwidth}
  \centering
  \includegraphics[width=\linewidth]{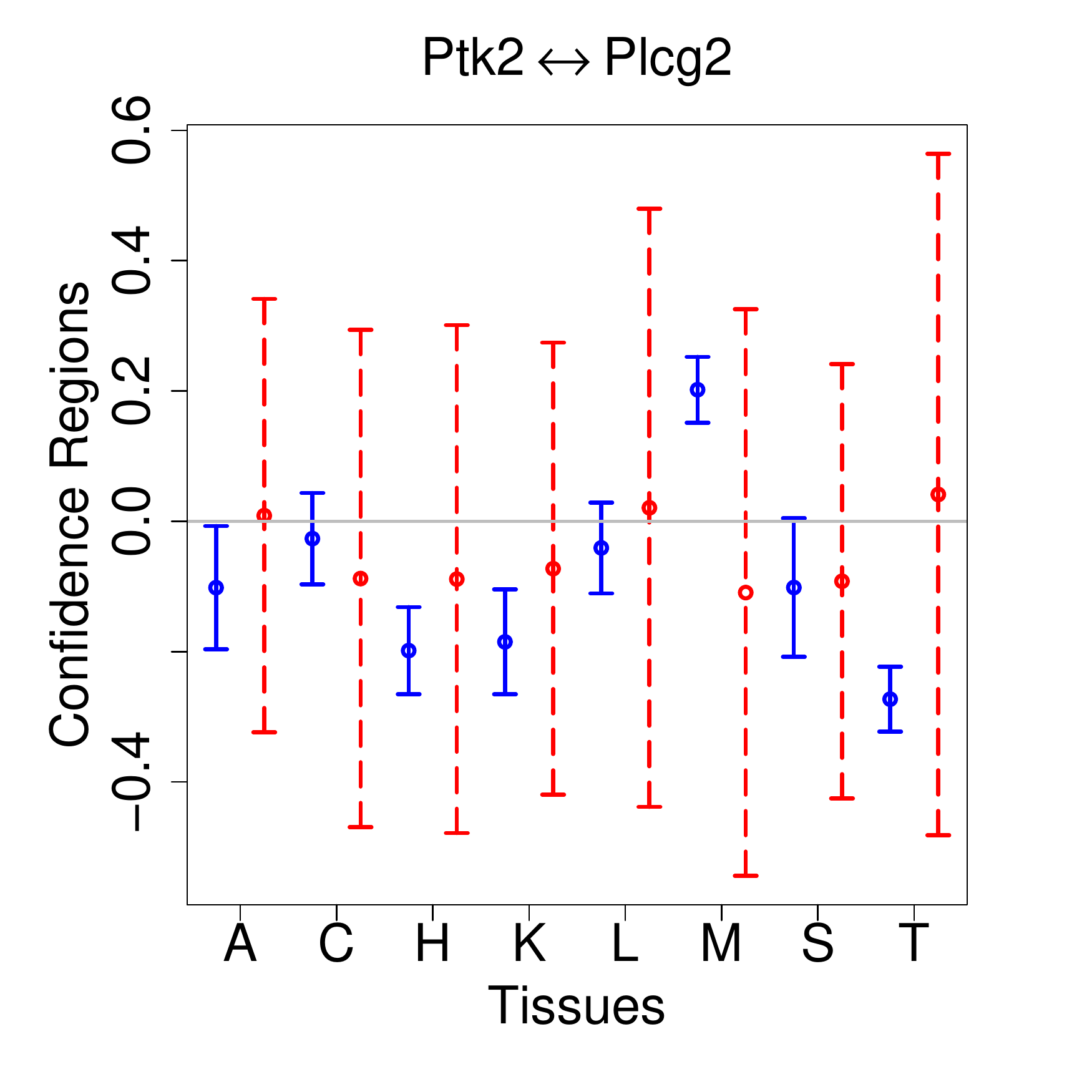}
%  \caption*{$\rho = 0.6, d = 100$}
  \label{arxivsupp:fig:sfig10006}
\endminipage

\minipage{0.32\textwidth}
  \centering
  \includegraphics[width=\linewidth]{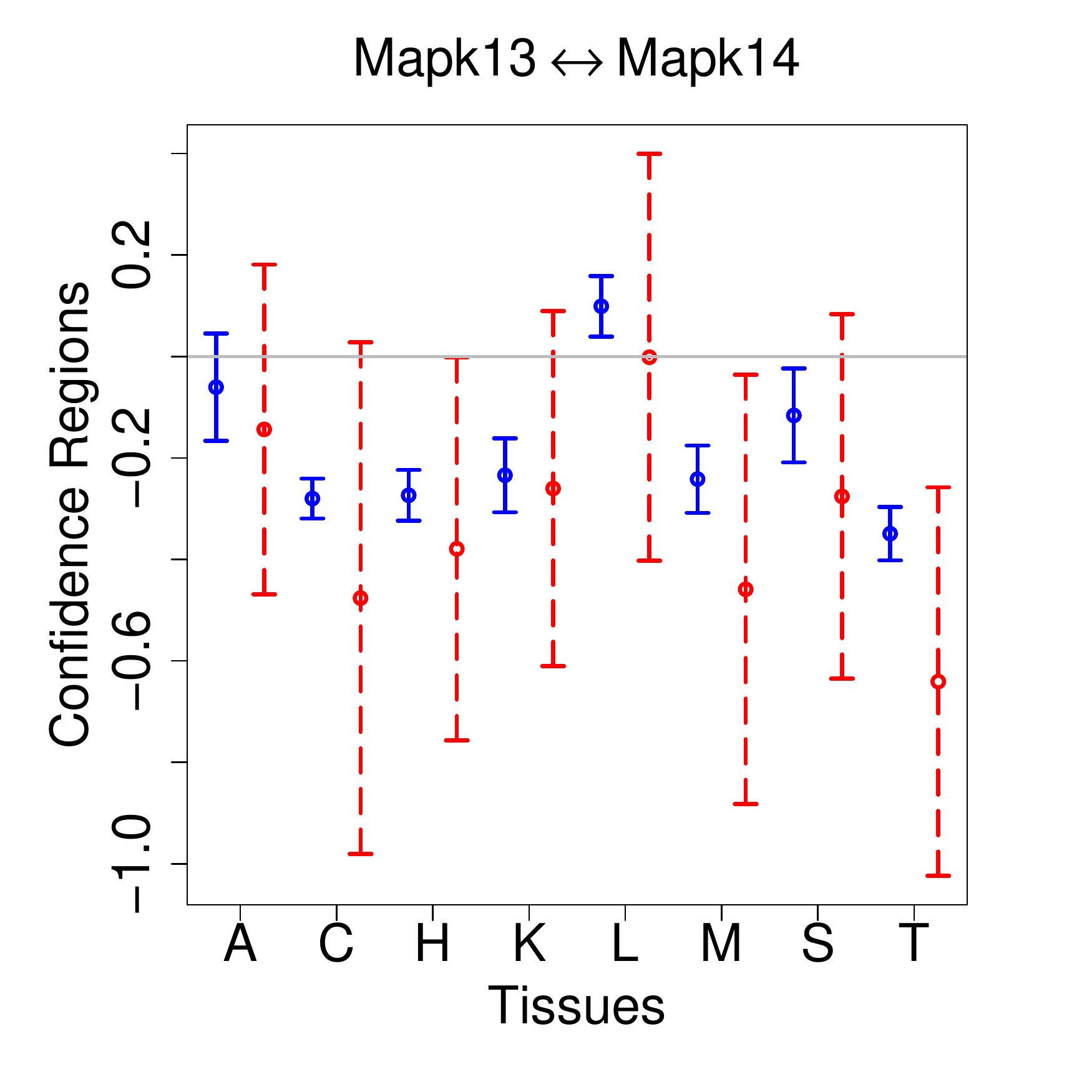}
  %\caption*{$\rho = 0.25, d = 200$}
  \label{arxivsupp:fig:sfig200025}
\endminipage  \hfill
\minipage{0.32\textwidth}
  \centering
  \includegraphics[width=\linewidth]{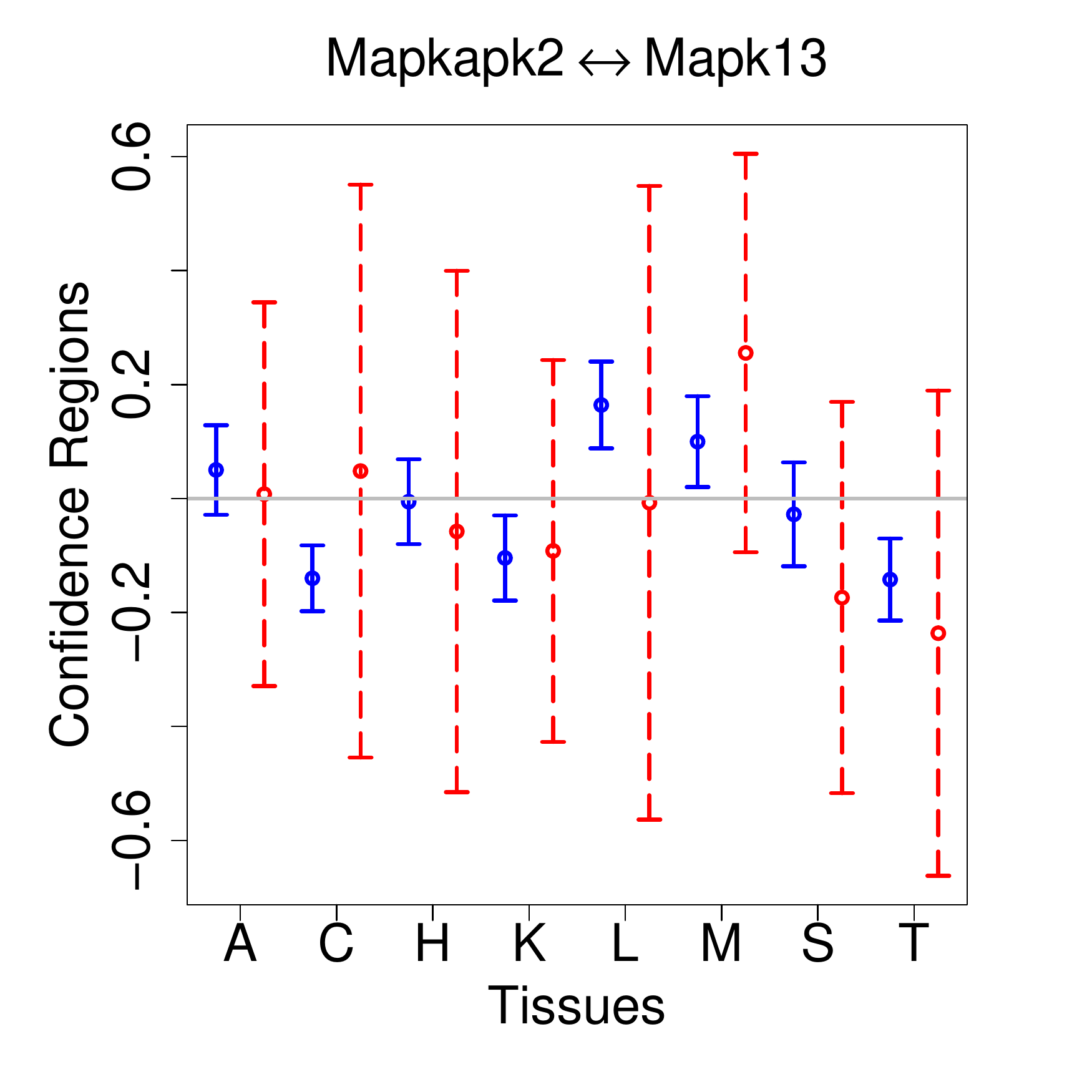}
%  \caption*{$\rho = 0.4, d = 200$}
  \label{arxivsupp:fig:sfig20004}
\endminipage\hfill
 \minipage{0.32\textwidth}
  \centering
  \includegraphics[width=\linewidth]{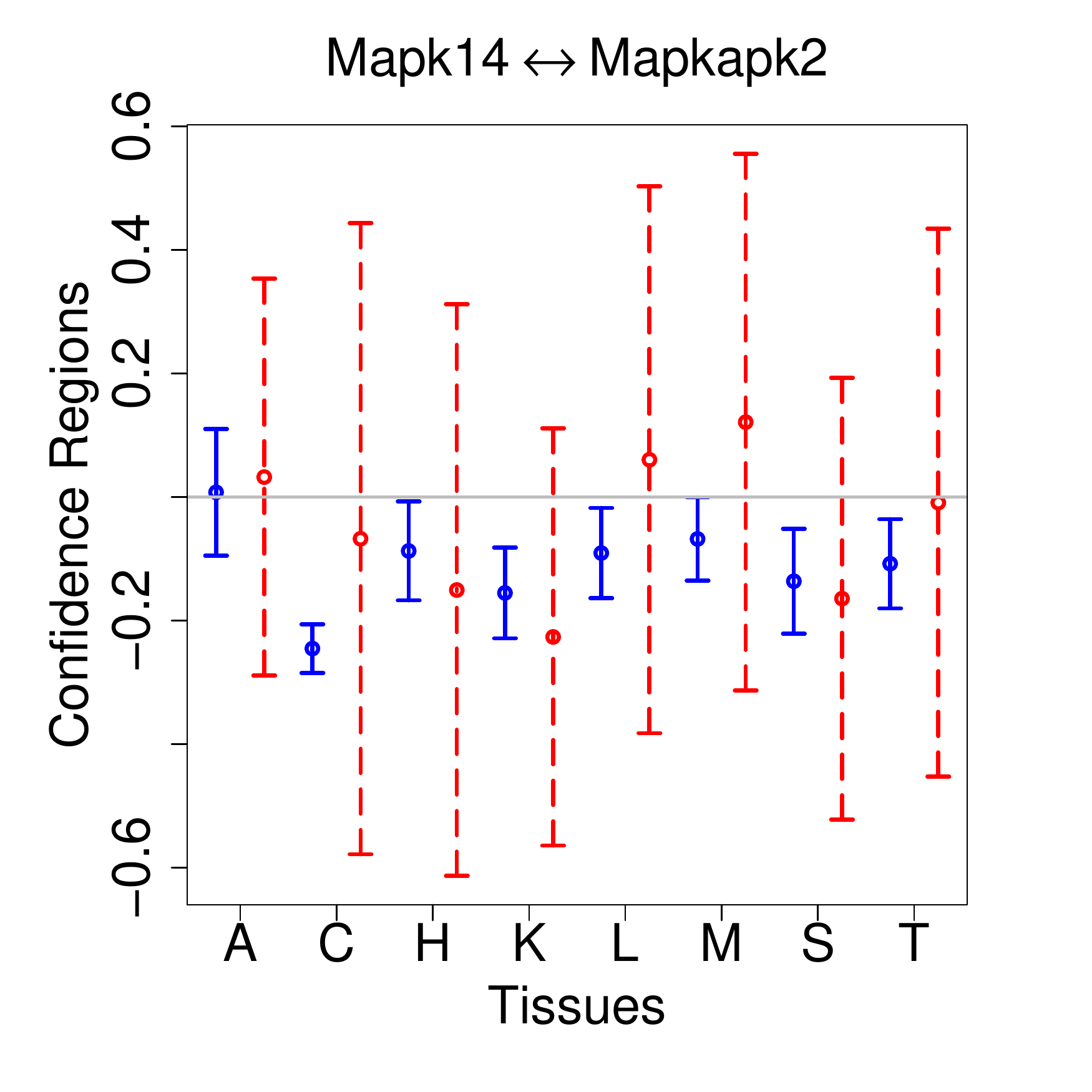}
%  \caption*{$\rho = 0.6, d = 200$}
  \label{arxivsupp:fig:sfig20006}
\endminipage
%\minipage{0.32\textwidth}
%  \centering
%  \includegraphics[width=\linewidth]{FIGS/rho0_25d500.pdf}
%  \caption*{$\rho = 0.25, d = 500$}
%  \label{arxivsupp:fig:sfig500025}
%\endminipage  \hfill
%\minipage{0.32\textwidth}
%  \centering
%  \includegraphics[width=\linewidth]{FIGS/rho0_4d500.pdf}
%  \caption*{$\rho = 0.4, d = 500$}
%  \label{arxivsupp:fig:sfig50004}
%\endminipage\hfill
% \minipage{0.32\textwidth}
%  \centering
%  \includegraphics[width=\linewidth]{FIGS/rho0_6d500.pdf}
%  \caption*{$\rho = 0.6, d = 500$}
%  \label{arxivsupp:fig:sfig50006}
%\endminipage
\caption{$95\%$ confidence intervals for the edges among the two sets of genes -- Plcg2, Pla2g6, and Pla2g6 (first row) and Mapk13, Mapk14 and Mapkapk2 (second row), within each of $8$ tissues indicated by their first letter. The confidence intervals based on our EE method are displayed in solid lines, while the intervals of \cite{jankova2014confidence} are displayed in dashed lines.}\label{setofgenesone}
\end{figure}

\section{Discussion} \label{discussionSec} In this paper we propose a generic procedure to construct confidence intervals for Z-estimators in a high dimensional setting. We establish a general theoretical framework, and illustrate it with several important applications including linear models, instrumental variables regression, graphical models, classification and time series models. Our framework has better numerical performance than previously suggested algorithms, and has the advantage of having a broader scope. In particular, it covers many applications (e.g. instrumental variables regression, linear discriminant analysis and vector autoregressive models) for which the inferential procedure is previously unexplored. 

Additionally our results can be easily extended to cases with multi-dimensional parameters of interest. We would like to mention that unlike approaches such as the ones developed by \cite{nickl2013confidence}, our methodology cannot be immediately extended to find a global \underline{honest} confidence region for the entire parameter $\bbeta$. %In contrast to \cite{nickl2013confidence}, our confidence regions asymptotically attain the nominal level. %and this key difference breaks once the number of parameters we would like to conduct inference for starts growing. 
It is an interesting problem to explore whether we can carry over certain results in the framework of honest confidence regions for $\bbeta$ under the linear regression considered by \cite{nickl2013confidence}, to the general estimating equations that we consider. We leave this question for future investigation.

%{\color{red} Include comments on multiple $\bbeta$}

%{\color{red} Include comments on tuning parameters}

%Much remains to be done in the current framework. In our future work, we will consider handling models with missing data and/or sampling bias. Moreover, we plan on extending the one parameter testing to the multiple testing case where the number of parameters will be allowed to scale with $n, d$. We note that the latter extension is not trivial. One possible approach is to use the multiplier bootstrap as described in \cite{chernozhukov2014gaussian}. 

%\newpage

\appendix

\setlength{\bibsep}{0.85pt}
{
\bibliographystyle{ims}
\bibliography{spglm}
}

\title{\huge Supplement: A Unified Theory of Confidence Regions and Testing for High Dimensional Estimating Equations}
%\author{}
%\author{Matey Neykov\footnotemark[1], ~~~  Yang Ning\footnotemark[2], ~~~ Jun S. Liu\footnotemark[3], ~~~ Han Liu\footnotemark[4]}
\date{}
\emptythanks
\maketitle{}
%\author{Matey Neykov\thanks{Department of Operations Research and Financial Engineering, Princeton University, Princeton, NJ 08544; e-mail: \href{mailto:mneykov@princeton.edu}{mneykov@princeton.edu}},~~~ Yang Ning\thanks{Department of Operations Research and Financial Engineering, Princeton University, Princeton, NJ 08544, USA; e-mail:\href{mailto:yning@princeton.edu}{yning@princeton.edu}}, ~~~Jun S. Liu\thanks{Department of Statistics, Harvard University, Cambridge, MA 02138, USA; e-mail: \href{mailto:jliu@stat.harvard.edu}{jliu@stat.harvard.edu}},~~~Han Liu\thanks{Department of Operations Research and Financial Engineering, Princeton University, Princeton, NJ 08544, USA; e-mail:\href{mailto:hanliu@princeton.edu}{hanliu@princeton.edu}}}
%\maketitle

\section{Theory on  Uniformly Valid Confidence Intervals} \label{arxivsupp:main:UWCUNsec}

In section \ref{masterthm:sec}, we showed that if $\bbeta^*$ is fixed, the solution to the equation $\hat S(\theta, \hat \bgamma) = 0$ can be used to construct  asymptotically valid confidence regions for the parameter $\theta$. In this Section we prove a stronger result which guarantees that the confidence interval is valid uniformly over the following parameter space:
$$
\bOmega = \{ \bbeta^* : \|\bbeta^*\|_0 \leq s^*\}.
$$

We restrict our attention to $\bOmega$ since we need the parameter $\bbeta^*$ to be sufficiently sparse in order for us to consistently estimate it. In the following, we introduce some assumptions guaranteeing uniform convergence. 

\begin{assumption}[Uniform Consistent Estimation] \label{arxivsupp:main:Uconsistencyassumpweakcn}
\begin{align*}
\lim_{n \rightarrow \infty} \sup_{\bbeta \in \bOmega} \PP_{\bbeta}(\|\hat \bbeta - \bbeta\|_1  \leq r_{1}(n)) = 1, ~~~~ \lim_{n \rightarrow \infty} \sup_{\bbeta \in \bOmega} \PP_{\bbeta}(\|\hat \vb - \vb\|_1 \leq r_{2}(n)) = 1,
\end{align*}
where $r_{1}(n), r_{2}(n) = o(1)$.
\end{assumption}

\begin{assumption} \label{arxivsupp:main:UnoiseassumpCI} Assume that there exists an $\eta > 0$ such that:
\begin{align}
\lim_{n \rightarrow \infty} \inf_{\bbeta \in \bOmega}\inf_{\check \theta \in \mathcal{N}_{\theta}} \PP_{\bbeta}(\| \tb(\Zb, \bbeta_{\check \theta}) - E_{\tb}(\bbeta_{\check \theta})\|_{\infty} \leq r_{3}(n)) & = 1, \label{arxivsupp:main:UbetastartassumpCI}\\
\lim_{n \rightarrow \infty} \inf_{\bbeta \in \bOmega} \inf_{\check \theta \in \mathcal{N}_\theta} \PP_{\bbeta}(|\vb^{T}\tb(\bZ, \bbeta_{\check \theta}) - \vb^{T}E_{\tb}(\bbeta_{\check \theta})| \leq r_{4}(n)) & = 1, \label{arxivsupp:main:UbetastartassumpCIvstar}\\
%\lim_{n \rightarrow \infty} \inf_{\bbeta \in \bOmega} \inf_{\check \theta \in \mathcal{N}_\theta} \PP_{\bbeta}\bigg( \sup_{\nu \in [0,1]} \|\hat \vb^{T} [\Tb(\Zb, \tilde \bbeta_{\nu}) ]_{-1} - \vb^{T} [E_{\Tb}(\bbeta_{\check \theta})]_{-1} \|_{\infty} \leq  r_5(n) \bigg)& = 1, \label{arxivsupp:UlambdaprimeasumpCI} 
\lim_{n \rightarrow \infty} \inf_{\bbeta \in \bOmega} \inf_{\check \theta \in \mathcal{N}_\theta} \PP_{\bbeta}\bigg( \sup_{\nu \in [0,1]} \|\hat \vb^{T} \Tb(\Zb, \tilde \bbeta_{\nu}) - \vb^{T} E_{\Tb}(\bbeta_{\check \theta}) \|_{\infty} \leq  r_5(n) \bigg)& = 1, \label{arxivsupp:main:UlambdaprimeasumpCI} 
\end{align}
where $\mathcal{N}_{\theta} = (\theta - \eta, \theta + \eta)$ and $\max(r_{3}(n), r_{4}(n), r_{5}(n)) = o(1)$, $\tilde \bbeta_\nu = \nu \hat \bbeta_{\check \theta} + (1 - \nu)\bbeta_{\check \theta}$. We also assume 
$$
\sup_{\bbeta \in \bOmega} \sup_{\check \theta \in \mathcal{N}_{\theta}} \|E_{\tb}(\bbeta_{\check \theta})\|_{\infty} < \infty
, ~~~~ \sup_{\bbeta \in \bOmega} \sup_{\check \theta \in \mathcal{N}_{\theta}}\| \vb^{T} [E_{\Tb}(\bbeta_{\check \theta})]_{-1}\|_{\infty} < \infty.
$$
\end{assumption}

We next prove the uniform consistency of the Z-estimator $\tilde{\theta}$, which is a uniform analogue of Theorem \ref{consistency:sol}.

\begin{proposition} \label{arxivsupp:main:Uconsistency:sol} Assume that the (stochastic) map $\check \theta \mapsto \hat S(\hat \bbeta_{\check \theta})$ is continuous with a single root or nondecreasing. In addition, assume that $\vb^{T} E_{\tb}( \bbeta_{\theta - \epsilon}) \times \vb^{T} E_{\tb}( \bbeta_{\theta + \epsilon}) < 0$ for any $\epsilon > 0$. Under conditions \ref{arxivsupp:main:Uconsistencyassumpweakcn} and \ref{arxivsupp:main:UnoiseassumpCI} we have that for any $\epsilon > 0$: $
\sup_{\bbeta \in \bOmega} \PP_{\bbeta}(|\tilde \theta - \theta| > \epsilon) = o(1).$
\end{proposition}

%\begin{remark} The statement remains valid even if we require that the maps $\check \theta \mapsto \hat S (\hat \bbeta_{\check \theta})$ are continuous with a single root or are nondecreasing with probability $1$.
%\end{remark}

\begin{assumption}[Uniform CLT]\label{arxivsupp:main:UCLTassump} For $\sigma^2 = \vb^{T} \bSigma \vb^{}$ we have:
\begin{align}
\lim_{n \rightarrow \infty} \sup_{\bbeta \in \bOmega} \sup_{t\in\RR} |\PP_{\bbeta}(\sigma^{-1}n^{1/2}S(\bbeta) \leq t ) - \Phi(t)| = 0 \label{arxivsupp:main:Unormalitycond},
\end{align}
where $\bSigma = \lim_{n\rightarrow \infty} n \Cov \tb(\Zb, \bbeta)$, and it is assumed that $\inf_{\bbeta \in \bOmega}\sigma^2 \geq C > 0$.
\end{assumption}

\begin{assumption}\label{arxivsupp:main:Ustabilityassump}  Assume that there exists a $\gamma > 0$ such that:
\begin{align}
%\lim_{n \rightarrow \infty} \sup_{\bbeta \in \bOmega} \sup_{\check \theta \in \mathcal{N}_\theta} \PP_{\bbeta}(| \hat \vb^T [\Tb(\Zb, \hat \bbeta_{\check \theta})]_{*1} - 1| \leq r_6(n)) = 1, \label{arxivsupp:Ustabasumpone}\\
\sup_{\bbeta \in \bOmega} \sup_{\max(\|\check \vb - \vb\|_1, \|\check \bbeta - \bbeta\|_1) < \gamma} \bigg|  \check \vb^T \frac{\partial}{\partial \theta} \left[\Tb(\Zb, \check \bbeta)\right]_{*1}\bigg| \leq \psi(\Zb) \label{arxivsupp:main:Ustabasumptwo},
\end{align}
where  $\psi$ is an integrable function such that $\sup_{\bbeta \in \bOmega} \EE\psi(\Zb) < \infty$.
\end{assumption}

\begin{assumption}[Uniform Consistency of Variance] \label{arxivsupp:main:UVCassump2} Assume there exists an estimator $\hat \sigma^2$ of $\sigma^2$, such that $
\lim_{n \rightarrow \infty} \inf_{\bbeta \in \bOmega}  \PP_{\bbeta}(|\hat \sigma^2 - \sigma^2| \leq r_6(n)) = 1, $
where $r_6(n) = o(1)$. 
\end{assumption}

In Section \ref{arxivsupp:main:unif:var:consistency} we provide sufficient conditions to obtain $\hat \sigma$ satisfying Assumption \ref{arxivsupp:main:UVCassump2}. Finally, we present a uniform weak convergence result for the Z-estimator which strengthens Theorem \ref{weakconvver2:CI}. Its proof can be found in Appendix \ref{arxivsupp:main:proofsofgeneraltheory}.

\begin{theorem}\label{arxivsupp:main:unifweakconv} Under Assumptions \ref{arxivsupp:main:Uconsistencyassumpweakcn} -- \ref{arxivsupp:main:UVCassump2}, the assumptions in Proposition \ref{arxivsupp:main:Uconsistency:sol} and $
n^{1/2}(r_{1}(n)r_{5}(n) + r_{2}(n)r_{3}(n)) = o(1),$ 
we have $
\lim_{n \rightarrow \infty}\sup_{\bbeta \in \bOmega} \sup_{t\in\RR} |\PP_{\bbeta}(\hat U_n \leq t) - \Phi(t)| = 0, ~~\textrm{where}~~\hat U_n = \frac{n^{1/2}}{\hat \sigma} (\tilde \theta - \theta). $
\end{theorem}

\begin{remark}\label{unifweakconv:CI} Notice that Theorem \ref{arxivsupp:main:unifweakconv} immediately implies that 
$$\lim_{n \rightarrow \infty} \sup_{\bbeta \in \bOmega} \sup_{t \in \RR_+} |\PP_{\bbeta}(|\hat U_n| \leq t) - \Phi(t) + \Phi(-t)| = 0.$$ 
The latter can be equivalently expressed as 
$$\lim_{n \rightarrow \infty} \sup_{\bbeta \in \bOmega} \sup_{t \in \RR_+} |\PP_{\bbeta}(\theta \in (\tilde \theta - \hat \sigma t/\sqrt{n}, \tilde \theta + \hat \sigma t/\sqrt{n})) - \Phi(t) + \Phi(-t)| = 0,$$
which implies that the confidence region $(\tilde \theta - \hat \sigma t/\sqrt{n}, \tilde \theta + \hat \sigma t/\sqrt{n})$ is uniformly valid over the parameter space $\bbeta \in \bOmega$ provided that the assumptions we discussed in this section hold. 
\end{remark}

\section{Sufficient Conditions for Variance Consistency}

In this brief Section we present sufficient conditions for having consistent variance estimators.

\subsection{Consistent Estimators of the Variance} \label{arxivsupp:main:var:consistency}

We now provide generic sufficient conditions for constructing a consistent estimate of the variance. Let $\hat \bSigma$ be a consistent estimator of $\Cov \tb(\Zb, \bbeta)$. In the case when $\tb(\Zb, \bbeta) = n^{-1} \sum_{i = 1}^n \hb(\bZ_i, \bbeta)$ one can use $\hat \bSigma := \frac{1}{n} \sum_{i = 1}^n \hb(\bZ_i, \hat \bbeta)^{\otimes 2}$. We consider the ``plugin'' estimator: $\hat \sigma^2 = \hat \vb^{T} \hat \bSigma \hat \vb$ of $\sigma^2$. Define: $\hat U_n = \hat \sigma^{-1} n^{1/2} (\tilde \theta - \theta^*)$. We are interested in showing that $\hat U_n$ converges  weakly to a standard normal distribution. To this end we define the following assumption:
%Define $\hat \bSigma := \frac{1}{n} \sum_{i = 1}^n \hb(\bZ_i, \hat \bbeta)^{\otimes 2}$. A great candidate for an estimate of $\vb^{*T} \bSigma \vb^{*}$ seems to be the ``plugin'' estimator: $\hat \sigma^2 = \hat \vb^{T} \hat \bSigma \hat \vb$. Define the statistic: $\hat U_n = \frac{n^{1/2}}{\sqrt{\hat \vb^{T} \hat \bSigma \hat \vb^{}}} (\tilde \theta - \theta^*)$. We are interested in showing that $\hat U_n$ converges  weakly to a standard normal distribution. To this end we define the following assumption:

\begin{assumption}[Variance Consistency]\label{arxivsupp:varconsistency} Assume that the following holds:
$$
\lim_{n \rightarrow \infty} \PP^*(\|\hat \bSigma - \bSigma\|_{\max} \leq r_7(n)) = 1,
$$
where $r_7(n) = o(1)$. 
\end{assumption}

\begin{proposition} \label{arxivsupp:weakconv} Assume the same assumptions as in Proposition \ref{weakconvver2:CI} plus Assumption \ref{arxivsupp:varconsistency}. Let furthermore $\|\bSigma\|_{\max} = O(1)$, $\|\vb^{*T} \bSigma\|_{\infty} r_2(n) = o(1)$ and $\|\vb^*\|_1^2 r_7(n) = o(1)$, then for any $t \in \RR$ we have:
$$
\lim_{n \rightarrow \infty} |\PP^*(\hat U_n \leq t) - \Phi(t))| = 0.
$$
\end{proposition}

The proof of Proposition \ref{arxivsupp:weakconv} can be found in Section \ref{arxivsupp:main:proofsofgeneraltheory}.

\subsection{Uniformly Consistent Estimators of the Variance} \label{arxivsupp:main:unif:var:consistency}

\begin{assumption}[Plugin  Variance Consistency] \label{arxivsupp:UVCassump} Assume that the following holds:
$$
\lim_{n \rightarrow \infty} \inf_{\bbeta \in \bOmega} \PP_{\bbeta}(\|\hat \bSigma - \bSigma\|_{\max} \leq r_8(n)) = 1,
$$
where $r_8(n) = o(1)$. %Let $\mathcal{G}_6^{\bbeta} = \{ \|\hat \bSigma - \bSigma\|_{\max} \leq r_6(n)\}.$
\end{assumption}  

We then have the following:

\begin{proposition}\label{arxivsupp:unifweakconv2} Assume that 
$$\sup_{\bbeta \in \bOmega}\|\bSigma\|_{\max} = O(1), ~~ \sup_{\bbeta \in \bOmega}\|\vb^{T} \bSigma\|_{\infty} r_2(n) = o(1), ~~ \sup_{\bbeta \in \bOmega}\|\vb\|_1^2 r_8(n) = o(1),$$ 
and $\hat \bSigma$ satisfies Assumption \ref{arxivsupp:UVCassump}. Then $\hat \sigma^2 = \hat \vb^T \hat \bSigma \hat \vb$ satisfies Assumption \ref{arxivsupp:main:UVCassump2}.
\end{proposition}

The proof of this proposition is omitted as it follows easily from the proof of Proposition \ref{arxivsupp:weakconv} and Lemma \ref{arxivsupp:main:normallemma} in the Appendix.

\section{Confidence Intervals for Instrumental Variables Regression}\label{dantzigselector:IVR}

Recall that $\bbeta := (\theta, \bgamma^T)^T$, and let $\bSigma_n = n^{-1} [\Xb; \Wb]^T [\Xb; \Wb]$ be the empirical estimator of $\bSigma_{} = \EE (\bX^T,\bW^T)^T (\bX^T,\bW^T)$. Conformally we decompose the matrices $\bSigma_{n}$ and $\bSigma_{}$ into four blocks corresponding to ordered pairings of $\bX$ and $\bW$ indicated by subscripting with the pair, e.g. $\bSigma_{\bW\bX,n}$ and $\bSigma_{\bW\bX}$ correspond to $n^{-1}\Wb^T\Xb$ and $\EE\bW\bX^T$ resectively. Our goal is to construct confidence intervals for the parameter $\theta$. It is easy to show that $\hat S(\bbeta)$ reduces to
$$
\hat S(\bbeta)=n^{-1}\hat\vb^T\Wb^T(\Xb\bbeta - \bY),
$$
where 
\begin{align}
\hat\vb=\argmin \|\vb\|_1,~~\textrm{subject to}~~ \|\vb^T\bSigma_{\bW\bX,n} -\eb_1\|_{\infty}\leq \lambda', \label{lambda:prime:dantz:IVR}
\end{align}
is an estimator of $\vb^* = \bSigma_{\bW\bX}^{-1} \eb_1^T$. We impose the following assumption.

\begin{assumption}\label{subGausserr:design:bounded:mom:matr:IVR} Let the error $\varepsilon:=Y-\bX^T\bbeta^*$,  the predictors $\bX$ and instruments $\bW$ be coordinate-wise sub-Gaussian, i.e.,
$$\|\varepsilon\|_{\psi_2} := K < \infty,  ~~~ \sup_{j \in \{1,\ldots,d\}} \max (\|X_j\|_{\psi_2},\|W_j\|_{\psi_2})  := K_{\bW\bX} < \infty,$$ 
for some fixed constants $K, K_{\bW\bX} > 0$ and $\Var(\varepsilon) \geq C_{\varepsilon} > 0$.  Furthermore, assume %, $\varepsilon$ and $\bX$ be independent, and the second moment matrix $\bSigma_{\bX}$ satisfies
$\lambda_{\min}(\bSigma_{\bW\bX}\bSigma_{\bX\bW}) \geq \delta^2 > 0,$ where $\delta$ is some fixed constant. Additionally recall that $\EE[\varepsilon^2 | \bZ] = \sigma^2$, $\EE[\bZ \varepsilon] = 0$.
\end{assumption}

\begin{assumption}\label{subGausserr:variance:assump:IVR} We impose the following assumptions on the covariance $\bSigma_{}$. Assume:
$$
\|\bSigma_{\bX\bW}^{-1}\bSigma_{\bW\bW}\bSigma_{\bW\bX}^{-1}\|_{\max} \leq D_{\max}, ~~~ \inf_{s \in \{1,\ldots,d\}, s \log d < \sqrt{n} \omega}\operatorname{CS}_{\bSigma_{\bW\bX}}(s, 1) \geq \kappa^* > 0, 
$$
for some sufficiently large $D_{\max} > 0$, $\omega > 0$ is a sufficiently small fixed constant, and the quantity $\operatorname{CS}_{\bSigma_{\bW\bX}}(s, 1)$ is defined in Definition \ref{arxivsupp:coord:sens:IVR}.
\end{assumption}

Assumption \ref{subGausserr:design:bounded:mom:matr:IVR} is mild and ensures that the random variables $\varepsilon,\bX,\bW$ are not heavy-tailed, the matrix $\bSigma_{\bW\bX}$ is invertible, the instrumental variables $\bW$ are uncorrelated with $\varepsilon$, and that $\varepsilon$ is homoscedastic given the instrumental variables. Assumption \ref{subGausserr:variance:assump:IVR} is a technical assumption. The first condition ensures that the random variable $\vb^{*T}\bW$ has a finite second moment, i.e. $\EE(\vb^{*T}\bW)^2 \leq D_{\max} < \infty$. The second condition implies that the matrix $\bSigma_{\bW\bX}$ is ``coordinate-wise sensitive'' with respect to the $L_1$ norm. Such a condition is first proposed by \cite{gautier2011high}, and can be viewed as an extension of the commonly used restricted eigenvalue (RE) condition. It is not hard to show that this condition holds (for all $s \in \{1,\ldots, d\}$) if for example $\lambda_{\min}(\bSigma_{\bX\bW} + \bSigma_{\bW\bX}) \geq 4 \kappa^*$, where the inequality is in the sense of eigenvalues comparison.

To construct confidence intervals for $\theta$, we consider $\hat U_n = \hat \Delta^{-1} n^{1/2}(\tilde{\theta} - \theta^*)$, where $\tilde \theta$ is defined as the solution to $\hat S(\theta, \hat \bgamma) = 0$, and  
\begin{align}\label{delta:hat:dantzig:IVR}
\hat \Delta := n^{-1} \sum_{i = 1}^n  ((\hat \vb^T \bW_i)(Y_i - \bX_i^T \hat \bbeta))^2,
\end{align}
is an estimator of the asymptotic variance $\Delta := \Var[\vb^{*T}\bW \varepsilon]$. 
As in the linear model we will assume that $\vb^*$ and $\bbeta^*$ are sparse. Let $s$ and $s_{\vb}$ denote the sparsity of $\bbeta^*$ and $\vb^*$ correspondingly, i.e., $\|\bbeta^*\|_0 = s$ and $\|\vb^*\|_0 = s_{\vb}$.
The next corollary of the general Theorem \ref{weakconvver2:CI} shows the asymptotic normality of $\hat U_n$ in instrumental variable regressions. To simplify the presentation of our result we will assume that $\|\vb^*\|_1$ is bounded, although this is not needed in our proofs.

\begin{corollary} \label{normdantzigfinal:IVR} 

Assume that condition \ref{subGausserr:design:bounded:mom:matr:IVR} and \ref{subGausserr:variance:assump:IVR} hold, and 
\begin{align*}
\max(s_{\vb}, s)\log d/\sqrt{n} = o(1), ~~~ \sqrt{\log d/n} = o(1).
\end{align*}
Then with $\lambda \asymp \sqrt{\log d/n}$ and $\lambda' \asymp \sqrt{\log d/n}$, $\hat U_n$ satisfies for any $t\in\RR$:
$$
\lim_{n \rightarrow \infty} |\PP^*(\hat U_n \leq t) - \Phi(t))| = 0.
$$
\end{corollary}
The proof of Corollary \ref{normdantzigfinal:IVR} can be found in Appendix \ref{arxivsupp:main:DantzigProofs:IVR}. The conditions in Corollary \ref{normdantzigfinal:IVR} agree with the existing conditions in  \cite{zhang2014confidence,van2013asymptotically} for the simple linear model. In fact, under the additional assumption $s_{\vb}^3/n = o(1)$, we can show that  $\hat U_n$ is uniformly asymptotically normal; see Remark \ref{arxivsupp:main:unif:CI:danztig:IVR}.

\section{Confidence Intervals for Transelliptical Models}\label{arxivsupp:main:SKEPTICCLIME}

In this subsection we consider the transelliptical graphical models (TGM), proposed by \cite{liu2012transelliptical}. We recall several definitions before we proceed.

\begin{definition}[Elliptical distribution \cite{fang1990symmetric}] Let $\bmu \in \RR^d$ and $\bSigma \in \RR^{d\times d}$. We say that the $d$-dimesnional vector $\bX$ has an elliptical distribution, and we denote it with $\bX \sim EC_d(\bmu,\bSigma, \xi)$ if $\bX \stackrel{d}{=}\bmu + \xi \Ab \bU$, where $\bU$  is a random vector uniformly distributed on the unit sphere in $\RR^q$, $\xi \geq 0$ is a scalar random variable independent of $\bU$, $\Ab \in \RR^{d \times q}$ is a deterministic matrix such that $\Ab \Ab^T = \bSigma$.
\end{definition}

\begin{definition}[Transelliptical distribution \cite{liu2012transelliptical}] 
	We call the continuous random vector $\bX= (X_1,\ldots, X_d)^T$ transelliptically distributed, and we denote it with $\bX \sim TE_d(\bSigma, \xi; f_1,\ldots,f_d)$, if there exists a set of monotone univariate functions $f_1,\ldots, f_d$ and a non-negative random variable $\xi$, with $\PP(\xi = 0) = 0$, such that:
	$$
	(f_1(X_1),\ldots, f_d(X_d))^T \sim EC_d(0,\bSigma, \xi),
	$$
	where $\bSigma$ is symmetric with $\diag(\bSigma) = 1$ and $\bSigma > 0$ in a positive-definite sense. Here $\bSigma$ is called the ``latent generalized correlation matrix''.
\end{definition}

%It is worth mentioning that the family of transelliptical distributions (TD) is a broad family of distributions, subsuming the family of nonparanormal distributions e.g., the definition of which can be found in \cite{liu2012high}.

The graphical structure in TGMs can then be defined through the notion of the ``latent generalized concentration matrix'' --- $\bOmega =  \bSigma^{-1}$, i.e. an edge is present between two variables $X_j,X_k$ if and only if  $\Omega_{jk} \neq 0$. To construct an estimate of $\bOmega$, \cite{liu2012transelliptical} suggested estimating the correlation matrix $\bSigma$ first. This can be done by using a non-parametric estimate of the correlation such as Kendall's tau, and transforming it back, to obtain an estimate of $\bSigma$. 

Assume that $\bX_1,\ldots, \bX_n$ are i.i.d. copies of $\bX$. Recall the definitions of $\hat \tau_{jk}$ and $\hat \bS^{\tau}_{jk}$ given in Remark \ref{trans:ell:graph:model}. %The Kendall's tau statistic is a matrix whose $jk$\textsuperscript{th} element is:
%$$
%\hat \tau_{jk} = \frac{2}{n(n-1)} \sum_{1 \leq i < i' \leq n} \sign \left((X_{ij} - X_{i'j})(X_{ik} - X_{i'k})\right).
%$$
Note that it is clear from the definition of $\hat \tau_{jk}$, that it is an unbiased estimator of:
$$
\tau_{jk} = \PP((Y_j - Y'_{j})(Y_k - Y'_{k}) > 0) - \PP((Y_j - Y'_{j})(Y_k - Y'_{k}) < 0),
$$
where $\bY, \bY' \sim \bX$ are i.i.d. random variables. %Define:
%$$
%\hat \bS^{\tau}_{jk} = \begin{cases}\sin\left(\frac{\pi}{2} \hat \tau_{jk}\right), &j \neq k; \\ 1, &j = k. \end{cases}
%$$
It can be seen that $\hat \bS^{\tau}_{jk}$ consistently estimates $\bSigma$ (see e.g. \citep{liu2012high}). Let $\bOmega^* = \bSigma^{-1}$. The TGM estimator with CLIME is given by:
$$
\hat{\bOmega}=\argmin\|\bOmega\|_1,~~\textrm{such that}~~ \|\hat \bS^{\tau}\bOmega-\Ib_d\|_{\max}\leq\lambda.
$$
To derive confidence intervals for the parameter $\Omega^*_{1m}$, we can apply a similar approach to the graphical models. Denote with $\bbeta=\bOmega_{* m}$. Then the CLIME with TGM reduces to
$$
\hat{\bbeta}=\argmin\|\bbeta\|_1,~~\textrm{such that}~~ \|\hat \bS^{\tau}\bbeta-\eb_m^T\|_{\infty}\leq\lambda.
$$
According to our formulation of the test statistic we have $
\hat S(\bbeta)=\hat{\vb}^T(\hat \bS^{\tau} \bbeta-\eb_m^T),$
where
\begin{align}\label{arxivsupp:main:lambda:prime:tgm:clime}
\hat{\vb}=\argmin\|\vb\|_1,~~\textrm{such that}~~ \|\vb^T\hat \bS^{\tau}-\eb_1\|_{\infty}\leq\lambda'.
\end{align}
The solution $\tilde \theta$ to equation $\hat S(\theta, \hat \bgamma) = 0$, has a closed form expression in this example, and it is given below:
\begin{align} \label{arxivsupp:main:onestepSKEPTIC}
\tilde \theta = \hat \theta - \frac{\hat \vb^T (\hat \bS^\tau\hat \bbeta - \eb_m^T)}{\hat \vb^T \hat \bS^\tau_{* 1}}.
\end{align}
Next we argue that $\tilde \theta$ can be  used to construct confidence intervals for the parameter $\theta$. We note that the estimating equation in the TGM with CLIME is not a sum of i.i.d. statistics, due to the $U$-statistic structure of $\bS^{\tau}$ as compared to estimating equations we considered in our previous examples. Nevertheless, the assumptions in Section \ref{LZEHDSsection} are general enough to handle such a case. Let $\bTheta$ be a $d \times d$ random matrix with entries $\Theta_{jk} := \pi \cos\left(\frac{\pi}{2}\tau_{jk}\right)\tau^{\bY}_{jk}$, where:
$$
\tau^{\bY}_{jk} = [\PP((Y_j - Y'_{j})(Y_k - Y'_{k}) > 0 | \bY) - \PP((Y_j - Y'_{j})(Y_k - Y'_{k}) < 0 | \bY) - \tau_{jk}],
$$
with $\bY, \bY'$ are i.i.d. copies of $\bX$ (and all $\tau^{\bY}_{jk}$ being a random variable depending on $\bY$). Define
\begin{align}\label{arxivsupp:main:deltaSKEPTICdef}
\Delta :=  \Var(\vb^{*T} \bTheta \bbeta^*).
\end{align}
\begin{assumption} \label{arxivsupp:main:distr:assump:tgm:clime} Let $\bX$ satisfy the following distributional assumption --- there exists $\alpha_{\min} > 0$ such that $
	\Delta \geq \alpha_{\min} \|\vb^*\|^2_2 \|\bbeta^*\|_2^2.%~~~ \alpha_{\min} > 0
	$
\end{assumption}
\begin{remark} As in the CLIME case, we can show that the condition $\Var(\vb^{*T} \bTheta \bbeta^*) \geq \alpha_{\min} \|\vb^*\|^2_2 \|\bbeta^*\|_2^2$ is equivalent to $\Var(\vb^{*T} \bTheta \bbeta^*) \geq V_{\min}$, assuming that the matrix $\bSigma \geq \delta > 0$. %Observe that $\|\bSigma\|_{\max} = 1$ by definition.
\end{remark}

Next, we proceed to define $\hat \Delta$ an estimate for $\Delta$. To this end define the following matrices --- $\hat \bTheta^i$:
\begin{align*}
\hat\tau^i_{jk} & := \frac{1}{n-1} \sum_{i' \neq i} \sign \left((X_{ij} - X_{i'j})(X_{ik} - X_{i'k})\right) - \hat \tau_{jk}, \quad \hat \Theta^i_{jk} :=  \pi \cos\left(\frac{\pi}{2}\hat \tau_{jk}\right)\hat\tau^i_{jk}.
\end{align*}
Note that $\hat \Theta^i_{jk}$ is symmetric by definition. Define the estimator $\hat \Delta := n^{-1}\sum_{i = 1}^n (\hat \vb^{T} \hat \bTheta^i \hat \bbeta)^2$. We have the following: 
\begin{corollary} \label{arxivsupp:main:skepticnormality} Assume the data follow from the transelliptical graphical models and Assumption \ref{arxivsupp:main:distr:assump:tgm:clime} holds. Furthermore, assume that the smallest eigenvalue of the correlation matrix satisfies $\lambda_{\min}(\bSigma) \geq \delta > 0$ is bounded away from 0, and $\diag(\bSigma) = 1$. Let $\|\bbeta^*\|_0 = s$ and $\|\vb^*\|_0= s_{\vb}$ and $\hat \Delta$ is a consistent estimate of $\Delta$. Under the following additional assumptions
	\begin{align}
	\max(s_{\vb},s)\|\vb^*\|_1\|\bbeta^*\|_1\log d/n^{1/2} & = o(1), ~~~~   \|\vb^*\|^2_1\|\bbeta^*\|^2_1 \max(s_{\vb},s) \sqrt{\log d/n} = o(1), \\
	& \exists~ k > 2: (s_{\vb} s)^{k}/n^{k - 1} = o(1),  
	\end{align}
	we can take $\lambda \asymp  \|\bbeta^*\|_1 \sqrt{\log d/n}$ and $\lambda' \asymp \|\vb^*\|_1 \sqrt{\log d/n}$ so that the statistic $\hat U_n = \hat \Delta^{-1/2} n^{1/2} (\tilde \theta - \theta)$ satisfies for all $t$: $
	\lim_{n\rightarrow\infty} |\PP(\hat U_n \leq t) - \Phi(t)| = 0.$
\end{corollary}

Similarly, we can show the following result regarding uniformly valid confidence intervals. Before that we define a class of correlation matrices, in analogy to the CLIME case:
\begin{align*}
\tilde{\mathcal{S}}(L, s) = \{\bSigma : \bSigma = \bSigma^T, 0 < \delta \leq \lambda_{\min}(\bSigma),  \diag(\bSigma) = 1, \|\bSigma^{-1}\|_1 \leq L, \max_i \|\bSigma^{-1}_{*i}\|_0 \leq s\},
\end{align*}
\begin{corollary}\label{arxivsupp:main:unifconvTGMCLIME} Let $\bX \sim TE_d(\bmu, \bSigma,\xi)$ with $\bSigma \in \tilde{\mathcal{S}}(L, s)$. Assume furthermore that $\bX$ satisfies Assumption \ref{arxivsupp:main:distr:assump:tgm:clime}, and
	\begin{align} 
	s^3n^{-1/2} = o(1), ~~~ s L^2 \log(d) n^{-1/2} = o(1), ~~~ s L^4\sqrt{\log d/n} = o(1).
	\end{align}
	Then we have $
	\lim_{n \rightarrow \infty} \sup_{\bSigma \in \tilde{\mathcal{S}}(L, s)} \sup_{t} |\PP_{\bbeta}(\hat U_n \leq t ) - \Phi(t)| = 0.$
\end{corollary}
%
%\section{Additional Data Analysis Plot}\label{additional:data:analysis:sec}
%
%\begin{figure}[H]
%\caption{$95\%$ confidence regions for the precision matrix estimated entries of the three possible interactions Mapk13 $\leftrightarrow $ Mapk14, Mapk14 $\leftrightarrow$ Mapkapk2 and Mapkapk2 $\leftrightarrow$ Mapk13 within each of $8$ tissues. EE transelliptical indicates the inference procedure we develop for SKEPTIC and the debias is the procedure developed by \cite{jankova2014confidence}.}
%\centering
%\includegraphics[scale = .75]{FIGS/2ndsetofgenes.pdf}\label{setofgenestwo}
%\end{figure}
%
%\appendix

%\renewcommand{\thesection}{S.\Alph{section}}

%\input{proofs.tex}
\section{Proofs of the General Theory}\label{arxivsupp:main:proofsofgeneraltheory}

Recall that $S(\bbeta) := \vb^{*T}\tb(\Zb, \bbeta)$, and $\hat S(\bbeta) := \hat \vb^T \tb(\Zb, \bbeta)$. We start by deriving an asymptotic expansion of the projected estimating equation $\hat S(\hat \bbeta_{\theta^*})$.

\begin{lemma} \label{arxivsupp:main:mastertheorem} Suppose Assumptions (\ref{consistencyassumpweakcn}), \ref{noiseassumpCI} and 
\begin{align}
n^{1/2}(r_{4}(n)r_{3}(n, \theta^*) + r_{5}(n)r_{1}(n, \theta^*) ) & = o(1), \label{arxivsupp:main:assumpone}
\end{align}
hold. Then we have the following asymptotic expansion:
$$
n^{1/2}\hat S(\hat \bbeta_{\theta^*}) = n^{1/2} S(\bbeta^*) + o_p(1). 
$$
\end{lemma}

\begin{proof}[Proof of Lemma \ref{arxivsupp:main:mastertheorem}] 

Let for brevity $r_1(n) = \sup_{\theta \in \cN_{\theta^*}} r_1(n,\theta)$ and $r_3(n) = \sup_{\theta \in \cN_{\theta^*}} r_3(n,\theta)$. We start by showing that for all $\theta \in \mathcal{N}_{\theta^*}$:
\begin{align}
\hat S(\hat \bbeta_{\theta}) = S(\bbeta^*_{\theta}) + o_p(1) = \vb^{*T}[E_{\tb}( \bbeta^*_{\theta})] + o_p(1). \label{arxivsupp:main:LLN}
\end{align}
This fact will be used in the proof of Theorem \ref{consistency:sol}. By the mean value theorem we have:
\begin{align*}
\hat S(\hat \bbeta_{\theta}) & = \vb^{*T}\tb(\Zb, \bbeta^*_{\theta}) + \underbrace{\hat \vb^{T} \Tb(\Zb, \tilde \bbeta_\nu) (\hat \bbeta_{\theta} - \bbeta^{*}_{\theta})}_{I_1} + \underbrace{(\hat \vb - \vb^*)^T \tb(\Zb, \bbeta^*_{\theta})}_{I_2} 
\end{align*}
Next we control $I_1$:
\begin{align}
|I_1| & \leq \| [\hat \vb^{T} \Tb(\Zb, \tilde \bbeta_\nu) ]_{-1} \|_{\infty} \|\hat \bbeta_{\theta} - \bbeta^{*}_{\theta} \|_{1} \leq O_p(r_3(n) + \underbrace{\|\vb^{*T}\left[E_{\Tb}(\bbeta^*_{\theta})\right]_{-1}\|_{\infty} }_{O(1)}) O_p(r_4(n)) \label{arxivsupp:main:masterI1bound}
\end{align}
where by $[\cdot]_{-1}$ we mean discarding the first entry (corresponding to $\theta$) of the vector. We proceed to bound $I_2$:
\begin{align}
| I_2| \leq \|\hat \vb - \vb^*\|_1 \| \tb(\Zb, \bbeta^*_{\theta})\|_{\infty}  =  O_p(r_5(n))O_p(r_1(n) + \underbrace{\|E_{\tb}(\bbeta^*_{\theta})\|_{\infty}}_{O(1)}). \label{arxivsupp:main:masterI2bound}
\end{align}
This combined with (\ref{betastartassumpCIvstar}) concludes the our initial statement. For the influence function expansion stated in Lemma \ref{arxivsupp:main:mastertheorem} observe that $E_{\tb}(\bbeta^*_{\theta^*}) = 0$ and $ \vb^{*T}\left[E_{\Tb}(\bbeta^*_{\theta^*})\right]_{-1} = 0$ and hence using (\ref{arxivsupp:main:assumpone}) we can modify bounds (\ref{arxivsupp:main:masterI1bound}) and (\ref{arxivsupp:main:masterI2bound}) to: 
\begin{align*}n^{1/2}(|I_1| + |I_2|) \leq n^{1/2}O_p(r_4(n)r_3(n, \theta^*) + r_5(n)r_1(n, \theta^*)) = o_p(1),\end{align*}
and we are done. 
\end{proof}

\begin{proof}[Proof of Theorem \ref{consistency:sol}] First assume that the map $\theta \mapsto \hat S(\hat \bbeta_\theta)$ is continuous and has a unique $0$. Take $\epsilon > 0$ so that both $\theta^* - \epsilon, \theta^* + \epsilon \in \mathcal{N}_{\theta^*}$. Without loss of generality let $\vb^{*T}E_{\tb}(\bbeta_{\theta^* - \epsilon}) < 0$ and $\vb^{*T}E_{\tb}(\bbeta_{\theta^* + \epsilon}) > 0$  for all $\epsilon > 0$. Note that by the continuity of $\theta \mapsto \hat S(\hat \bbeta_\theta)$, and the fact that $\tilde \theta$ is the unique root $
\PP^*(\hat S(\hat \bbeta_{\theta^* - \epsilon}) <  0, \hat S(\hat \bbeta_{\theta^* + \epsilon}) > 0) \leq \PP^*(\theta^* - \epsilon < \tilde \theta < \theta^* + \epsilon).$
Now by (\ref{arxivsupp:main:LLN}) from Lemma \ref{arxivsupp:main:mastertheorem} we have that the left hand side converges to $1$ and this concludes the proof in the first case. The same argument goes through in the second case.
\end{proof}

\begin{proof}[Proof of Theorem \ref{weakconvver2:CI}] Let $U_n = \frac{n^{1/2}}{\sqrt{\vb^{*T}\bSigma\vb^*}}(\tilde \theta - \theta^*)$. It suffices to show that the statement holds for $U_n$, as the statement for $\hat U_n$ is a corollary after an application of Slutsky's theorem. By the mean value theorem:
\begin{align}
\hat S(\hat \bbeta_{\tilde \theta}) = \hat S(\hat \bbeta_{\theta^*}) + \hat \vb^{T} [\Tb(\Zb, \hat \bbeta_{\theta^*})]_{*1} (\tilde \theta - \theta^*) + \frac{1}{2}\hat \vb^{T} \frac{\partial}{\partial \theta}[\Tb(\Zb, \tilde \bbeta_{\nu})]_{*1} (\tilde \theta - \theta^*)^2, \label{arxivsupp:main:meanvaluethm}
\end{align}
where $\tilde \bbeta_{\nu} = \nu \hat \bbeta_{\tilde \theta} + (1 - \nu) \hat \bbeta_{\theta^*}$ for some $\nu \in [0,1]$. By (\ref{stabtwo}) and the fact that $\hat \vb$ and $\tilde \bbeta_\nu$ are consistent (by (\ref{consistencyassumpweakcn}) and Theorem \ref{consistency:sol}) we have that:
$$
 \bigg | \frac{1}{2} \hat \vb^{T} \frac{\partial}{\partial \theta}[\Tb(\Zb, \tilde \bbeta_{\nu})]_{*1} (\tilde \theta - \theta^*)^2 \bigg| \leq (\tilde \theta - \theta^*)^2 O_p(1) = |\tilde \theta - \theta^*|o_p(1).
$$

Observe that by Lemma \ref{arxivsupp:main:mastertheorem} and Assumption \ref{CLTcond} we have $\frac{n^{1/2}}{\sqrt{\vb^{*T}\bSigma\vb^*}} \hat S(\hat \bbeta_{\theta^*}) = O_p(1)$, and more precisely $\frac{n^{1/2}}{\sqrt{\vb^{*T}\bSigma\vb^*}} \hat S(\hat \bbeta_{\theta^*}) \rightsquigarrow N(0,1)$. Next by Assumption \ref{noiseassumpCI},
$$n^{1/2} |\hat \theta - \theta^*|(\underbrace{|\hat \vb^{T} [\Tb(\Zb, \hat\bbeta_{\theta^*})]_{*1} |}_{1 + o_p(1)} + o_p(1)) = O_p(1),$$
and hence we conclude that $\hat \theta - \theta^* = O_p(n^{-1/2})$. Thus by Assumption \ref{noiseassumpCI} and Sltusky's:
$$
\frac{n^{1/2} (\tilde \theta - \theta^*)}{\sqrt{\vb^{*T}\bSigma\vb^*}} = - \frac{n^{1/2} \hat S(\hat \bbeta_{\theta^*})}{\sqrt{\vb^{*T}\bSigma\vb^*} \hat \vb^{T} [\Tb(\Zb, \hat \bbeta_{\theta^*})]_{*1}} + o_p(1) \rightsquigarrow N(0,1),
$$
which concludes the proof.
\end{proof}

\begin{proof}[Proof of Proposition \ref{arxivsupp:main:Uconsistency:sol}] We start by showing that for all $\epsilon > 0$:
$$
\sup_{\bbeta \in \bOmega} \sup_{\check \theta \in \mathcal{N}_\theta} \PP_{\bbeta}\left(|\hat S(\hat \bbeta_{\check \theta}) - \vb^{T}E_{\tb}(\bbeta_{\check \theta})| > \epsilon \right) = o(1).
$$
We have:
\begin{align}\label{arxivsupp:main:unif:conv:eq}
\MoveEqLeft \sup_{\bbeta \in \bOmega} \sup_{\check \theta \in \mathcal{N}_\theta} \PP_{\bbeta}\left(|\hat S(\hat \bbeta_{\check \theta}) - \vb^{T}E_{\tb}(\bbeta_{\check \theta})| > \epsilon \right) \nonumber \\
& \leq  \sup_{\bbeta \in \bOmega} \sup_{\check \theta \in \mathcal{N}_\theta} \PP_{\bbeta}\bigg(\bigg|\vb^{T}\tb(\Zb, \bbeta_{\check \theta}) - \vb^{T}E_{\tb}(\bbeta_{\check \theta})\bigg| > \frac{\epsilon}{3} \bigg) \nonumber \\
&  + \sup_{\bbeta \in \bOmega} \sup_{\check \theta \in \mathcal{N}_\theta}  \PP_{\bbeta}\bigg(\sup_{\nu \in [0,1]}\bigg\|\hat \vb^{T} [\Tb(\Zb, \tilde \bbeta_\nu)]_{-1}\bigg\|_{\infty} \|\hat \bbeta_{\check \theta} - \bbeta_{\check \theta}\|_1 > \frac{\epsilon}{3}\bigg)\nonumber \\
& + \sup_{\bbeta \in \bOmega} \sup_{\check \theta \in \mathcal{N}_\theta} \PP_{\bbeta}\bigg(\|\hat \vb - \vb\|_1 \| \tb(\Zb, \bbeta_{\check \theta})\|_{\infty} > \frac{\epsilon}{3}\bigg),
\end{align} 
where $\tilde \bbeta_\nu = \nu \hat\bbeta_{\check \theta} + (1- \nu)\bbeta_{\check \theta}$. By Assumptions \ref{arxivsupp:main:Uconsistencyassumpweakcn} and \ref{arxivsupp:main:UnoiseassumpCI} it follows that the RHS is $o(1)$, as we claimed.  First let us assume that the maps $\check \theta \mapsto \hat S(\hat \bbeta_{\check \theta})$ are continuous. To shorten the notation in the remaining of the proof we will use use the following notation: If $A, B$ are random variables we write $
\PP(A \gtrless c_1, B \lessgtr c_2) := \PP((A > c_1 \cap B < c_2) \cup (A < c_1 \cap B > c_2)).$
Then following inequality holds $
\inf_{\bbeta \in \bOmega} \PP_{\bbeta} (\hat S(\hat \bbeta_{\theta- \epsilon} ) \lessgtr 0, \hat S(\hat \bbeta_{\theta + \epsilon} ) \gtrless 0 ) \leq \inf_{\bbeta \in \bOmega} \PP_{\bbeta} (\theta - \epsilon < \tilde \theta < \theta + \epsilon).$
Next for small enough $\epsilon > 0$ such that $\theta \pm \epsilon \in \mathcal{N}_{\theta}$, we have:
\begin{align*}
&\inf_{\bbeta \in \bOmega} \PP_{\bbeta} (\vb^T E_{\tb}(\bbeta_{\theta - \epsilon}) \lessgtr 0, \vb^T E_{\tb}(\bbeta_{\theta + \epsilon})\gtrless 0) \\
- & \sup_{\bbeta \in \bOmega} \PP_{\bbeta} (\hat S(\hat \bbeta_{\theta- \epsilon} ) \gtrless 2\vb^T E_{\tb}(\bbeta_{\theta - \epsilon}), \hat S(\hat \bbeta_{\theta + \epsilon}) \lessgtr 2\vb^T E_{\tb}(\bbeta_{\theta + \epsilon}))\\
\leq & \inf_{\bbeta \in \bOmega} \PP_{\bbeta} (\hat S(\hat \bbeta_{\theta- \epsilon} ) \lessgtr 0, \hat S(\hat \bbeta_{\theta + \epsilon} ) \gtrless 0).
\end{align*}
The LHS goes to $1$ by (\ref{arxivsupp:main:unif:conv:eq}), and hence the proof is complete in this case. In the case when $\check \theta \mapsto \hat S(\hat \bbeta_{\check \theta})$ are non-decreasing, exactly the same argument goes through.
\end{proof}

\begin{lemma} \label{arxivsupp:main:normallemma} Let $X_n(\br)$ and $\xi_n(\br)$ be two sequences of random variables, depending on a parameter $\br \in \bR$. Suppose that $\lim_{n} \sup_{\br \in \bR} \sup_{t } |\PP_{\br}(X_n(\br) \leq t) - F(t)| = 0$ where $F$ is a continuous cdf, and $\lim_{n} \inf_{\br \in \bR} \PP_{\br}(|1 - \xi_n(\br)| \leq \tau(n)) = 1$ for $\tau(n) = o(1)$. Assume in addition that $F$ is Lipschitz, i.e. there exist $\kappa > 0$, such that for any $t, s \in \RR: |F(t) - F(s)| \leq \kappa |t - s|$. Then we have:
$$
\lim_{n}  \sup_{\br \in \bR} \sup_{t} \bigg|\PP_{\br}\bigg(\frac{X_n(\br)}{\xi_n(\br)} \leq t\bigg) - F(t)\bigg| = 0.
$$
\end{lemma}

\begin{proof}[Proof of Lemma \ref{arxivsupp:main:normallemma}]  The proof follows by a direct calculation, and we omit the details. %Take $t$, and any $\alpha > 0$. Observe that the following inequality holds:
%\begin{align*} \PP_{\br}\bigg(\frac{X_n(\br)}{\xi_n(\br)} \leq t \bigg) 
%&  \leq \PP_{\br}\bigg(X_n(\br) \leq t + \alpha \bigg) + \PP_{\br}(|X_n(\br)| > \alpha^{-1}) \\
%& + \PP_{\br}(|1 - \xi_n(\br)| > 2\alpha^2) + \PP_{\br}\bigg(|\xi_n(\br)| \leq \frac{1}{2}\bigg).
%\end{align*}
%Take $\alpha = \sqrt{\tau(n)}$, to obtain:
%\begin{align*}
%\PP_{\br}\bigg(\frac{X_n(\br)}{\xi_n(\br)} \leq t \bigg)  - F(t) & \leq \PP_{\br}\bigg(X_n(\br) \leq t + \sqrt{\tau(n)} \bigg) - F(t + \tau(n)) \\
%& + \kappa \sqrt{\tau(n)}  + \PP_{\br}(|X_n(\br)| > \tau(n)^{-1/2}) \\
%& + \PP_{\br}(|1 - \xi_n(\br)| > 2 \tau(n)) + \PP_{\br}\bigg(|\xi_n(\br)| \leq \frac{1}{2}\bigg).\\
%\end{align*}
%Taking $\lim_n \sup_{\br \in \bR} \sup_t$, shows that:
%$$
%\lim_n \sup_{\br \in \bR} \sup_t \PP_{\br}\bigg(\frac{X_n(\br)}{\xi_n(\br)} \leq t \bigg)  - F(t) \leq 0.
%$$
%The reverse inequality, i.e.
%$$
%\lim_n \inf_{\br \in \bR} \inf_t \PP_{\br}\bigg(\frac{X_n(\br)}{\xi_n(\br)} \leq t \bigg)  - F(t) \geq 0,
%$$
%can be established similarly. This completes the proof.
\end{proof}

\begin{lemma} \label{arxivsupp:main:unifconvhelperlemma} Under Assumptions \ref{arxivsupp:main:Uconsistencyassumpweakcn} --- \ref{arxivsupp:main:Ustabilityassump}, we have:
\begin{align}
\lim_{n \rightarrow \infty} \inf_{\bbeta \in \bOmega} \PP_{\bbeta}\left(\left|\hat S(\hat \bbeta_{\theta}) - S( \bbeta) \right| \leq r_{1}(n)r_{5}(n) + r_{2}(n)r_{3}(n) \right) = 1 \label{arxivsupp:main:UBoundShat}.
\end{align}
If in addition $n^{1/2}(r_{1}(n)r_{5}(n) + r_{2}(n)r_{3}(n)) = o(1)$, we have:
\begin{align}
\lim_{n \rightarrow \infty} \sup_{\bbeta \in \bOmega} \sup_t | \PP_{\bbeta}(\sigma^{-1} n^{1/2}\hat S(\hat \bbeta_{\theta})\leq t ) - \Phi(t)| = 0 \label{arxivsupp:main:UCLTShat}.
\end{align}
\end{lemma}

\begin{proof}[Proof of Lemma \ref{arxivsupp:main:unifconvhelperlemma}]
Let $\mathcal{G}^{\bbeta}_i$ be the event inside the probability measures in assumptions \ref{arxivsupp:main:Uconsistencyassumpweakcn} --- \ref{arxivsupp:main:Ustabilityassump} corresponding to the rate $r_i(n)$ for $i = 1,\ldots, 5$. It is clear that $\inf_{\bbeta \in \bOmega} \PP_{\bbeta}(\mathcal{G}^{\bbeta} ) \geq 1 - \sum_{i = 1}^5 \sup_{\bbeta \in \bOmega}\PP_{\bbeta}[ (\mathcal{G}_i^{\bbeta})^c] \rightarrow 1$. Next, the proof of (\ref{arxivsupp:main:UBoundShat}) can be done through the same argument as in the proof of Theorem \ref{arxivsupp:main:mastertheorem}, but using the uniform convergence assumptions. Note that the bounds (\ref{arxivsupp:main:masterI1bound}) and (\ref{arxivsupp:main:masterI2bound}) continue to hold on the event $\mathcal{G}^{\bbeta} = \mathcal{G}_1^{\bbeta} \cap \ldots \cap \mathcal{G}_5^{\bbeta}$. Hence by Assumptions \ref{arxivsupp:main:Uconsistencyassumpweakcn} and \ref{arxivsupp:main:UnoiseassumpCI} the proof of (\ref{arxivsupp:main:UBoundShat}) is complete.

Next we show (\ref{arxivsupp:main:UCLTShat}). Let $\kappa(n) = n^{1/2}C^{-1/2}(r_{1}(n)r_{5}(n) + r_{2}(n)r_{3}(n))$, where we recall the definition of $C$: $\inf_{\bbeta \in \bOmega} \vb^T \bSigma \vb \geq C > 0$. Then we have:
\begin{align*}
\PP_{\bbeta}(\sigma^{-1} n^{1/2}\hat S(\hat \bbeta_{ \theta}) \leq t ) & \leq \PP_{\bbeta}(\sigma^{-1}n^{1/2}\hat S(\hat \bbeta_{ \theta}) \leq t, \mathcal{G}^{\bbeta} ) + \PP_{\bbeta}( (\mathcal{G}^{\bbeta})^c)\\
& \leq \PP_{\bbeta}(\sigma^{-1}n^{1/2} S(\bbeta) \leq t + \kappa(n)) + \PP_{\bbeta}( (\mathcal{G}^{\bbeta})^c).
\end{align*}
The above implies the following inequality:
\begin{align*}
& \PP_{\bbeta}(\sigma^{-1} n^{1/2}\hat S(\hat \bbeta_{ \theta})\leq t) - \Phi(t) \\
& \leq \PP_{\bbeta}(\sigma^{-1} n^{1/2}S(\bbeta)  \leq t + \kappa(n)) - \Phi(t + \kappa(n)) +( \Phi(t + \kappa(n))- \Phi(t) ) + \PP_{\bbeta}( (\mathcal{G}^{\bbeta})^c)\\
&\leq  \PP_{\bbeta}(\sigma^{-1}n^{1/2}S(\bbeta)\leq t + \kappa(n)) - \Phi(t + \kappa(n)) + \frac{\kappa(n)}{\sqrt{2\pi}}+ \PP_{\bbeta}( (\mathcal{G}^{\bbeta})^c),
\end{align*}
where we took into account the fact that $\Phi$ is Lipschitz with constant $\leq \frac{1}{\sqrt{2\pi}}$. Now using Assumption \ref{arxivsupp:main:UCLTassump}, the fact that $\kappa(n) = o(1)$ and  $\PP_{\bbeta}( (\mathcal{G}^{\bbeta})^c) = o(1)$ we conclude that:
$$
\limsup_{n \rightarrow \infty} \sup_{\bbeta \in \bOmega} \sup_t  \PP_{\bbeta}(\sigma^{-1}n^{1/2}\hat S(\hat \bbeta_{ \theta}) \leq t) - \Phi(t) \leq 0 .
$$
With a similar argument one can show the converse, namely:
$$
\liminf_{n \rightarrow \infty} \inf_{\bbeta \in \bOmega} \inf_t  \PP_{\bbeta}(\sigma^{-1}n^{1/2}\hat S(\hat \bbeta_{ \theta})\leq t) - \Phi(t) \geq 0.
$$
This concludes the proof of  (\ref{arxivsupp:main:UCLTShat}).
\end{proof}

\begin{proof}[Proof of Theorem \ref{arxivsupp:main:unifweakconv}] 
By Assumption \ref{arxivsupp:main:UVCassump2}, we have:
\begin{align}
%\inf_{\bbeta} \PP_{\bbeta}\left(  \hat \sigma \geq \sqrt{C} - \frac{r_6(n)}{\sqrt{C}} \right)  & \geq \inf_{\bbeta} \PP_{\bbeta}\left(|\hat \sigma -  \sigma| \leq \frac{r_6(n)}{\sqrt{C}}\right) \geq \inf_{\bbeta} \PP_{\bbeta}\left(|\hat \sigma^2 -  \sigma^{2}| \leq r_6(n)\right) \\
\label{arxivsupp:main:sigmasigmabound}
\inf_{\bbeta} \PP_{\bbeta}\bigg(\bigg|1 - \hat \sigma \sigma^{-1}\bigg| \leq \frac{r_6(n)}{C}\bigg) \geq \PP_{\bbeta}\left(|\hat \sigma -  \sigma| \leq \frac{r_6(n)}{\sqrt{C}}\right) &\geq \inf_{\bbeta} \PP_{\bbeta}\left(|\hat \sigma^2 -  \sigma^{2}| \leq r_6(n)\right) \nonumber \\
& = 1 - o(1),
\end{align}
where recall that $\inf_{\bbeta \in \bOmega} \vb^T \bSigma \vb \geq C > 0$. The last expression implies that:
%\begin{align} \label{arxivsupp:main:sigmasigmabound}
%\inf_{\bbeta} \PP_{\bbeta}\bigg(\bigg|1 - \hat \sigma \sigma^{-1}\bigg| \leq \frac{r_6(n)}{C}\bigg) \geq \inf_{\bbeta} \PP_{\bbeta}\bigg(|\hat \sigma -  \sigma| \leq \frac{r_6(n)}{\sqrt{C}}\bigg) = 1 - o(1).
%\end{align}
Next define:
%\begin{align*}
%\zeta_n(\bbeta) & =  \hat \sigma^{-1} \sigma \bigg(\hat \vb^{T} [\Tb(\Zb, \tilde \bbeta_{\theta})]_{*1}\bigg), ~~~~ \eta_n(\bbeta, \nu) =  \frac{\hat \sigma^{-1} \sigma}{2}\hat \vb^{T} \frac{\partial}{\partial \theta}[\Tb(\Zb, \tilde \bbeta_{\nu})]_{*1} (\tilde \theta - \theta),
%\end{align*}
\begin{align*}
\zeta_n(\bbeta) & =  \hat \sigma\sigma^{-1}  \bigg(\hat \vb^{T} [\Tb(\Zb, \tilde \bbeta_{\theta})]_{*1}\bigg), ~~~~ \eta_n(\bbeta, \nu) =  \frac{\hat \sigma \sigma^{-1}}{2}\hat \vb^{T} \frac{\partial}{\partial \theta}[\Tb(\Zb, \tilde \bbeta_{\nu})]_{*1} (\tilde \theta - \theta),
\end{align*}
where $\tilde \bbeta_{\nu} = \nu \hat \bbeta_{\tilde \theta} + (1 - \nu) \hat \bbeta_{\theta}$ and let $\xi_n(\bbeta, \nu) = \zeta_n(\bbeta) + \eta_n(\bbeta, \nu)$. 

We will show that 
\begin{align} \label{arxivsupp:main:xi:close:one}
\lim_{n} \inf_{\bbeta \in \bOmega} \inf_{\nu \in[0,1]} \PP_{\bbeta}(|1 - \xi_n(\bbeta, \nu)| \leq \tau(n)) = 1,
\end{align}
for some $\tau(n) = o(1)$, or equivalently --- for every $\epsilon > 0:$ $\sup_{\bbeta \in \bOmega} \sup_{\nu \in[0,1]} \PP_{\bbeta}(|1 - \xi_n(\bbeta, \nu)| > \epsilon) = o(1)$. We proceed with the following:
\begin{align*}
\sup_{\bbeta \in \bOmega} \sup_{\nu \in[0,1]}\PP_{\bbeta}(|1 - \xi_n(\bbeta, \nu)| > \epsilon) \leq \underbrace{\sup_{\bbeta \in \bOmega} \PP_{\bbeta}(|1 - \zeta_n(\bbeta)| > \epsilon/2)}_{I_1} + \underbrace{\sup_{\bbeta \in \bOmega} \sup_{\nu \in[0,1]}\PP_{\bbeta}(|\eta_n(\bbeta, \nu)| > \epsilon/2)}_{I_2}.
\end{align*}
First we tackle $I_1$. We have:
\begin{align*}
I_1 &\leq \sup_{\bbeta \in \bOmega} \PP_{\bbeta} (|1 - \hat \sigma \sigma^{-1}| > \epsilon/4 ) + \sup_{\bbeta \in \bOmega} \PP_{\bbeta}(| \hat \sigma \sigma^{-1}| > 2 )  + \sup_{\bbeta \in \bOmega} \PP_{\bbeta} (| 1 - \hat \vb^{T} [\Tb(\Zb, \tilde \bbeta_{\theta})]_{*1}| > \epsilon/8),
\end{align*}
and all of the terms on the RHS are $o(1)$ due to (\ref{arxivsupp:main:sigmasigmabound}) and Assumption \ref{arxivsupp:main:Ustabilityassump} respectively.

Next we handle $I_2$ term. Let $E = \sup_{\bbeta \in \bOmega} \EE_{\bbeta}\psi(\Zb)$, where the function $\psi$ is defined in Assumption \ref{arxivsupp:main:Ustabilityassump}. Fix an $\alpha > 0$, and proceed as:
\begin{align*}
I_2 & \leq \sup_{\bbeta \in \bOmega} \PP_{\bbeta} (|\hat \sigma \sigma^{-1}| \geq 2) + \sup_{\bbeta \in \bOmega} \PP_{\bbeta} (|\tilde \theta - \theta^*| > E^{-1}\epsilon \alpha/2) + \sup_{\bbeta \in \bOmega}\PP_{\bbeta}(\psi(\Zb) \geq \alpha^{-1}E) \leq o(1) + \alpha,
\end{align*}
where the last inequality follows from (\ref{arxivsupp:main:sigmasigmabound}), Proposition \ref{arxivsupp:main:Uconsistency:sol} and Markov's inequality. Taking the limit $n \rightarrow \infty$ shows that $\lim_{n} \sup_{\bbeta \in \bOmega} \sup_{\nu \in[0,1]}\PP_{\bbeta}(|\eta_n(\bbeta, \nu)| > \epsilon/2)  \leq \alpha$. Taking $\alpha \rightarrow 0$ shows (\ref{arxivsupp:main:xi:close:one}). Next observe that by (\ref{arxivsupp:main:meanvaluethm}) we have the following identity:
\begin{align}
\frac{n^{1/2} (\tilde \theta - \theta)}{\hat \sigma} = -\frac{n^{1/2} \hat S(\hat \bbeta_{\theta})}{\sigma \xi_n(\bbeta, \nu)}. \label{arxivsupp:main:final:identity}
\end{align}
Now combining the fact that (\ref{arxivsupp:main:xi:close:one}) with our results from Lemmas \ref{arxivsupp:main:normallemma} and \ref{arxivsupp:main:unifconvhelperlemma} in addition to (\ref{arxivsupp:main:final:identity}) completes the proof.
\end{proof}

\begin{proof}[Proof of Proposition \ref{arxivsupp:weakconv}] Note that the only thing left to show is the consistency of the plugin estimate $\hat \vb^{T}\hat \bSigma \hat \vb$ for $\vb^{*T} \bSigma \vb^{*}$, with the rest of the argument following from Theorem \ref{weakconvver2:CI} and Slutsky's theorem. By the triangle inequality we have:
\begin{align*}
|\hat \vb^{T}\hat \bSigma \hat \vb - \vb^{*T} \bSigma \vb^{*}| & \leq \underbrace{\|\hat \vb^{T} - \vb^*\|_1 \|\hat \bSigma (\hat \vb - \vb^{*})\|_{\infty}}_{I_1} + 2 \underbrace{\|\vb^{*T} \hat \bSigma\|_\infty\| \hat \vb - \vb^* \|_{1}}_{I_2} + \underbrace{\|\vb^*\|^2_1 \|\hat \bSigma - \bSigma\|_{\max}}_{I_3}
\end{align*}
Next we control $
|I_1| \leq O_p(r_2(n)^2 r_7(n) + \|\bSigma\|_{\max}r^2_2(n)) = o_p(1).$
Next 
$$|I_2| \leq 2 O_p(\|\vb^*\|_1 r_2(n)r_7(n) + \|\vb^{*T} \bSigma\|_\infty r_2(n) ) = o_p(1).$$
Finally, for $I_3$ we have $
|I_3| \leq \|\vb^*\|_1^2 O_p(r_7(n)) = o_p(1).$
\end{proof}

\section{Proofs for the Dantzig Selector} \label{arxivsupp:main:DantzigProofs}

We recall the definition of \textit{restricted eigenvalue} (RE) assumption \citep{bickel2009simultaneous}.

\begin{definition}[RE]\label{arxivsupp:RE} We say that the symmetric positive semi-definite matrix $\Mb_{k\times k}$ possesses the restricted eigenvalue property if:
	$$\operatorname{RE}_\Mb(s, \xi) = \min_{S \subset \{1,\ldots,k\}, |S| \leq s} \min \left\{\frac{\ub^T \Mb \ub}{\|\ub_{S}\|^2_2} : \ub \in \mathbb{R}^d\setminus\{0\}, \|\ub_{S^c}\|_1 \leq \xi \|\ub_{S}\|_1 \right\} > 0.$$
\end{definition}

\begin{definition} 
Denote with $\Xb$ the $n \times d$ matrix whose rows are the $\bX_i^T$ vectors stacked together. Let $\bY$ be the an $n \times 1$ vector stacking the observations $Y_i$ for $i = 1,\ldots,n$ and let $\bvarepsilon = \bY - \Xb\bbeta^*$.
\end{definition}

\begin{proof}[Proof of Corollary \ref{normdantzigfinal}] We will prove this result by validating all of the assumptions required of the general Theorem \ref{weakconvver2:CI}. Regarding Assumption (\ref{consistencyassumpweakcn}), observe that from Lemma \ref{arxivsupp:main:vdiff}, we have that $\|\vb^*-\hat\vb\|_1 = O_p\left( \|\vb^*\|_1s_{\vb}\sqrt{\log d/n}\right)$ and by Lemma \ref{arxivsupp:main:l1betadiff} --- $\|\bbeta^*-\hat\bbeta\|_1 = O_p\left(s\sqrt{\log d/n}\right)$. 

Assumption \ref{noiseassumpCI} can be verified as follows. To see (\ref{betastartassumpCI}), fix a $|\theta - \theta^*| < \epsilon$, for some $\epsilon > 0$. Next by the triangle inequality:
$$
\bigg\|\bSigma_n\bbeta^*_{\theta} - n^{-1}\Xb^T \bY - \bSigma_{\bX}(\bbeta^*_{\theta} - \bbeta^*)\bigg\|_{\infty} \leq \|n^{-1} \Xb^T\bvarepsilon \|_{\infty} + \|\bSigma_{n,*1} - \bSigma_{\bX, *1}\|_{\infty}\epsilon.
$$
The two terms on the RHS are $O_p(\sqrt{\log d/n})$, by the proof of Lemma \ref{arxivsupp:main:lemmabetaclose} (see \ref{arxivsupp:main:condprobbound}) and Lemma \ref{arxivsupp:main:samplecovpopcov}. The same logic shows that $|\vb^{*T}\bSigma_n\bbeta^*_{\theta} - n^{-1}\vb^{*T}\Xb^T \bY - \vb^{*T}\bSigma_{\Xb,*1}(\theta - \theta^*) | = O_p(\|\vb^*\|_{1}\sqrt{\log d/n})$, which implies (\ref{betastartassumpCIvstar}). Since the Hessian $\Tb$ in (\ref{lambdaprimeasumpCI}) is free of $\bbeta$ we are allowed to set $r_3(n) = \lambda' \asymp \|\vb^*\|_1\sqrt{\log d/n} = o(1)$ (by Lemma \ref{arxivsupp:main:vdiff}). Finally the two expectations in Assumption \ref{noiseassumpCI}, are bounded as we see below:
\begin{align*}
\|\bSigma_{\bX}(\bbeta^*_{\theta} - \bbeta^*)\|_{\infty} \leq \|\bSigma_{\bX,*1} \|_{\infty}\epsilon \leq 2 K_{\bX}^2 \epsilon, ~~~~~~~~ \|\vb^{*T}\bSigma_{\bX, -1}\|_{\infty}  = 0.
\end{align*}
By adding up the following two identities:
\begin{align*}
\sqrt{n} O_p\left(\|\vb^*\|_1 \sqrt{\log d/n}\right)O_p\left(s\sqrt{\log d/n}\right) & = o_p(1),\\
\sqrt{n}O_p\left( \|\vb^*\|_1s_{\vb}\sqrt{\log d/n}\right) O_p\left(\sqrt{\log d/n}\right) = O_p\left(\|\vb^*\|_1s_{\vb}\log d/\sqrt{n}\right) & = o_p(1),
\end{align*}
we get that (\ref{assumpone}) is also valid in this case.
 
To verify the consistency of $\tilde \theta$ we check the assumptions in Theorem \ref{consistency:sol}. Clearly the map $\theta \mapsto \vb^{*T}\bSigma_{\bX} (\bbeta^*_{\theta} - \bbeta^*) = (\theta - \theta^*)$ has a unique $0$ when $\theta = \theta^*$. Moreover, the map $\theta \mapsto \hat \vb^T(\bSigma_{n} \hat\bbeta_{\theta} - n^{-1}\Xb^T \bY)$ is continuous as it is linear. In addition, it has a unique zero except in cases when $\hat \vb^T \bSigma_{n,*1} = 0$. However note that $|\hat \vb^T \bSigma_{n,*1} - 1| \leq \lambda'$ by (\ref{lambda:prime:dantz}), and hence for small enough values of $\lambda'$ a unique zero will always exist.

Assumption \ref{CLTcond} is verified in Lemma \ref{arxivsupp:main:CLTDantzig}. We move on to show (\ref{stabtwo}). Clearly (\ref{stabtwo}) is trivial as its LHS $\equiv 0$ in this case. Finally Proposition \ref{arxivsupp:main:pluginDantzig} checks that $\hat \Delta$ is a consistent estimate of $\Delta$. This completes the proof.
\end{proof}

\begin{remark}\label{arxivsupp:main:unif:CI:danztig} In fact, under the additional assumption $s_{\vb}^3/n = o(1)$, the proof of Corollary \ref{normdantzigfinal} implies that the uniform types of assumptions in Section \ref{arxivsupp:main:UWCUNsec} are satisfied, and hence under the same assumptions as in Corollary \ref{normdantzigfinal}, we have:
$$
\lim_{n \rightarrow \infty} \sup_{\|\bbeta\|_0 \leq s} \sup_{t\in\RR} |\PP_{\bbeta}(\hat U_n \leq t) - \Phi(t))| = 0. 
$$
By Remark \ref{unifweakconv:CI} the above equality readily translates from estimator uniform consistency to confidence region uniform consistency. It is noteworthy to mention that the space of uniformity is $\bOmega = \{\bbeta: \|\bbeta\|_0 \leq s\}$, provided that the conditions of Corollary \ref{normdantzigfinal} hold, which includes that the covariates $\bX$ satisfy $s_{\vb} = \|\vb^*\|_0$ and $\max(s_{\vb}, s) \|\vb^*\|_1\log d/\sqrt{n} = o(1), ~~~  \|\vb^*\|_1^2\sqrt{\log d/n} = o(1)$.
\end{remark}

\begin{remark}\label{arxivsupp:v:lower:bound} Note here that it is implied that $\lambda' = o(1)$ and hence since $\|\vb^*\|_1 \geq 2 K_{\bX}^{-2}$ it follows that $\lambda = o(1)$ as well.
\end{remark}

\begin{remark} Observe that $\|\vb^*\|_1 \leq \sqrt{s_{\vb}}\|\vb^*\|_2 \leq  \sqrt{s_{\vb}} \delta$. This yields sufficient conditions by substituting $\|\vb^*\|_1$ with $\sqrt{s_{\vb}}$. Moreover, under the assumption that $\vb^{*T}\bX$ is sub-Gaussian, we can further relax the requirements on sparsity $s_{\vb}$ dimension $d$ and number of observations $n$.
\end{remark}

\begin{lemma} \label{arxivsupp:main:CLTDantzig} Assume that condition \ref{subGausserr:design:bounded:mom:matr} holds and $\max(s_{\vb}, s) \|\vb^*\|_1\log d/\sqrt{n} = o(1)$.
Then $
\Delta^{-1/2} n^{1/2}  S(\bbeta^*) \rightsquigarrow N(0,1).$
\end{lemma}

\begin{proof}[Proof of Lemma \ref{arxivsupp:main:CLTDantzig}]
To show the weak convergence we verify Lyapunov's condition for the CLT. We need to show that $
 \frac{n^{-2}}{\Delta^{2}} \sum_{i=1}^n \EE \left|\vb^{*T}\bX_i(\bX_i^T\bbeta^* - Y_i)\right|^4$ converges to $0$.
Note that we have $\Delta^{2} \geq \lambda_{\min}(\bSigma_{\bX})\|\vb^{*}\|_2^4 \Var(\varepsilon)^{2} = O(1) \|\vb^{*}\|_2^4$. Therefore it suffices to consider the following expression:
\begin{align}
 \frac{n^{-2}}{\|\vb^{*}\|_2^4}  \sum_{i=1}^n \EE \left|\vb^{*T}\bX_i(\bX_{i}^T\bbeta^* - Y_i)\right|^4 & \leq  n^{-2}  \sum_{i=1}^n \EE \left\|(\bX_i\varepsilon_i)_{S_{\vb}}\right\|_2^4 \leq n^{-1} s^{2}_{\vb} M, \label{arxivsupp:main:LyapunovBound}
\end{align}
where $M =  2^8 (K K_{\bX})^4$, and the last inequality holding from Lemma \ref{arxivsupp:main:Xepsbound}. Using the fact that $\|\vb^*\|_1 \geq 2 K_{\bX}^{-2}$ as seen in Remark \ref{arxivsupp:v:lower:bound}, we have that $\max(s_{\vb}, s) \|\vb^*\|_1\log d/\sqrt{n} = o(1)$ implies $s^2_{\vb}/n = o(1)$ and completes the proof.
\end{proof}

\begin{remark} Using the Berry-Esseen theorem for non-identical random variables in combination with Lemma \ref{arxivsupp:main:Xepsbound} we can further show:
$$
\sup_t \left|\PP^* \left(\frac{n^{1/2}}{\sqrt{\Delta}} S(\bbeta^*) \leq t\right) - \Phi(t)\right| \leq C_{BE} (6 K K_{\Xb})^3 n^{-1/2} s^{3/2}_{\vb} = o(1),
$$
where $M$ and $C_{BE}$ are absolute constants.
\end{remark}

\begin{proposition} \label{arxivsupp:main:pluginDantzig} 
Under Assumption \ref{subGausserr:design:bounded:mom:matr}, and the following additional assumption $
 \|\vb^*\|_1^2\sqrt{\frac{\log d}{n}} = o(1),$
 we have that $\hat \Delta \rightarrow_p \Delta$.
\end{proposition}

\begin{proof}[Proof of Proposition \ref{arxivsupp:main:pluginDantzig}]
We show that each of the two sums is corresponding to its population counterpart, and then the proof follows upon an application of Slutsky's theorem. We start with the first term:
\begin{align*}
\bigg|\frac{1}{n}\sum_{i = 1}^n (\hat \vb^T \bX_i)^2 -  \vb^{*T}\bSigma_{\bX} \vb^*\bigg| &\leq \bigg|\underbrace{\frac{1}{n}\sum_{i = 1}^n [(\hat \vb^T \bX_i)^2 -  ( \vb^{*T} \bX_i)^2 ]}_{I_1} \bigg| + |\underbrace{\vb^{*T}\bSigma_n \vb^* - \vb^{*T}\bSigma_{\bX} \vb^*}_{I_2}|,\\
|I_1| & \leq \|\hat \vb - \vb^*\|_1( \|\bSigma_n \hat \vb \|_{\infty} + \|\bSigma_n \vb^*\|_{\infty}).
\end{align*}
We know from Lemma \ref{arxivsupp:main:vdiff}, that $\|\vb^*-\hat\vb\|_1 = O_p\left( \|\vb^*\|_1s_{\vb}\sqrt{\log d/n}\right)$, and by definition $\|\bSigma_n \hat \vb \|_{\infty} \leq 1 + \lambda'$. In the proof of Lemma \ref{arxivsupp:main:vdiff} we also show that
$$ \|\bSigma_n \vb^* \|_{\infty} = 1 + O_p\left( \|\vb^*\|_1\sqrt{\log d/n}\right),$$ 
upon appropriately choosing $\lambda' \asymp \|\vb^*\|_1 \sqrt{\log d/n}$, with a large enough proportionality constant. Thus since $O_p\left( \|\vb^*\|_1s_{\vb}\sqrt{\log d/n}\right)\left(2 + O_p\left( \|\vb^*\|_1\sqrt{\log d/n}\right)\right) = o_p(1)$ we have shown $|I_1| = o_p(1)$. We next tackle $
|I_2| \leq \|\vb^*\|_1^2 \|\bSigma_n - \bSigma_{\bX}\|_{\max}.$
Lemma \ref{arxivsupp:main:samplecovpopcov} gives us that $ \|\bSigma_n - \bSigma_{\bX}\|_{\max} = O_p(\sqrt{\log d/n})$, and thus because of our extra assumption we have $|I_2| = O_p( \|\vb^*\|_1^2\sqrt{\log d/n}) = o_p(1)$.

Now we turn to the second part of the proof:
\begin{align*}
\MoveEqLeft \bigg|\frac{1}{n} \sum_{i = 1}^n (Y_i - \bX_i^T \hat \bbeta)^2 - \Var(\varepsilon)\bigg| \\
& \leq \bigg|\underbrace{\frac{1}{n} \sum_{i = 1}^n (Y_i - \bX_i^T \hat \bbeta)^2 - \frac{1}{n} \sum_{i = 1}^n (Y_i - \bX_i^T \bbeta^*)^2 }_{I_3}\bigg| + \bigg|\underbrace{\frac{1}{n} \sum_{i = 1}^n \varepsilon_i^2 - \Var(\varepsilon)}_{I_4}\bigg|.
\end{align*}
The term $I_4$ is clearly $o_p(1)$ because of the LLN ($\varepsilon_i$ are centered and have finite variance as sub-Gaussian random variables). Thus we are left to deal with $I_3$:
\begin{align*}
|I_3| & \leq \frac{1}{n}\|\Xb(\hat \bbeta - \bbeta^*)\|_2^2 +  \frac{2}{n} \sum_{i = 1}^n |\bX_i^T (\hat \bbeta - \bbeta^*)| |\varepsilon_i| \\
& \leq   \frac{1}{n}\|\Xb(\hat \bbeta - \bbeta^*)\|_2^2  +  \frac{2}{n} \|\Xb(\hat \bbeta - \bbeta^*)\|_2\sqrt{ \sum_{i = 1}^n \varepsilon_i^2},
\end{align*}
where $\Xb_{n\times d}$ is a matrix, with rows $\bX_i^T$ stacked together. (\ref{arxivsupp:main:XbetaL2}) in Lemma \ref{arxivsupp:main:l1betadiff} gives us that $ \frac{1}{n}\|\Xb(\hat \bbeta - \bbeta^*)\|_2^2 = O_p(\frac{s \log d}{n}) = o_p(1)$, and since by LLN $\sqrt{  \frac{1}{n}\sum_{i = 1}^n \varepsilon_i^2} = O_p(1)$, we have $|I_3| = o_p(1)$, which shows the consistency of the second estimator and concludes the proof.
\end{proof}

\begin{lemma} \label{arxivsupp:main:samplecovpopcov} We have that with probability at least $1 - 2 d^{2 - \bar c A_X^2}$:
$$\left\|\bSigma_n - \bSigma_{\bX}\right\|_{\max} \leq 4 A_X K_{\bX}^2 \sqrt{\log d/n}.$$
{\bf Note.} The constant $\bar c$ is a universal constant independent of the $\bX$ distribution, $K_{\bX}$ is as defined in the main text, and $A_X > 0$ is an arbitrarily chosen constant satisfying $A_X\sqrt{\log d/n} \leq 1$.
\end{lemma}

\begin{proof}[Proof of Lemma \ref{arxivsupp:main:samplecovpopcov}]
First we note that the elements of the matrix -- $\bX^{\otimes 2}$ are sub-exponential random variables. This fact can be seen since by Cauchy-Schwartz, one can easily obtain that:
\begin{align}
\|X_iX_j\|_{\psi_1} \leq  2\|X_i\|_{\psi_2} \|X_j\|_{\psi_2} \leq 2K^2_{\bX}. \label{arxivsupp:main:subexpo}
\end{align}
Next by the triangle inequality it is clear that $\|X_i X_j - \EE X_i X_j\|_{\psi_1} \leq  4K^2_{\bX}.$ The proof is completed by applying a standard Bernstein tail bound (see Proposition 5.16 in \cite{vershynin2010introduction} e.g.).
%Using a Bernstein type of tail bound, for sub-exponential distributions (see Proposition 5.16 in \cite{vershynin2010introduction}) in addition to the union bound we get:
%
%$$\PP\left(\left \|\bSigma_n - \bSigma_{\bX} \right \|_{\max} \geq t \right) \leq 2 d^2 \exp\left[-\bar c\min \left(\frac{t^2 n}{16 K_{\bX}^4}, \frac{t n}{4K^2_{\bX}}\right)\right], $$
%where $\bar c$ is a absolute constant independent of the distribution of $X$. Therefore plugging in $t = 4 A_X K_{\bX}^2 \sqrt{\log d/n}$, would yield that $\left\|\bSigma_n - \bSigma_{\bX} \right \|_{\max} \leq 4 A_X K_{\bX}^2 \sqrt{\log d/n}$ with probability at least $1 - 2 d^{2 - \bar c A_X^2}$, provided that $A_X \sqrt{\log d/n} \leq 1$. 
\end{proof}

\begin{lemma} \label{arxivsupp:main:samplecovRE} Assume the same conditions as in Lemma \ref{arxivsupp:main:samplecovpopcov}, and assume further that the minimum eigenvalue  $\lambda_{\min}(\bSigma_{\bX}) > 0$ and $s \sqrt{\log d/n} \leq  (1 - \kappa)\frac{\lambda_{\min}(\bSigma_{\bX})}{(1 + \xi)^2 4 A_X K_{\bX}^2}$, where $0 < \kappa < 1$. We then have that $\bSigma_n$ satisfies the RE property with $\operatorname{RE}_{\bSigma_n}(s, \xi) \geq \kappa \operatorname{RE}_{\bSigma_{\bX}}(s, \xi) \geq \kappa \lambda_{\min} (\bSigma_{\bX}) > 0$ with probability at least $1 - 2 d^{2 - \bar c A_X^2}$.
\end{lemma}

\begin{proof}[Proof of Lemma \ref{arxivsupp:main:samplecovRE}] The proof follows a standard argument so omit the details.  
%Take a non-zero vector in the cone: $\ub \in \{ \|\ub_{S^c}\|_1 \leq \xi \|\ub_{S}\|_1 \}$, with $|S| \leq s$. Note that we have the following:
%\begin{align*}
%\left|\ub^T \bSigma_n \ub - \ub^T \bSigma_{\bX} \ub\right| & \leq \|\ub\|_1^2 \left\| \bSigma_n  - \bSigma_{\bX}\right\|_{\max} \leq (1 + \xi)^2 \|\ub_S\|_1^2 \left\| \bSigma_n  - \bSigma_{\bX}\right\|_{\max}\\
%& \leq (1 + \xi)^2 s \|\ub_S\|_2^2 \left\| \bSigma_n  - \bSigma_{\bX}\right\|_{\max}.
%\end{align*}
%The last of course implies:
%$$\operatorname{RE}_{\bSigma_n}(s, \xi) \geq \operatorname{RE}_{\bSigma_{\bX}}(s, \xi) - s(1 + \xi)^2 \left\| \bSigma_n  - \bSigma_{\bX}\right\|_{\max}.$$
%Now, on the event:
%$$\left\| \bSigma_n  - \bSigma_{\bX}\right\|_{\max} \leq 4 A_X K_{\bX}^2 \sqrt{\log d/n},$$
%we have:
%$$\operatorname{RE}_{\bSigma_n}(s, \xi) \geq \operatorname{RE}_{\bSigma_{\bX}}(s, \xi) - s(1 + \xi)^2 4 A_X K_{\bX}^2 \sqrt{\log d/n}.$$
%Thus if $s\sqrt{\log d/n} \leq (1- \kappa) \frac{\lambda_{\min}(\bSigma_{\bX})}{(1 + \xi)^2 4 A_X K_{\bX}^2}$, for some $0 < \kappa < 1$ we conclude that:
%$$\operatorname{RE}_{\bSigma_n}(s, \xi) \geq \kappa \operatorname{RE}_{\bSigma_{\bX}}(s, \xi) \geq \kappa \lambda_{\min}(\bSigma_{\bX}) > 0,$$
%where the probability bound on the event follows from Lemma \ref{arxivsupp:main:samplecovpopcov}.
\end{proof}

\begin{definition} For a fixed $0 < \kappa < 1$, let $\operatorname{RE}_{\kappa}(s, \xi) = \kappa \operatorname{RE}_{\bSigma_{\bX}}(s, \xi)$.
\end{definition}

\begin{lemma} \label{arxivsupp:main:vdiff} Assume that --- $\lambda_{\min}(\bSigma_{\bX}) > \delta > 0$, $s_{\vb} \sqrt{\log d/n} \leq (1 - \kappa) \frac{\lambda_{\min}(\bSigma_{\bX})}{(1 + 1)^2 4 A_X K_{\bX}^2}$, where $0 < \kappa < 1$ and $\lambda' \geq  \|\vb^*\|_1 4 A_X K_{\bX}^2 \sqrt{\log d/n}$. Then we have that $\|\hat \vb - \vb^*\|_{1} \leq \frac{8 \lambda' s_{\vb}}{ \operatorname{RE}_{\kappa}(s_{\vb}, 1)} $ with probability at least  $1 - 2 d^{2 - \bar c A_X^2}$. Additionally we have:
\begin{align}\label{arxivsupp:main:l1normineq}
\|\hat \vb_{S^c_{\vb}} -  \vb^*_{S^c_{\vb}}\|_1 \leq \|\hat \vb_{S_{\vb}} -  \vb^*_{S_{\vb}}\|_1.
\end{align}
\end{lemma}

\begin{proof}[Proof of Lemma \ref{arxivsupp:main:vdiff}] The proof follows a standard argument so omit the details.
\end{proof}

\begin{lemma} \label{arxivsupp:main:lemmabetaclose}

Assume the same conditions as in Lemma \ref{arxivsupp:main:samplecovpopcov} and that $\sqrt{\log d/n} \leq C$ for some constant $C$. Let $S = \supp(\bbeta^*)$, and let $\lambda = AK\sqrt{\frac{\log{d}}{n}}$. Then, with probability at least $1 - e d^{1 - \frac{cA^2}{2(1 + 2CA_X)K_{\bX}^2}} - 2 d^{2 - \bar c A_X^2}$ (where $c$ is a universal constant independent of the distribution of $\varepsilon$, $K = \|\varepsilon\|_{\psi_2}$, and the other constants are defined in Lemma \ref{arxivsupp:main:samplecovpopcov}) we have:
\begin{align}
\|\hat \bbeta_{S^c} -  \bbeta^*_{S^c}\|_1 \leq \|\hat \bbeta_{S} -  \bbeta^*_{S}\|_1,  ~~~ \mbox{and:}\label{arxivsupp:main:betadiff}
\end{align}
\begin{align}
\left\|\bSigma_n ( \bbeta^* - \hat \bbeta)\right\|_{\infty} \leq 2 \lambda. \label{arxivsupp:main:hessbetadiff}
\end{align}
\end{lemma}

\begin{proof}[Proof of Lemma \ref{arxivsupp:main:lemmabetaclose}]  Note that by a Hoeffding's type of inequality for sub-Gaussian random variables (see Proposition 5.10 \citep{vershynin2010introduction}) and the union bound, we have:
\begin{align} \label{arxivsupp:main:condprobbound}
\PP\left(\left\|\frac{1}{n} \Xb^T \bvarepsilon\right\|_{\infty} \geq t \Big | \Xb \right) \leq e d \exp\left( - \frac{c n t^2}{K^2 \|\bSigma_n\|_{\max}}\right),
\end{align}
where $c$ is a universal constant. The remainder of the proof follows by standard arguments so we omit the details.
\end{proof}

\begin{lemma} \label{arxivsupp:main:l1betadiff} Assume the same conditions as in Lemmas \ref{arxivsupp:main:samplecovpopcov}, \ref{arxivsupp:main:samplecovRE} (with $\xi = 1$), and \ref{arxivsupp:main:lemmabetaclose}, so that $\bSigma_n$ satisfies the RE assumption with $\operatorname{RE}_{\kappa}(s,1)$ with high probability. Set $\lambda = AK\sqrt{\frac{\log{d}}{n}}$, as in Lemma \ref{arxivsupp:main:lemmabetaclose}. Then with probability at least $1 - e d^{1 - \frac{cA^2}{2(1 + 2CA_X)K_{\bX}^2}} - 2 d^{2 - \bar c A_X^2}$ we have:
\begin{align}
\|\hat \bbeta - \bbeta^* \|_1 \leq \frac{8 A K}{\operatorname{RE}_{\kappa}(s,1)} s \sqrt{\log d/n} \label{arxivsupp:main:betaL1},\\
\|\Xb (\hat \bbeta - \bbeta^*)\|_2^2 \leq \frac{16 A^2 K^2}{\operatorname{RE}_{\kappa}(s,1)} s \log d. \label{arxivsupp:main:XbetaL2}
\end{align}
\end{lemma}

\begin{proof}[Proof of Lemma \ref{arxivsupp:main:l1betadiff}] The proof follows by standard arguments and we omit the details.
%From Lemma \ref{arxivsupp:main:lemmabetaclose}, we know that on the event $E$ we have, that (\ref{arxivsupp:main:betadiff}) and (\ref{arxivsupp:main:hessbetadiff}) hold. Thus on the event $E$, we have:
%\begin{align*}
%\frac{1}{n} \|\Xb (\bbeta^* - \hat \bbeta)\|_2^2 & \leq \|\bSigma_n (\bbeta^* - \hat \bbeta)\|_{\infty}\| (\bbeta^* - \hat \bbeta)\|_{1} \leq 2\lambda (2 \|\bbeta^*_{S} - \hat \bbeta_{S}\|_1) \leq 4 \lambda \sqrt{s} \|\bbeta^*_{S} - \hat \bbeta_{S}\|_2.
%\end{align*}
%Now note that on the event $E$ --- $\frac{1}{n}\Xb^T \Xb = \bSigma_n$ satisfies the RE condition and thus we have:
%$$\frac{1}{n} \|\Xb (\bbeta^* - \hat \bbeta)\|_2^2 \geq \operatorname{RE}_{\kappa}(s,1)  \|\bbeta^*_{S} - \hat \bbeta_{S}\|_2^2.$$
%The last two inequalities yield:
%\begin{align*}
% \|\bbeta^*_{S} - \hat \bbeta_{S}\|_2 & \leq \frac{4 \lambda \sqrt{s}}{ \operatorname{RE}_{\kappa}(s,1)}, ~~~~ \frac{1}{n} \|\Xb (\bbeta^* - \hat \bbeta)\|_2^2 \leq \frac{16 \lambda^2 s}{ \operatorname{RE}_{\kappa}(s,1)}.
%\end{align*}
%The last inequality actually gives (\ref{arxivsupp:main:XbetaL2}). To get (\ref{arxivsupp:main:betaL1}), note that by (\ref{arxivsupp:main:betadiff}) we have:
%\begin{align*}
%\|\hat \bbeta -  \bbeta^*\|_1 \leq 2 \|\hat \bbeta_{S} -  \bbeta^*_{S}\|_1 \leq \frac{8 AK \lambda s}{\operatorname{RE}_{\kappa}(s,1)},
%\end{align*}
%and we are done. 
\end{proof}

\begin{lemma} \label{arxivsupp:main:Mbound} Let $\{\bX_i\}_{i = 1}^n$ are identical (not necessarily independent), $d$-dimensional sub-Gaussian vectors with $\max_{\substack{1 \leq i \leq n, 1 \leq j \leq d}}\|X_{ij}\|_{\psi_2} = K$. Then we have:
\begin{align*}
\max_{i = 1,\ldots,n}\|\bX_i^{\otimes 2}\|_{\max} = O_p(\log(nd)).
\end{align*}
\end{lemma}

\begin{proof}[Proof of Lemma \ref{arxivsupp:main:Mbound}] The proof follows after an application of a Bernstein type of of tail bound \citep[see e.g.]{vershynin2010introduction} and we omit the details.
%Note that by the union bound for a fixed $i$ we have:
%$$
%\PP(\|\bX_i\|_{\infty} \geq t) \leq d \exp(1- ct^2/K^2),
%$$
%where $c$ is an absolute constant, by (5.10) in \cite{vershynin2010introduction}. With yet another union bound we get:
%$$
%\PP(\max_{i = 1,\ldots,n}\|\bX_i\|_{\infty} \geq t) \leq nd \exp(1- ct^2/K^2).
%$$
%Thus as long as $t \geq C \sqrt{\log(nd)}$ for a large enough $C$ the above probability will converge to $0$. This finishes the proof, since clearly:
%$$
%\max_{i = 1,\ldots,n}\|\bX_i^{\otimes 2}\|_{\max} \leq \max_{i = 1,\ldots,n} \|\bX_i\|_{\infty}^2 \leq t^2.
%$$
\end{proof}

\begin{lemma} \label{arxivsupp:main:simple:proof:exp:psi1} Let $X_i, i = 1\ldots k$ are sub-exponential with $\|X_i\|_{\psi_\ell} \leq U$, for some $\ell \geq 1$ and denote by $\bX$ the vector with entries $X_i$. Then for any $p,q \geq 1$ we have:
$$
\EE \|\bX\|^p_{q} \leq [\EE[\|\bX\|^{pq}_{q}]]^{1/(qp)} \leq (pq)^{p/\ell} U^p k^{p/q}
$$
\end{lemma}

\begin{proof}[Proof of Lemma \ref{arxivsupp:main:simple:proof:exp:psi1}]
We apply Jensen's followed by Minkowski's inequality to obtain the following:
$$
[\EE[\|\bX\|^p_{q}]]^{1/p} \leq [\EE[\|\bX\|^{pq}_{q}]]^{1/(qp)} \leq [k (\EE (X_i)^{pq})^{1/q}]^{1/p} \leq (pq)^{1/\ell} U k^{1/p},
$$
where the last inequality follows by the definition of $\psi_\ell$ norm. Raising this inequality to the power of $p$ finishes the proof.
\end{proof}

\begin{lemma} \label{arxivsupp:main:Xepsbound} Let $R \subset \{1,\ldots,d\}$ with $|R| = r$. Then we have the following:
$$
\EE \|(\bX \varepsilon)_R\|_2^4 \leq r^{2} 2^8 (K K_{\bX})^4.
$$
\end{lemma}

\begin{proof}[Proof of Lemma \ref{arxivsupp:main:Xepsbound}] Simply observe that $\EE\|(\bX \varepsilon)_R \|^4_{2} \leq \sqrt{\EE |\varepsilon|^8} \sqrt{\EE\|(\bX)_R \|^8_{2}}$, and apply Lemma \ref{arxivsupp:main:simple:proof:exp:psi1} for $\psi_2$.
\end{proof}

\section{Proofs for IVR} \label{arxivsupp:main:DantzigProofs:IVR}
%We recall the definition of 

\begin{definition}[CS]\label{arxivsupp:coord:sens:IVR} For the (not necessarily symmetric) matrix $\Mb_{k\times k}$ we define its coordinate-wise sensitivity with respect to the $L_1$ norm by:
	\begin{align*}\operatorname{CS}_\Mb(s, \xi) & = \min_{S \subset \{1,\ldots,k\}, |S| \leq s} \min \left\{ s\|\Mb \ub\|_{\infty} : \ub \in \mathbb{R}^d\setminus\{0\}, \|\ub_{S^c}\|_1 \leq \xi \|\ub_{S}\|_1, \|\ub_{S}\|_1 = 1\right\} > 0.
	\end{align*}
	This definition is inspired by \cite{gautier2011high}.
\end{definition}

\begin{definition} 
Denote with $\Xb$ and $\Wb$ the $n \times d$ matrices whose rows are the $\bX_i^T$ and $\bW_i^T$ vectors stacked together respectively. Let $\bY$ be the an $n \times 1$ vector stacking the observations $Y_i$ for $i = 1,\ldots,n$ and let $\bvarepsilon = \bY - \Xb\bbeta^*$.
\end{definition}

\begin{proof}[Proof of Corollary \ref{normdantzigfinal:IVR}] 
The proof is the same as the proof of Corollary \ref{normdantzigfinal} upon usages of the Lemmas developed in this section. We omit the details. 
\end{proof}

\begin{remark}\label{arxivsupp:main:unif:CI:danztig:IVR} In fact, the proof of Corollary \ref{normdantzigfinal:IVR} implies that the uniform types of assumptions in Section \ref{arxivsupp:main:UWCUNsec} are satisfied, and hence under the same assumptions as in Corollary \ref{normdantzigfinal:IVR}, we have:
$$
\lim_{n \rightarrow \infty} \sup_{\|\bbeta\|_0 \leq s} \sup_{t\in\RR} |\PP_{\bbeta}(\hat U_n \leq t) - \Phi(t))| = 0. 
$$
\end{remark}
%\begin{remark}\label{arxivsupp:v:lower:bound} Note here that it is implied that $\lambda' = o(1)$ and hence since $\|\vb^*\|_1 \geq 2 K_{\bW\bX}^{-2}$ it follows that $\lambda = o(1)$ as well.
%\end{remark}
%
%\begin{remark} Observe that $\|\vb^*\|_1 \leq \sqrt{s_{\vb}}\|\vb^*\|_2 \leq  \sqrt{s_{\vb}} \delta$. This yields sufficient conditions by substituting $\|\vb^*\|_1$ with $\sqrt{s_{\vb}}$. Moreover, under the assumption that $\vb^{*T}\bZ$ is sub-Gaussian, we can further relax the requirements on sparsity $s_{\vb}$ dimension $d$ and number of observations $n$.
%\end{remark}

\begin{lemma} \label{arxivsupp:main:CLTDantzig:IVR} Assume that condition \ref{subGausserr:design:bounded:mom:matr} holds and $\max(s_{\vb}, s) (\|\vb^*\|_1 \vee 1)\log d/\sqrt{n} = o(1)$.
Then:
$$
\Delta^{-1/2} n^{1/2}  S(\bbeta^*) \rightsquigarrow N(0,1). 
$$
\end{lemma}

\begin{proof}[Proof of Lemma \ref{arxivsupp:main:CLTDantzig:IVR}]

The proof is the same as that of Lemma \ref{arxivsupp:main:CLTDantzig} after using the Lemmas developed below. We omit the details.
%{\color{green} USE $\max(s_{\vb}, s) (\|\vb^*\|_1 \vee 1)\log d/\sqrt{n} = o(1)$}
%{\color{red} TOBE OMITTED; To show the weak convergence we verify Lyapunov's condition for the CLT. We need to show that the following expression converges to $0$:
%$$
% \frac{n^{-2}}{\Delta^{2}} \sum_{i=1}^n \EE \left|\vb^{*T}\bW_i(\bX_i^T\bbeta^* - Y_i)\right|^4.
%$$
%Note that we have $\Delta^{2} \geq \lambda_{\min}(\bSigma_{\bW\bX}\bSigma_{\bW\bX}^T)\|\vb^{*}\|_2^4 \Var(\varepsilon)^{2} = O(1) \|\vb^{*}\|_2^4$. Therefore it suffices to consider the following expression:
%\begin{align}
% \frac{n^{-2}}{\|\vb^{*}\|_2^4}  \sum_{i=1}^n \EE \left|\vb^{*T}\bW_i(\bX_{i}^T\bbeta^* - Y_i)\right|^4 & \leq  n^{-2}  \sum_{i=1}^n \EE \left\|(\bW_i\varepsilon_i)_{S_{\vb}}\right\|_2^4 \leq n^{-1} s^{2}_{\vb} M, \label{arxivsupp:main:LyapunovBound:IVR}
%\end{align}
%where $M =  2^8 (K K_{\bW\bX})^4$, and the last inequality holding from Lemma \ref{arxivsupp:main:Xepsbound}. This completes the proof under our assumptions. 
%}
\end{proof}

\begin{remark} Using the Berry-Esseen theorem for non-identical random variables in combination with Lemma \ref{arxivsupp:main:Xepsbound} we can further show:
$$
\sup_t \left|\PP^* \left(\frac{n^{1/2}}{\sqrt{\Delta}} S(\bbeta^*) \leq t\right) - \Phi(t)\right| \leq C_{BE} (6 K K_{\bW\Xb})^3 n^{-1/2} s^{3/2}_{\vb} = o(1),
$$
where $M$ and $C_{BE}$ are absolute constants.
\end{remark}

\begin{proposition} \label{arxivsupp:main:pluginDantzig:IVR} 
Under assumption \ref{subGausserr:design:bounded:mom:matr}, and the following additional assumption:
$$
 \|\vb^*\|_1^2\sqrt{\log d/n} = o(1),
$$
 we have that $\hat \Delta \rightarrow_p \Delta$.
\end{proposition}

\begin{proof}[Proof of Proposition \ref{arxivsupp:main:pluginDantzig:IVR}]
By the triangle inequality, for any two vectors $\ab$ and $\bb$ we have $
|\|\ab\|_2 - \|\bb\|_2|\leq \|\ab - \bb\|_2.$ Making multiple usages of this inequality one realizes that it suffices to show:
\begin{align*}
n^{-1} \sum_{i = 1}^n ((\vb^{*} - \hat \vb)^T \bW_i)^2 (( \bbeta^* - \hat \bbeta)^{T} \bX_i)^2 = o_p(1), n^{-1} \sum_{i = 1}^n (\vb^{*T} \bW_i)^2 (( \bbeta^* - \hat \bbeta)^{T} \bX_i)^2 = o_p(1),\\
n^{-1} \sum_{i = 1}^n ((\vb^{*} - \hat \vb)^T \bW_i)^2 \varepsilon_i^2 = o_p(1), n^{-1} \sum_{i = 1}^n (\vb^{*T} \bW_i)^2 \varepsilon_i^2 - \EE (\vb^{*T} \bW_i)^2 \varepsilon_i^2= o_p(1),
\end{align*}
and $\EE (\vb^{*T} \bW_i)^2\varepsilon_i^2  < \infty$. We show these convergences in turn. For the first term we have:
$$
n^{-1} \sum_{i = 1}^n ((\vb^{*} - \hat \vb)^T \bW_i)^2 (( \bbeta^* - \hat \bbeta)^{T} \bX_i)^2 \leq \|\vb^{*} - \hat \vb\|_1^2 \|\bbeta^* - \hat \bbeta\|_1^2 O_p(\log(nd)^2) = o_p(1),
$$
with high probability, where use used Lemma \ref{arxivsupp:main:Mbound} and Lemmas \ref{arxivsupp:main:l1betadiff:IVR} and  \ref{arxivsupp:main:vdiff:IVR}. For the second term:
\begin{align*}
n^{-1} \sum_{i = 1}^n (\vb^{*T} \bW_i)^2 (( \bbeta^* - \hat \bbeta)^{T} \bX_i)^2 & \leq \|\bbeta^* - \hat \bbeta\|_1^2 \max_{i = 1,\ldots,n} \|\bX_i\|^2_{\infty} n^{-1}\sum_{i = 1}^n (\vb^{*T}\bW_i)^2\\
&  \leq O_p\Big(\frac{s^2 \log d \log(nd)}{n}\Big) n^{-1}\sum_{i = 1}^n (\vb^{*T}\bW_i)^2,
\end{align*}
with high probability. By Lemma \ref{arxivsupp:main:samplecovpopcov:IVR} we have:
$$
n^{-1}\sum_{i = 1}^n (\vb^{*T}\bW_i)^2 \leq \|\vb^*\|^2_1 \underbrace{\|n^{-1}\sum_{i = 1}\bW_i^{\otimes 2} - \bSigma_{\bW\bW}\|_{\max}}_{O_p(\sqrt{\log d/n})} + \underbrace{\vb^{*T}\bSigma_{\bW\bW}\vb^{*}}_{\leq D_{\max}} = O_p(1),
$$
which combined with the bound in the previous display completes the proof for the second term. The third term bound follows upon noticing: 
\begin{align*}
n^{-1} \sum_{i = 1}^n ((\vb^{*} - \hat \vb)^T \bW_i)^2 \varepsilon_i^2 \leq \|\hat \vb - \vb^*\|_1^2 \max \|\bW_i\|^2_{\infty} \underbrace{n^{-1} \sum \varepsilon^2_i}_{O_p(1)} & \leq O_p\Big (\frac{s^2 \log d \log (nd)}{n}\Big) O_p(1) \\
& = o_p(1).
\end{align*}
Moving to the last term we have:
$$
\Big | n^{-1} \sum_{i = 1}^n (\vb^{*T} \bW_i)^2 \varepsilon_i^2 - \EE (\vb^{*T} \bW_i)^2 \varepsilon_i^2 \Big | \leq \|\vb^{*}\|^2_{1} \|n^{-1}\sum_{i = 1}^n \bW_{i,S_{\vb}}^{\otimes 2}\varepsilon_i^2 - \bSigma_{\bW\bW, S_{\vb}S_{\vb}} \sigma^2\|_{\infty},
$$
where $S_{\vb} = \supp(\vb^*)$. The final concentration is handled in Lemma \ref{conc:poly:subGauss:IVR}. Applying this lemma in conjunction with the union bound gives us the existence of a constant $C_{K\bW\bX}$ depending on $K$ and $K_{\bW\bX}$ such that:
$$
\PP( \|n^{-1}\sum_{i = 1}^n \bW_{i,S_{\vb}}^{\otimes 2}\varepsilon_i^2 - \bSigma_{\bW\bW, S_{\vb}S_{\vb}} \sigma^2\|_{\infty} \geq t) \leq s_{\vb}^2 \frac{C_{K\bW\bX}^k [\sqrt{k/n} + k^2/n]^k}{t^k},
$$
for all $k \in \NN$. Selecting $t = 2e^4C_{K\bW\bX}\sqrt{\frac{\log d}{n}}$, $k = \lceil \min(\log d, (n \log d)^{1/4}) \rceil$ brings the above bound of the order $O(s^2_{\vb}\exp(-4 \lceil \min(\log d, (n \log d)^{1/4}) \rceil)) = o(1)$, and shows that:
$$
\Big | n^{-1} \sum_{i = 1}^n (\vb^{*T} \bW_i)^2 \varepsilon_i^2 - \EE (\vb^{*T} \bW_i)^2 \varepsilon_i^2\Big | \leq O\Big (\|\vb^*\|_1^2 \sqrt{\frac{\log d}{n}}\Big) = o(1).
$$
We conclude the proof with noticing that $\EE [\EE [(\vb^{*T} \bW_i)^2 \varepsilon_i^2 | \bW]] = \sigma^2 \EE[\vb^{*T} \bW_i^{\otimes2} \vb^*] < \infty$. 

\end{proof}

\begin{lemma} \label{arxivsupp:main:samplecovpopcov:IVR} We have that with probability at least $1 - 2 d^{2 - \bar c A_X^2}$:

$$\left\|\frac{1}{n} \sum_{i = 1}^n ([\bW_i^T, \bX_i^T]^{T})^{\otimes 2} - \bSigma_{}\right\|_{\max} = \left\|\bSigma_n - \bSigma_{}\right\|_{\max} \leq 4 A_{WX} K_{\bW\bX}^2 \sqrt{\log d/n}.$$
{\bf Note.} The constant $\bar c$ is a universal constant independent of the $\bX$ and $\bW$ distributions, $K_{\bW\bX}$ is as defined in the main text, and $A_{WX} > 0$ is an arbitrarily chosen constant satisfying $A_{WX}\sqrt{\log d/n} \leq 1$.
\end{lemma}

\begin{proof}[Proof of Lemma \ref{arxivsupp:main:samplecovpopcov:IVR}] Proof is follows by the same argument as in Lemma \ref{arxivsupp:main:samplecovpopcov}, so we omit the details.
%First we note that the elements of the matrix -- $[\bW;\bX]^{\otimes 2}$ are sub-exponential random variables. This fact can be seen since by Cauchy-Schwartz, for two sub-Gaussian variables $W$ and $X$ such that $\max(\|W\|_{\psi_2},\|X\|_{\psi_2})\leq K_{\bW\bX}$ one can easily obtain that:
%\begin{align}
%\|WX\|_{\psi_1} \leq  2\|W\|_{\psi_2} \|X\|_{\psi_2} \leq 2K^2_{\bW\bX}. \label{arxivsupp:main:subexpo:IVR}
%\end{align}
%Next by the triangle inequality it is clear that $\|WX - \EE WX\|_{\psi_1} \leq  4K^2_{\bW\bX}.$ Using a Bernstein type of tail bound, for sub-exponential distributions (see Proposition 5.16 in \cite{vershynin2010introduction}) in addition to the union bound we get:
%
%$$\PP\left(\left \|\bSigma_n - \bSigma \right \|_{\max} \geq t \right) \leq 2 d^2 \exp\left[-\bar c\min \left(\frac{t^2 n}{16 K_{\bW\bX}^4}, \frac{t n}{4K^2_{\bW\bX}}\right)\right], $$
%where $\bar c$ is a absolute constant. Therefore plugging in $t = 4 A_{WX} K_{\bW\bX}^2 \sqrt{\log d/n}$, would yield that $\left\|\bSigma_n - \bSigma_{} \right \|_{\max} \leq 4 A_{WX} K_{\bW\bX}^2 \sqrt{\log d/n}$ with probability at least $1 - 2 d^{2 - \bar c A_{WX}^2}$, provided that $A_{WX} \sqrt{\log d/n} \leq 1$. 
\end{proof}

\begin{lemma} \label{arxivsupp:main:samplecovRE:IVR} Assume the same conditions as in Lemma \ref{arxivsupp:main:samplecovpopcov:IVR}, and assume further that the matrix $\bSigma_{\bW\bX}$ satisfies $\operatorname{CS}_{\bSigma_{\bW\bX}}(s,\xi) > \kappa^*$ and that $s$ is sufficiently small so that $s \sqrt{\log d/n} \leq  (1 - \kappa)\frac{\kappa^*}{(1 + \xi) 4 A_{WX} K_{\bW\bX}^2}$, where $0 < \kappa < 1$. We then have that $\bSigma_{\bW\bX,n}$ satisfies the CS property with $\operatorname{CS}_{\bSigma_{\bW\bX,n}}(s, \xi) \geq \kappa \operatorname{CS}_{\bSigma_{\bW\bX}}(s, \xi) > 0$ with probability at least $1 - 2 d^{2 - \bar c A_{WX}^2}$.
\end{lemma}

\begin{proof}[Proof of Lemma \ref{arxivsupp:main:samplecovRE:IVR}] This proof is simply using Definition \ref{arxivsupp:coord:sens:IVR} and  Lemma \ref{arxivsupp:main:samplecovpopcov:IVR}.
%Take a non-zero vector in the cone: $\ub \in \{\ub: \|\ub_{S^c}\|_1 \leq \xi \|\ub_{S}\|_1,  \|\ub_{S}\|_1 = 1\}$, with $|S| \leq s$. Note that we have the following:
%\begin{align*}
%\left\||\bSigma_{\bW\bX,n} \ub - \bSigma_{\bW\bX} \ub\right\|_{\infty} & \leq \left\| \bSigma_{\bW\bX,n}  - \bSigma_{\bW\bX}\right\|_{\max}\|\ub\|_{1} \leq (1 + \xi) \left\| \bSigma_{\bW\bX,n}  - \bSigma_{\bW\bX}\right\|_{\max}
%\end{align*}
%The last of course implies:
%$$\operatorname{CS}_{\bSigma_{\bW\bX,n}}(s, \xi) \geq \operatorname{CS}_{\bSigma_{\bW\bX}}(s, \xi) - s(1 + \xi) \left\| \bSigma_{\bW\bX,n}  - \bSigma_{\bW\bX}\right\|_{\max}.$$
%Now, on the event:
%$$\left\| \bSigma_{\bW\bX,n}  - \bSigma_{\bW\bX}\right\|_{\max} \leq 4 A_{WX} K_{\bW\bX}^2 \sqrt{\log d/n},$$
%we have:
%$$\operatorname{CS}_{\bSigma_{\bW\bX,n}}(s, \xi) \geq \operatorname{CS}_{\bSigma_{\bW\bX}}(s, \xi) - s(1 + \xi) 4 A_{WX} K_{\bW\bX}^2 \sqrt{\log d/n}.$$
%Thus if $s\sqrt{\log d/n} \leq (1- \kappa) \frac{\operatorname{CS}_{\bSigma_{\bW\bX}}(s, \xi)}{(1 + \xi) 4 A_{WX} K_{\bW\bX}^2}$, for some $0 < \kappa < 1$ we conclude that:
%$$\operatorname{CS}_{\bSigma_{\bW\bX,n}}(s, \xi) \geq \kappa \operatorname{CS}_{\bSigma_{\bW\bX}}(s, \xi)> 0,$$
%where the probability bound on the event follows from Lemma \ref{arxivsupp:main:samplecovpopcov:IVR}.
\end{proof}

\begin{definition} For a fixed $0 < \kappa < 1$, let $\operatorname{CS}_{\kappa}(s, \xi) = \kappa \operatorname{CS}_{\bSigma_{\bW\bX}}(s, \xi)$.
\end{definition}

\begin{lemma} \label{arxivsupp:main:vdiff:IVR} Assume that $\operatorname{CS}_{\bSigma_{\bW\bX}}(s_{\vb}, 1) \geq \kappa^* > 0$, and that further $s_{\vb}$ is small enough so that $s_{\vb} \sqrt{\log d/n} \leq (1 - \kappa) \frac{\kappa^*}{(1 + 1) 4 A_{WX} K_{\bW\bX}^2}$, where $0 < \kappa < 1$ and $\lambda' \geq  \|\vb^*\|_1 4 A_{WX} K_{\bW\bX}^2 \sqrt{\frac{\log d}{n}}$. Then we have that $\|\hat \vb - \vb^*\|_{1} \leq \frac{8 \lambda' s_{\vb}}{\operatorname{CS}_{\bSigma_{\bW\bX}}(s_{\vb}, 1)} $ with probability at least  $1 - 2 d^{2 - \bar c A_{WX}^2}$.
\end{lemma}

\begin{proof}[Proof of Lemma \ref{arxivsupp:main:vdiff:IVR}] 
%We start by showing that $\vb^*$ satisfies its corresponding constraint, i.e.
%
%$$\bigg\|\frac{1}{n}\sum_{i=1}^n\vb^{*T}\bW_i\bX_i^{T}-\eb_1\bigg\|_\infty\leq \lambda',$$
%with high probability. To this end, note that:
%\begin{align*}
%\bigg\|\frac{1}{n}\sum_{i=1}^n\vb^{*T}\bW_i\bX_i^{T}-\eb_1\bigg\|_\infty & \leq \|\vb^*\|_1 \left\| \bSigma_{\bW\bX,n}  - \bSigma_{\bW\bX}\right\|_{\max} \leq  \|\vb^*\|_1 4 A_{WX} K_{\bW\bX}^2 \sqrt{\log d/n},
%\end{align*}
%where the last inequality holds with probability at least $1 - 2 d^{2 - 2 \bar c A_{WX}^2}$ as in Lemma \ref{arxivsupp:main:samplecovpopcov:IVR}. Thus for values of $\lambda' \geq  \|\vb^*\|_1 4 A_{WX} K_{\bW\bX}^2 \sqrt{\log d/n}$, the above gives us that:
Using a standard argument and Lemma \ref{arxivsupp:main:samplecovpopcov:IVR} we can show that with probability  at least $1 - 2 d^{2 - 2 \bar c A_{WX}^2}$, $\vb^*$ satisfies (\ref{lambda:prime:dantz:IVR}) and consequently
\begin{align} \label{arxivsupp:main:vminusvhatcov:IVR}
\bigg\|\frac{1}{n}\sum_{i=1}^n (\hat \vb - \vb^*)^T\bW_i\bX_i^{T}\bigg\|_\infty & \leq \bigg\|\frac{1}{n}\sum_{i=1}^n\hat \vb^{T}\bW_i\bX_i^{T}-\eb_1\bigg\|_\infty + \bigg\|\frac{1}{n}\sum_{i=1}^n\vb^{*T}\bW_i\bX_i^{T}-\eb_1\bigg\|_\infty \leq 2 \lambda'.
\end{align}

Let $S_{\vb} = \supp(\vb^*)$, with $s_{\vb} = |S_{\vb}|$. 
%We can therefore conclude that:
%$$\| \vb^*_{S_{\vb}}\|_1 = \| \vb^*\|_1 \geq \|\hat \vb\|_1 =  \|\hat \vb_{S_{\vb}}\|_1 +  \|\hat \vb_{S^c_{\vb}}\|_1.$$
%Furthermore, by the triangle inequality:
%$$ \|\hat \vb_{S_{\vb}} \|_1 \geq \| \vb^*_{S_{\vb}}\|_1 - \|\hat \vb_{S_{\vb}} -  \vb^*_{S_{\vb}}\|_1.$$
%Combining the last two inequalities we get that:
Using the formulation of program (\ref{lambda:prime:dantz:IVR}) it is not hard to show that:
\begin{align}\label{arxivsupp:main:l1normineq:IVR}
\|\hat \vb_{S^c_{\vb}} -  \vb^*_{S^c_{\vb}}\|_1 \leq \|\hat \vb_{S_{\vb}} -  \vb^*_{S_{\vb}}\|_1.
\end{align}
By Lemma \ref{arxivsupp:main:samplecovRE:IVR}, $\bSigma_{\bW\bX,n}$ satisfies the CS assumption under our conditions and hence
%\begin{align*}
%\frac{1}{n}(\hat \vb - \vb^* )^T\Zb^T \Xb (\hat \vb - \vb^* ) & \leq \left \|\bSigma_{\bZ\bX,n} (\hat \vb - \vb^*)\right\|_{\infty}\| (\hat \vb - \vb^*)\|_{1} \stackrel{\mbox{\tiny by (\ref{arxivsupp:main:vminusvhatcov}),(\ref{arxivsupp:main:l1normineq})}}{\leq} 2\lambda' (2 \|\hat \vb_{S_\vb} -  \vb^*_{S_\vb}\|_1)\\
%& \leq 4 \lambda' \sqrt{s_{\vb}} \|\hat \vb_{S_\vb} -  \vb^*_{S_\vb}\|_2.
%\end{align*}
\begin{align*}
\frac{\|\hat \vb_{S_{\vb}} - \vb_{S_{\vb}}^*\|_1\operatorname{CS}_{\bSigma_{\bW\bX}}(s_{\vb}, 1) }{s_{\vb}} & \leq \left \|\bSigma_{\bW\bX,n} (\hat \vb - \vb^*)\right\|_{\infty} \stackrel{\mbox{\tiny by (\ref{arxivsupp:main:vminusvhatcov:IVR})}}{\leq} 2\lambda'
\end{align*}
%& \leq 4 \lambda' \sqrt{s_{\vb}} \|\hat \vb_{S_\vb} -  \vb^*_{S_\vb}\|_2.
Hence by (\ref{arxivsupp:main:l1normineq:IVR}) we conclude $
\|\hat \vb -  \vb^*\|_1 \leq 4 s_{\vb}\lambda' /\operatorname{CS}_{\bSigma_{\bW\bX}}(s_{\vb}, 1),$
which is what we wanted to show.
%On the other hand, by Lemma \ref{arxivsupp:main:samplecovRE}, we know that the matrix $\bSigma_{\bZ\bX,n}$ satisfies the RS condition with $\operatorname{RE}_{\kappa}(s_{\vb}, 1)$, on the same event on which we are working on, provided that $s_{\vb} \sqrt{\log d/n} \leq (1 - \kappa) \frac{SV_{\min}(\bSigma_{\bZ\bX})}{(1 + 1)^2 2 A_{WX} K_{\bW\bX}^2}$, for some $\kappa < 1$. This implies that:
%$$\frac{1}{n}(\hat \vb - \vb^* )^T\Zb^T \Xb (\hat \vb - \vb^* )\geq \operatorname{RE}_{\kappa}(s_{\vb}, 1) \| \hat \vb_{S_{\vb}} - \vb^*_{S_{\vb}} \|_2^2.$$
%The last inequality gives us that:
%$$ \|  \hat \vb - \vb^* \|_1 \leq 2 \| \hat \vb_{S_{\vb}} - \vb^*_{S_{\vb}} \|_1 \leq 2 \sqrt{s_{\vb}} \| \hat  \vb_{S_{\vb}} - \vb^*_{S_{\vb}} \|_2 \leq \frac{8 \lambda' s_{\vb}}{ \operatorname{RE}_{\kappa}(s_{\vb}, 1)}, $$
%with probability at least $1 - 2 d^{2 - \bar c A_{WX}^2}$, as claimed.
\end{proof}

\begin{lemma} \label{arxivsupp:main:lemmabetaclose:IVR}

Assume the same conditions as in Lemma \ref{arxivsupp:main:samplecovpopcov:IVR} and that $\sqrt{\log d/n} \leq C$ for some constant $C$. Let $S = \supp(\bbeta^*)$, and let $\lambda = AK\sqrt{\frac{\log{d}}{n}}$. Then, with probability at least $1 - 2d^{1 - (cA^2)/(4 K^2_{\bX\bX})}$ (where $c$ is a universal constant Lemma \ref{arxivsupp:main:samplecovpopcov:IVR}) we have:
\begin{align}
\|\hat \bbeta_{S^c} -  \bbeta^*_{S^c}\|_1 \leq \|\hat \bbeta_{S} -  \bbeta^*_{S}\|_1, \label{arxivsupp:main:betadiff:IVR}
\end{align}
and:
\begin{align}
\left\|\bSigma_{\bW\bX,n} ( \bbeta^* - \hat \bbeta)\right\|_{\infty} \leq 2 \lambda. \label{arxivsupp:main:hessbetadiff:IVR}
\end{align}
\end{lemma}

\begin{proof}[Proof of Lemma \ref{arxivsupp:main:lemmabetaclose:IVR}]

Note that by a Bernstein type of inequality for sub-exponential random variables (see Proposition 5.16 \citep{vershynin2010introduction}) and the union bound, we have:

\begin{align} \label{arxivsupp:main:condprobbound:IVR}
\PP\left(\left\|\frac{1}{n} \Wb^T \bvarepsilon\right\|_{\infty} \geq t \right) \leq 2 d \exp\left( - c \min \Big(\frac{n t^2}{4K^2K^2_{\bW\bX}}, \frac{nt}{2K K_{\bW\bX}}\Big) \right),
\end{align}
where $c$ is a universal constant, and we used the fact that $\max_{i \in \{1,\ldots,d\}}\|\varepsilon W_i\|_{\psi_1} \leq 2 K K_{\bZ\bX}$. Set $t = \lambda$. Provided that $A \leq 2 K_{\bW\bX}$ the above probability bounds gives us that the event $E := \left\{\left\|\frac{1}{n} \Wb^T \bvarepsilon\right\|_{\infty} \leq \lambda \right\}$ holds with probability at least $1 - 2d^{1 - (cA^2)/(4 K^2_{\bX\bX})}$. 
%Under the assumption that $\sqrt{\log d/n} < C$, we have that on the event considered in Lemma \ref{arxivsupp:main:samplecovpopcov:IVR}, that $\|\bSigma_{\bW\bW,n}\|_{\max} \leq \|\bSigma_{\bW\bW}\|_{\max} + 4CA_{WX}K_{\bW\bX}^2$ with probability at least $1 - 2 d^{2 - \bar c A_{WX}^2}$. Note that $\|\bSigma_{\bW\bW}\|_{\max} \leq \max_{i = 1,\ldots, d}\EE(W_i)^2 \leq 2 K_{\bW\bX}^2$, by the sub-Gaussian assumption on $\bZ$ and the definition of $\psi_2$ norm. Hence $\|\bSigma_{\bW\bW,n}\|_{\max}\leq 2(1 + 2CA_{WX})K_{\bW\bX}^2$. Let $E_{WX} = \{\|\bSigma_{\bW\bW,n}\|_{\max} \leq 2(1 + 2CA_{WX})K_{\bW\bX}^2\}$.

%Thus on the event $E_{WX}$, setting the value $t = \lambda = A K \sqrt{\log d/n}$, the probability bound (\ref{arxivsupp:main:condprobbound:IVR}) becomes $e d^{1 - \frac{cA^2}{2(1 + 2CA_{WX})K_{\bW\bX}^2}}$. Denote with $E = \left\{\left\|\frac{1}{n} \Wb^T \bvarepsilon\right\|_{\infty} \leq \lambda \right\} \cap E_{WX}$, which holds with probability at least $1 - e d^{1 - \frac{cA^2}{2(1 + 2CA_{WX})K_{\bW\bX}^2}} - 2 d^{2 - \bar c A_{WX}^2}$ by the union bound. 
Note that when $E$ holds, the true parameter satisfies the Dantzig selector constraint and thus we can obtain (\ref{arxivsupp:main:betadiff:IVR}) in the same manner as in Lemma \ref{arxivsupp:main:vdiff:IVR}. Using the triangle inequality on $E$ shows (\ref{arxivsupp:main:hessbetadiff:IVR}).
%Note that when $E$ holds, the true parameter satisfies the Dantzig selector constraint and thus we can obtain (\ref{arxivsupp:main:betadiff:IVR}) in the same manner as in Lemma \ref{arxivsupp:main:vdiff:IVR}. To obtain (\ref{arxivsupp:main:hessbetadiff:IVR}), note that by the triangle inequality on the event $E$ we have:
%\begin{align*}
%\left\|\bSigma_{\bW\bX,n} ( \bbeta^* - \hat \bbeta)\right\|_{\infty} & \leq \left\|n^{-1} \Wb^T \bvarepsilon\right\|_{\infty} + \left\|n^{-1}  \Wb^T (\bY - \Xb \hat \bbeta)\right\|_{\infty} \leq 2 \lambda,
%\end{align*}
%with probability at least $1 - 2d^{1 - (cA^2)/(4 K^2_{\bX\bX})}$ as claimed.
\end{proof}

\begin{lemma} \label{arxivsupp:main:l1betadiff:IVR} Assume the same conditions in Lemmas \ref{arxivsupp:main:samplecovpopcov:IVR}, \ref{arxivsupp:main:samplecovRE:IVR} (with $\xi = 1$), and \ref{arxivsupp:main:lemmabetaclose:IVR}, so that $\bSigma_{\bW\bX,n}$ satisfies the CS assumption with $\operatorname{CS}_{\kappa}(s,1)$ with high probability. Set $\lambda = AK\sqrt{\frac{\log{d}}{n}}$, as in Lemma \ref{arxivsupp:main:lemmabetaclose:IVR}. Then with probability at least $1 - 2d^{1 - (cA^2)/(4 K^2_{\bX\bX})}$ we have:
\begin{align}
\|\hat \bbeta - \bbeta^* \|_1 \leq \frac{4 A K}{\operatorname{CS}_{\kappa}(s,1)} s \sqrt{\log d/n} \label{arxivsupp:main:betaL1:IVR}.
\end{align}
\end{lemma}

\begin{proof} Recall from Lemma \ref{arxivsupp:main:lemmabetaclose:IVR}, that on the event $E$ we have that (\ref{arxivsupp:main:betadiff:IVR}) and (\ref{arxivsupp:main:hessbetadiff:IVR}) hold, and furthermore $\frac{1}{n}\Wb^T \Wb = \bSigma_{\bW\bX,n}$ satisfies the CS condition. Thus on the event $E$, we have:
\begin{align*}
\frac{\|\hat \bbeta_{S} - \bbeta_S^*\|_1\operatorname{CS}_{\kappa}(s,1)}{s}& \leq \|\bSigma_{\bW\bX,n} (\bbeta^* - \hat \bbeta)\|_{\infty} \leq 2\lambda.
\end{align*}
To get (\ref{arxivsupp:main:betaL1:IVR}), note that by (\ref{arxivsupp:main:betadiff:IVR}) we have: $\|\hat \bbeta -  \bbeta^*\|_1 \leq 2 \|\hat \bbeta_{S} -  \bbeta^*_{S}\|_1 \leq \frac{4 AK \lambda s}{\operatorname{CS}_{\kappa}(s,1)},$
and we are done. 
\end{proof}

\begin{lemma}\label{conc:poly:subGauss:IVR} Let $\{X_i\}_{i = 1}^n$ be an i.i.d. collection of sub-Gaussian random variables satisfying $\|X_i\|_{\psi_2} \leq K$. Then we have:
$$
\PP\Big (|n^{-1} \sum_{i = 1}^n X_i^4 - \EE X_i^4| \geq t\Big ) \leq  \frac{ \tilde K_4^k [ \sqrt{k/n} + k^2/n]^k}{t^k},
$$
for any $k \in \NN$ and some fixed constant $\tilde K_4$ depending solely on $K$.
\end{lemma}

\begin{proof} We make usage of Theorem 1.4 of \cite{adamczak2015concentration}, which provides a convenient concentration bound for higher moments of sub-Gaussian random variables. Using this result it is not hard to check that: $
[\EE |n^{-1} \sum_{i = 1}^n X_i^4 - \EE X_i^4|^k]^{1/k} \leq \tilde K_4 [\sqrt{k/n} + k^2/n],$
where $\tilde K_4$ is a constant depending solely on $K$. Consequently, applying Markov's inequality we obtain the final conclusion.
\end{proof}

\section{Proofs for Graphical Models} 

\subsection{Proofs for Graphical Models with CLIME}\label{arxivsupp:main:edgetestproofsSEC}

\begin{proof}[Proof of Corollary \ref{IFCLIME}] Before we proceed with the proof note that we are guaranteed to have $\|\vb^*\|_1 \geq (\bSigma_{\bX}^{-1})_{11} \geq (\bSigma_{\bX,11})^{-1} \geq (2K_{\bX}^{2})^{-1} > 0$, and similarly $\|\bbeta^*\|_1 \geq  (2K_{\bX}^{2})^{-1} > 0$. Hence, $\max(s_{\vb}\|\vb^*\|_1, s\|\bbeta^*\|_1)\|\vb^*\|_1 \|\bbeta^*\|_1  \log d \log(nd)/n  = o(1)$ implies that 
$$\max(s_{\vb},s)\|\vb^*\|_1\|\bbeta^*\|_1\log d/\sqrt{n} = o(1).$$

We show this result by verifying the conditions of Section \ref{masterthm:sec}. To see Assumption (\ref{consistencyassumpweakcn}), we can use Lemma \ref{arxivsupp:main:vdiff} to argue that $\|\hat\bbeta - \bbeta^*\|_1 = O_p\left(\|\bbeta^*\|_1s\sqrt{\log d/n}\right)$, $\|\hat\vb - \vb^*\|_1 = O_p\left(\|\vb^*\|_1s_{\vb}\sqrt{\log d/n}\right)$ provided that $\lambda$ and $\lambda'$ are large enough. 

Next we check Assumption \ref{noiseassumpCI}. To see (\ref{betastartassumpCI}), fix a $|\theta - \theta^*| < \epsilon$, for some $\epsilon > 0$. By the triangle inequality:
$$
\|\bSigma_n\bbeta^*_{\theta} - \bSigma_{\bX}\bbeta_{\theta}^*\|_{\infty} \leq \|\bSigma_{n} - \bSigma_{\bX}\|_{\max}(\|\bbeta\|_1^* + \epsilon).
$$
The RHS is $O_p\bigg(\|\bbeta\|_1^*\sqrt{\log d/n}\bigg)$, by Lemma \ref{arxivsupp:main:samplecovpopcov}. The same logic shows that $\bigg|\vb^{*T}\bSigma_n\bbeta^*_{\theta} - \vb^{*T}\bSigma_{\bX}\bbeta_{\theta}^* \bigg| = O_p\bigg((\|\bbeta^*\|_1 + \epsilon)\|\vb^*\|_{1}\sqrt{\log d/n}\bigg)$, which implies (\ref{betastartassumpCIvstar}). Since the Hessian $\Tb$ in (\ref{lambdaprimeasumpCI}) is free of $\bbeta$ we are allowed to set $r_3(n) = \lambda' \asymp \|\vb^*\|_1\sqrt{\log d/n} = o(1)$ (by Lemma \ref{arxivsupp:main:vdiff}). Finally the two expectations in Assumption \ref{noiseassumpCI}, are bounded as we see below:
\begin{align*}
\|\bSigma_{\bX}\bbeta^*_{\theta} - \be_m^T\|_{\infty} = \|\bSigma_{\bX}(\bbeta^*_{\theta} - \bbeta^*)\|_{\infty} \leq \|\bSigma_{\bX,*1} \|_{\infty}\epsilon \leq 2 K_X^2 \epsilon, ~~~~~~~~ \|\vb^{*T}\bSigma_{\bX, -1}\|_{\infty}  = 0.
\end{align*}
By adding up the following two identities:
\begin{align*}
\sqrt{n} O_p\left(\|\vb^*\|_1 \sqrt{\log d/n}\right)O_p\left(\|\bbeta^*\|_1s\sqrt{\log d/n}\right) & = o_p(1),\\
\sqrt{n} O_p\left( \|\vb^*\|_1s_{\vb}\sqrt{\log d/n}\right) O_p\left(\|\bbeta^*\|_1\sqrt{\log d/n}\right) & = o_p(1),
\end{align*}
we get that (\ref{assumpone}) is also valid in this case after a usage of (\ref{CLIMEratecond}). 

To verify the consistency of $\tilde \theta$ we check the assumptions in Theorem \ref{consistency:sol}. Clearly the map $\vb^{*T}\bSigma_{\bX} (\bbeta^*_{\theta} - \bbeta^*) = (\theta - \theta^*)$ has a unique $0$ when $\theta = \theta^*$. Moreover, the map $\theta \mapsto \hat \vb^T(\bSigma_{n} \hat\bbeta_{\theta} - \be_m^T)$ is continuous as it is linear. In addition, it has a unique zero except in cases when $\hat \vb^T \bSigma_{n,*1} = 0$. However note that $|\hat \vb^T \bSigma_{n,*1} - 1| \leq \lambda'$ by (\ref{lambda:prime:clime}), and hence for small enough values of $\lambda'$ there will exist a unique zero.

Assumption \ref{CLTcond} is verified in Lemma \ref{arxivsupp:main:normalityCLIME}. Observe that (\ref{stabtwo}) is trivial as its LHS $\equiv 0$ in this case. Finally, the fact that $\hat \Delta$ is consistent for $\Delta$ is verified in Lemma \ref{arxivsupp:main:pluginCLIME}.
\end{proof}

Next, we proceed to formulate a uniform weak convergence result. To this end, for fixed $M > \delta > 0$, define the following parameter space of covariance matrices:
$$
\mathcal{S}(L, s) = \{\bSigma  : \bSigma = \bSigma^T, 0 < \delta \leq \lambda_{\min}(\bSigma),  \|\bSigma\|_{\max} \leq M, \|\bSigma^{-1}\|_1 \leq L, \max_{j=1,...,d} \|\bSigma^{-1}_{*j}\|_0 \leq s\}.
$$
We have the following result in terms of uniform convergence:

\begin{corollary}\label{arxivsupp:main:unifconvCLIME} Let (\ref{subGausserr:CLIME}) holds, and $\Cov(\bX) = \bSigma_{\bX} \in \mathcal{S}(L, s)$.  Let $\bOmega = (\bSigma_{\bX})^{-1}$ and denote $\bbeta = \bOmega_{*m}$, and $\vb = \bOmega_{*1}$. Assume there exist two constants $V_{\min}$ and $V_{\max}$ such that:
\begin{align}
\Var({\vb}^{T}\bX^{\otimes 2}\bbeta) \geq V_{\min} > 0, ~~~ \EE(\vb^T \bX^{\otimes 2} \bbeta)^4 \leq V_{\max} < \infty.  \label{arxivsupp:main:momentass} 
\end{align}
Then under the following conditions:
\begin{align}
s L^3 \log(d) \log(nd)/n = o(1), ~~~ s^3/\sqrt{n} = o(1)\label{arxivsupp:main:ass1unifconvCLIME},%\\
\end{align}
%{\color{red} CHANGE!}
we have $
\lim_{n \rightarrow \infty} \sup_{\bSigma_{\bX} \in \mathcal{S}(L,s)} \sup_{t\in\RR} |\PP_{\bbeta}(\hat U_n \leq t ) - \Phi(t)| = 0.$
\end{corollary}

The conditions in this Corollary are essentially the same as those in Corollary \ref{IFCLIME}. Condition $\EE(\vb^T \bX^{\otimes 2} \bbeta)^4 \leq V_{\max}$ ensures $\EE(\vb^T \bX^{\otimes 2} \bbeta)^2 < \infty$ and $\Var((\vb^{T}\bX^{\otimes 2}\bbeta^{})^2) = o({n})$. The only other difference is that the second condition in (\ref{arxivsupp:main:ass1unifconvCLIME}) is stronger than the counterpart in Corollary \ref{IFCLIME}. We need this condition to apply the Berry€-Esseen theorem to control the normal approximation error uniformly.

\begin{lemma}\label{arxivsupp:main:normalityCLIME} Under the assumptions of Corollary \ref{IFCLIME} we have that: $
\Delta^{-1/2} n^{1/2} S(\bbeta^*) \rightsquigarrow N(0,1).$
\end{lemma}

\begin{proof}[Proof of Lemma \ref{arxivsupp:main:normalityCLIME}] Similarly to Lemma \ref{arxivsupp:main:CLTDantzig} we will verify Lyapunov's condition for the CLT. It suffices to bound the quantity for some $k > 2$: 
$$\frac{n^{-k/2}}{\|\vb^*\|^k_2\|\bbeta^*\|^k_2 \alpha_{\min}^{k/2}\\} \sum_{i=1}^n \EE \left|\vb^{*T}\bX_i^{\otimes 2}\bbeta^* - {\vb}^{*T}\bSigma_{\bX}\bbeta^* \right|^k,$$
where we used assumption Assumption \ref{varXXTassump}. By Cauchy-Schwartz we can bound the above expression by following (up to a constant factor): $
n^{-k/2} \sum_{i=1}^n \EE \left\|(\bX_i^{\otimes 2} - \bSigma_{\bX})_{S_{\vb},S}\right\|_{F}^k,$
where by subscripting the matrix we mean setting all elements not in the supports of $\vb^*$ or $\bbeta^*$ ($S_{\vb}$, and $S$ correspondigly) to 0, and $\|\cdot \|_{F}$ is the Frobenius norm of the matrix. Finally using Lemma \ref{arxivsupp:main:XXTbound} we conclude that we can control the expression above by $n^{-k/2 + 1}(s_{\vb}s)^{k/2} (8k K_{\bX})^k$, and hence the conclusion follows.
\end{proof}

\begin{remark} Using the Berry-Esseen theorem for non-identical random variables in combination with the bound we derived above, we can further show:
$$
\sup_t \left|\PP^* \left(\Delta^{-1/2}n^{1/2} S(\bbeta^*) \leq t\right) - \Phi(t)\right| \leq C_{BE} n^{-1/2}(24K_{\bX})^3 (s_{\vb} s)^{3/2}  = o(1),
$$
where $C_{BE}$ is an absolute constant.
\end{remark}

\begin{lemma}\label{arxivsupp:main:pluginCLIME} Under the assumptions from Corollary \ref{IFCLIME}, we have that the plugin estimator  $\hat \Delta \stackrel{P}{\rightarrow}\Delta$.
\end{lemma}

\begin{proof}[Proof of Lemma \ref{arxivsupp:main:pluginCLIME}] 
Note that $
\hat \Delta = \frac{1}{n}\sum_{i = 1}^n (\hat \vb^T \bX_i^{\otimes 2} \hat \bbeta)^2 - (\hat \vb^T \bSigma_n \hat \bbeta)^2.$
Similarly to the analysis of the first term in Proposition \ref{arxivsupp:main:pluginDantzig} one can show that $\hat \vb^T \bSigma_n \hat \bbeta$ is consistent for $\vb^{*T}\bSigma_{\bX} \bbeta^*$ under our assumptions. Hence it suffices to show that $|\frac{1}{n}\sum_{i = 1}^n (\hat \vb^T \bX_i^{\otimes 2} \hat \bbeta)^2 - \EE(\vb^{*T} \bX^{\otimes 2} \bbeta^*)^2| = o_p(1)$.
We start by arguing that $
 \underbrace{n^{-1} \bigg[\sum_{i = 1}^n (\hat \vb^T \bX_i^{\otimes 2}\hat \bbeta  -  \vb^{*T} \bX_i^{\otimes 2} \bbeta^*)^2\bigg]}_{I},$
is asymptotically negligible. 
$$
I^{1/2} \leq \underbrace{  \bigg[n^{-1}\sum_{i = 1}^n (\hat \vb^T \bX_i^{\otimes 2}(\hat \bbeta  -  \bbeta^*))^2\bigg]^{1/2}}_{I_{1}} + \underbrace{\bigg[n^{-1}\sum_{i = 1}^n ((\hat \vb - \vb^*)^T \bX_i^{\otimes 2} \bbeta^*)^2\bigg]^{1/2}}_{I_{2}}.
$$
We first handle $
I_1^2 =   (\hat \bbeta - \bbeta^*)^T \underbrace{n^{-1}\sum_{i = 1}^n \bX_i^{\otimes 2} \hat \vb  \hat \vb^T  \bX_i^{\otimes 2}}_{\Mb}  (\hat \bbeta - \bbeta^*).$
Using Lemma \ref{arxivsupp:main:Mbound}, we can bound $\Mb$ in the following way:
\begin{align}
\|\Mb\|_{\max} \leq \max \|\bX_i^{\otimes 2}\|_{\max} n^{-1} \sum_{i = 1}^n \hat \vb^T \bX_i^{\otimes 2} \hat \vb \leq O_p(\log(nd))  \|\hat \vb\|_1 \|\hat \bSigma \hat \vb\|_{\infty}. \label{arxivsupp:main:MboundpropCLIME}
\end{align}
By the definition of $\hat \vb$ we have:  $ \|\hat \bSigma \hat \vb\|_{\infty} \leq (1 + \lambda')$. Hence:
\begin{align*}
\|\Mb\|_{\max} & \leq O_p(\log(nd)) (\|\vb^*\|_1 + \|\hat \vb - \vb^*\|_1) (1 + \lambda') = O_p(\log(nd))\|\vb^*\|_1,
\end{align*}
where we used that $\lambda' \asymp \|\vb^*\|_1 \sqrt{\log d/n}$ and $\|\hat \vb - \vb^*\|_1 = O_p(\|\vb^*\|_1 s_{\vb}\sqrt{\log d /n})$ which are quantities going to $0$ under our assumptions. Thus:
$$
I^2_{1} \leq  \|\hat \bbeta - \bbeta^* \|_1^2 \|\vb^*\|_1 O_p(\log(nd)) =  O_p\left(\|\vb^*\|_1\|\bbeta^*\|^2_1s^2\log d/n \log(nd) \right) = o_p(1),
$$
By a similar argument we can show that $I_{2}$ is of similar order. Putting everything together we conclude: 
\begin{align}\label{arxivsupp:square:rate}
I = O_p\left(\max(s_{\vb}^2\|\vb^*\|_1, s^2 \|\bbeta^*\|_1) \|\vb^*\|_1 \|\bbeta^*\|_1\log d \log(nd)/n \right) = o_p(1).
\end{align}
Next, we argue that $n^{-1}\sum_{i = 1}^n (\vb^{*T}\bX_i^{\otimes 2}\bbeta^*)^2 - \EE(\vb^{*T}\bX^{\otimes 2}\bbeta^*)^2$ is small. Recall that we are assuming $\Var((\vb^{*T} \bX^{\otimes 2} \bbeta^*)^2) = o(n)$. A usage of Chebyshev's inequality shows that $n^{-1}\sum_{i = 1}^n (\vb^{*T}\bX_i^{\otimes 2}\bbeta^*)^2- \EE(\vb^{*T}\bX^{\otimes 2}\bbeta^*)^2 = o_p(1)$. Finally note that by the triangle inequality the following two inequalities hold:
\begin{align*}
\bigg[n^{-1} \sum_{i = 1}^n (\hat \vb^T \bX_i^{\otimes 2}\hat \bbeta)^2 \bigg]^{1/2}& \leq \bigg[n^{-1} \sum_{i = 1}^n(\vb^{*T} \bX_i^{\otimes 2} \bbeta^*)^2\bigg]^{1/2} + I^{1/2},\\
\bigg[n^{-1} \sum_{i = 1}^n(\vb^{*T} \bX_i^{\otimes 2} \bbeta^* )^2\bigg]^{1/2}& \leq \bigg[n^{-1} \sum_{i = 1}^n (\hat \vb^T \bX_i^{\otimes 2}\hat \bbeta)^2 \bigg]^{1/2}+ I^{1/2}.
\end{align*}
Observe that $n^{-1}\sum_{i = 1}^n (\vb^{*T} \bX_i^{\otimes 2} \bbeta^*)^2 = \EE( \vb^{*T} \bX^{\otimes 2} \bbeta^* )^2 + o_p(1) = O_p(1)$. This completes the proof.
\end{proof}

%\begin{proof}[Proof of Remark \ref{arxivsupp:main:pluginest2CLIME}] We have that:
%$$
%\frac{1}{n}\sum_{i = 1}^n (\hat \vb^T \bX_i^{\otimes 2}\hat \bbeta - \hat \vb^T \eb_m^T)^2 = \underbrace{\frac{1}{n}\sum_{i = 1}^n (\hat \vb^T \bX_i^{\otimes 2}\hat \bbeta )^2}_{I_1} - 2 \underbrace{(\hat \vb^T \bSigma_n \hat \bbeta)  \hat \vb^T \eb_m^T}_{I_2} + \underbrace{(\hat \vb^T \eb_m^T)^2}_{I_3}.
%$$
%As a consequence of the proof of Proposition \ref{arxivsupp:main:pluginCLIME}, we have that $I_1 \rightarrow_p  \EE(\vb^{*T} \bX^{\otimes 2} \bbeta^*)^2$, also that $\hat \vb^T \bSigma_n \hat \bbeta \rightarrow_p  \vb^{*T} \bSigma_{\bX}  \bbeta^*$. Thus, with the help of the continuous mapping theorem, all it remains to show is that: $\hat \vb^T \eb_m^T$ is consistent for $\vb^{*T} \eb_m^T$. However this follows from:
%$$
%|\hat \vb^T \eb_m^T -  \vb^{*T} \eb_m^T| \leq \|\hat \vb - \vb^*\|_1 = o_p(1),
%$$
%by Lemma \ref{arxivsupp:main:vdiff}. This completes the proof.
%\end{proof}
\begin{proof}[Proof of Corollary \ref{arxivsupp:main:unifconvCLIME}] To prove this corollary note that all bounds we showed in the proof of Corollary \ref{IFCLIME} hold uniformly in the parameter set $\mathcal{S}(L, s)$. Note that as both $\vb$ and $\bbeta$ are columns of $\bOmega$ we have that $\|\vb\|_0, \|\bbeta\|_0 \leq s, \|\vb\|_1, \|\bbeta\|_1 \leq L$ and $M^{-1} \leq \|\vb\|_2, \|\bbeta\|_2 \leq \delta^{-1}$. These conditions in conjunction with the assumptions of the present result, can be seen to imply the conditions from Section \ref{arxivsupp:main:UWCUNsec} and this completes the proof.
\end{proof}
\begin{lemma}  \label{arxivsupp:main:XXTbound} Let $R_{\vb}, R \subset \{1,\ldots,d\}$ with $|R_{\vb}| = r_{\vb}, |R| = r$. Then we have the following $
\EE \left\|(\bX^{\otimes 2} - \bSigma)_{R_{\vb},R}\right\|_{F}^k \leq (r_{\vb}r)^{k/2}(8k K_{\bX}^2)^k.$
\end{lemma}

\begin{proof} 
Apply (\ref{arxivsupp:main:subexpo}) to get $\|X_i X_j\|_{\psi_1} \leq 2 K_{\bX}^2$. Combined with the triangle inequality it gives us $\|X_i X_j - \sigma_{ij}\|_{\psi_1} \leq 4 K_{\bX}^2$. Next simply use Lemma \ref{arxivsupp:main:simple:proof:exp:psi1} for $\psi_1$ to complete the proof.
\end{proof}

\subsection{Proofs for Transelliptical Graphical Models} \label{arxivsupp:main:proofsforTGM}

In the transelliptical case the lemmas from the CLIME case are no longer applicable, as the estimator of $\bSigma$ is constructed in a completely different manner. Furthermore, the vector $\bX$ is coming from a nonparanormal family and need not be sub-Gaussian. Fortunately, \cite{liu2012high} provide a concentration result which we state below:

\begin{theorem}[Liu 2012]\label{arxivsupp:main:kendaltaucovthm} For any $n > 1$ with probability at least $1 - 1/d$, we have
\begin{align}
\|\hat \bS^\tau - \bSigma\|_{\max} \leq 2.45 \pi \sqrt{\log d/n} \label{arxivsupp:kendaltaucovrate}.
\end{align}
\end{theorem}
While this theorem is defined within the framework of nonparanormal models, the proof doesn't utilize the fact that the family is nonparanormal, and thus extends to the transelliptical case. As we can see from the theorem, the rate of Kendall's tau estimate (\ref{arxivsupp:kendaltaucovrate}), is no different than the one using the sample covariance matrix, provided in Lemma \ref{arxivsupp:main:samplecovpopcov}.

\begin{proof}[Proof of Corollary \ref{arxivsupp:main:skepticnormality}] The proof of this result is the same as Corollary \ref{IFCLIME}, except we use $\bSigma$ in place of $\bSigma_{\bX}$, $\hat \bS^{\tau}$ in place of $\bSigma_n$, Lemma \ref{arxivsupp:main:vdiffkendaltau} in place of Lemma \ref{arxivsupp:main:vdiff} and Theorem \ref{arxivsupp:main:kendaltaucovthm} in place of Lemma \ref{arxivsupp:main:samplecovpopcov}. The only different step is the verification of Assumption \ref{CLTcond} which we provide in Lemma \ref{arxivsupp:main:CLTcondTGM}. We omit the rest of the details.
\end{proof}
\begin{lemma} \label{arxivsupp:main:CLTcondTGM} Under the assumptions of Corollary \ref{arxivsupp:main:skepticnormality} we have that $
\Delta^{-1/2} n^{1/2} S(\bbeta^*) \rightsquigarrow N(0,1),$
where $\Delta$ is defined as in (\ref{arxivsupp:main:deltaSKEPTICdef}).
\end{lemma}

\begin{proof}[Proof of Lemma \ref{arxivsupp:main:CLTcondTGM}] Note that by the mean value theorem we have the following representation:
\begin{align*}
n^{1/2}  \vb^{*T} \left(\hat \bS^\tau \bbeta^* -\eb_m^T\right) & = n^{1/2}  \vb^{*T} \left(\hat \bS^\tau  - \bSigma \right)\bbeta^* = n^{1/2} \sum_{\substack {j \in S_{\vb}, k \in S\\ j \neq k}} v^{*}_j \beta^*_k \left(\sin \left(\hat \tau_{jk}\frac{\pi}{2}\right) - \sin \left(\tau_{jk}\frac{\pi}{2}\right)\right)\\
& = n^{1/2} \sum_{\substack {j \in S_{\vb}, k \in S\\ j \neq k}} v^{*}_j \beta^*_k \cos \left(\tau_{jk}\frac{\pi}{2}\right)\frac{\pi}{2}  \left(\hat \tau_{jk} - \tau_{jk}\right)\\
& - \frac{n^{1/2}}{2} \sum_{\substack {j \in S_{\vb}, k \in S\\ j \neq k}} v^{*}_j \beta^*_k \sin \left(\tilde \tau_{jk}\frac{\pi}{2}\right)\left(\frac{\pi}{2}  \left(\hat \tau_{jk} - \tau_{jk}\right)\right)^2, 
\end{align*}
where $\tilde \tau_{ij}$ is a number between $\hat \tau_{ij}$ and $\tau_{ij}$. We will first deal with the first term in the sum above. Since this term is a linear combination of second order (dependent) $U$-statistics, we will make usage of  H\'{a}ejk's projection method. A similar approach was used in the celebrated paper of \cite{hoeffding1948class}. To this end we define the following notations:
\begin{align*}
\tau_{jk}^{ii'} & = \sign \left((X_{ij} - X_{i'j} )(X_{ik} - X_{i'k}) \right)-  \tau_{jk}, ~~~~ \tau_{jk}^{ii'|i}  = \EE[\tau_{jk}^{ii'} | \bX_i],\\
\tau_{jk}^{i} & = \frac{1}{n-1}\sum_{i' \neq i}\tau_{jk}^{ii'|i}, ~~~~~~~~~~~~~~~~~~~~~~~~~~~~~~~~~ w^{ii'}_{jk} = \tau_{jk}^{ii'} - \tau_{jk}^{ii'|i} - \tau_{jk}^{ii'|i'},
\end{align*}
In terms of these notations we therefore have $
\hat \tau_{jk} - \tau_{jk} =  \frac{2}{n} \sum_{i = 1}^n \tau_{jk}^{i} + \frac{2}{n(n-1)}\sum_{1 \leq i < i' \leq n} w^{ii'}_{jk}.$
This gives us the following identity:
\begin{align*}
 n^{1/2} \sum_{\substack {j \in S_{\vb}, k \in S\\ j \neq k}} v^{*}_j \beta^*_k \cos \left(\tau_{jk}\frac{\pi}{2}\right)\frac{\pi}{2}  \left(\hat \tau_{jk} - \tau_{jk}\right) =   \underbrace{\pi n^{-1/2} \sum_{\substack {j \in S_{\vb}, k \in S\\ j \neq k}} v^{*}_j \beta^*_k \cos \left(\tau_{jk}\frac{\pi}{2}\right) \sum_{i = 1}^n \tau_{jk}^{i}}_{I_1} & \\
  + \underbrace{\frac{\pi}{n^{1/2} (n-1)}\sum_{\substack {j \in S_{\vb}, k \in S\\ j \neq k}} v^{*}_j \beta^*_k \cos \left(\tau_{jk}\frac{\pi}{2}\right)\sum_{1 \leq i < i' \leq n} w_{jk}^{ii'}}_{I_2}&. 
\end{align*}
We first deal with $I_1$ which can clearly be represented as a sum of iid mean 0 terms, by verifying Lyapunov's condition for the CLT. 
$I_1$ can be rewritten as:
\begin{align}
I_1 = {n^{-1/2}}  \sum_{i = 1}^n \underbrace{\sum_{\substack {j \in S_{\vb}, k \in S\\ j \neq k}} v^{*}_j \beta^*_k \pi\cos \left(\tau_{jk}\frac{\pi}{2}\right)\tau_{jk}^{i}}_{M_i}. \label{arxivsupp:main:sumofiid}
\end{align}
Construct the matrix $\bTheta^i \in \RR^{d \times d}$ given entrywise by 
\begin{align} \label{arxivsupp:main:Thetaidef}
\Theta^i_{jk} =  \pi\cos \left(\tau_{jk}\frac{\pi}{2}\right) \tau_{jk}^{i}, \mbox{ hence } \Theta^i_{jj} = 0.
\end{align}
We can then rewrite $M_i = \vb^{*T} \bTheta^i \bbeta^* = \vb^{*T}_{S_{\vb}} \bTheta^i_{S_{\vb},S} \bbeta^*_{S}$. Calculating the variance of $M_i$ gives $
\Var(M_i) = \EE(\vb^{*T} \bTheta^i \bbeta^*)^2 \geq \alpha_{\min} \|\vb^*\|^2_2 \|\bbeta^*\|_2^2$, where the inequality follows by assumption. We proceed to verify Lyapunov's condition for some $k > 2$ (where we ignore the constant $\alpha_{\min} > 0$):
$$
\frac{n^{-k/2}}{\|\vb^*\|^k_2 \|\bbeta^*\|_2^k} \sum_{i = 1}^n \EE|M_i|^k \leq n^{-k/2} \sum_{i = 1}^n \EE\|\bTheta^i_{S_{\vb}, S}\|_F^k \leq \frac{(s_{\vb}s)^{k/2} (2 \pi)^k}{n^{k/2 - 1}} = o(1),
$$
where the first inequality follows from Cauchy-Schwartz, and to see the second one notice that each element of $\Theta^i$ is bounded $|\Theta^i_{jk}| \leq 2 \pi$, and hence $\|\bTheta^i_{S_{\vb}, S}\|_F^k \leq (s_{\vb}s)^{k/2} (2 \pi)^k$. %Thus finally:
%$$
%\frac{n^{-k/2}}{\|\vb^*\|^k_2 \|\bbeta^*\|_2^k} \sum_{i = 1}^n \EE|M_i|^k \leq \frac{(s_{\vb}s)^{k/2} (2 \pi)^k}{n^{k/2} - 1} = o(1)
%$$
This implies that $I_1\rightsquigarrow N(0, \Delta)$, with $\Delta = \EE(\vb^{*T} \bTheta^i \bbeta^*)^2$.

Next we deal with the second term $I_2$, which is mean $0$, by showing that its (standardized) variance goes to $0$ asymptotically. Before we compute its variance we make several preliminary calculations:
\begin{align*}
\EE(w_{jk}^{ii'}w_{ls}^{rr'}) & = \EE(\tau^{ii'}_{jk}\tau^{rr'}_{ls}) - \EE(\tau^{ii'}_{jk}\tau^{rr'|r}_{ls}) - \EE(\tau^{ii'}_{jk}\tau^{rr'|r'}_{ls}) - \EE(\tau^{ii'|i}_{jk}\tau^{rr'}_{ls}) + \EE(\tau^{ii'|i}_{jk}\tau^{rr'|r}_{ls}) \\
& + \EE(\tau^{ii'|i}_{jk}\tau^{rr'|r'}_{ls}) - \EE(\tau^{ii'|i'}_{jk}\tau^{rr'}_{ls}) + \EE(\tau^{ii'|i'}_{jk}\tau^{rr'|r}_{ls}) + \EE(\tau^{ii'|i'}_{jk}\tau^{rr'|r'}_{ls}).
\end{align*}
In the expression above we have taken $j \neq k$, $l \neq s$, $r \neq i \neq i' \neq r \neq r' \neq i$. Notice now that all elements above are independent and since $\EE(w_{jk}^{ii'}) = \EE(w_{ls}^{rr'})  = 0$, we conclude that $\EE(w_{jk}^{ii'}w_{ls}^{rr'}) = 0$. Following, the same logic, for $j \neq k$, $l \neq s$, $i \neq i' \neq r' \neq i$:
\begin{align*}
\EE(w_{jk}^{ii'}w_{ls}^{ir'}) & = \EE(\tau^{ii'}_{jk}\tau^{ir'}_{ls}) - \EE(\tau^{ii'}_{jk}\tau^{ir'|i}_{ls}) - \EE(\tau^{ii'|i}_{jk}\tau^{ir'}_{ls}) + \EE(\tau^{ii'|i}_{jk}\tau^{ir'|i}_{ls}) 
\end{align*}
where all the rest terms are 0, by the same argument as in the first case. Using iterated expectation by conditioning on $\bX_i$ it can be easily seen that all terms become equal to --- $\EE(\tau_{jk}^{ii'|i}\tau_{jk}^{ir'|i})$, and we can conclude that $\EE(w_{jk}^{ii'}w_{ls}^{ir'}) = 0$. Since $\EE I_2 = 0$, we have:
\begin{align*}
\frac{\Var(I_2)}{\Var(M_i)} & \leq \frac{\EE(I_2^2)}{\alpha_{\min} \|\vb^*\|^2_2 \|\bbeta^*\|_2^2 } \\
& = \frac{\pi^2}{\alpha_{\min} \|\vb^*\|^2_2 \|\bbeta^*\|_2^2n (n-1)^2} \sum_{1 \leq i < i' \leq n} \EE \left(\sum_{\substack {j \in S_{\vb}, k \in S\\ j \neq k}} v^{*}_j \beta^*_k \cos \left(\tau_{jk}\frac{\pi}{2}\right) w_{jk}^{ii'}\right)^2\\
& \leq \frac{\pi^2{n \choose 2}36 \Big(\sum_{j \in S_{\vb}} |v_j^*|\Big)^2\Big(\sum_{k \in S} |\beta_k^*|\Big)^2}{n (n-1)^2 \alpha_{\min}\|\vb^*\|^2_2 \|\bbeta^*\|_2^2} \leq \frac{\pi^218 s_{\vb}s}{(n-1)\alpha_{\min}} = o(1),
\end{align*}
where in the next to last inequality we used the trivial bound $|w^{ii'}_{jk}| \leq | \tau_{jk}^{ii'} | + |\tau_{jk}^{ii'|i}| +|\tau_{jk}^{ii'|i'}| \leq 6$. Thus the term $\frac{\Var(I_2)}{\Var(M_i)} = o(1)$ and therefore, Chebyshev's inequality gives us that $\frac{I_2}{\sqrt{\Var(M_i)}} = o_p(1)$.

Finally we deal with the standardized version of the last term:
\begin{align} \label{arxivsupp:main:lastterm}
\frac{1}{\sqrt{\Var(\vb^{*T} \bTheta \bbeta^*)}}\frac{n^{1/2}}{2} \sum_{\substack {j \in S_{\vb}, k \in S\\ j \neq k}} v^{*}_j \beta^*_k \sin \left(\tilde \tau_{jk}\frac{\pi}{2}\right)\left(\frac{\pi}{2}  \left(\hat \tau_{jk} - \tau_{jk}\right)\right)^2.
\end{align}
As we mentioned previously it's clear that $\hat \tau_{jk}$ is a $U$-statistic, and its kernel is a bounded function (between $-1$ and $1$). Furthermore, we have that $\EE \hat \tau_{jk} = \tau_{jk}$. Thus, we can apply Hoeffding's inequality for $U$-statistics (see \cite{hoeffding1963probability} equation (5.7)), to obtain that:
\begin{align}\label{arxivsupp:main:hoeffdingustat}
\PP(\sup_{jk} |\hat \tau_{jk} - \tau_{jk}| > t) \leq 2 d^2 \exp\left(-\frac{n t^2}{4}\right).
\end{align}
It follows that selecting $t=9\sqrt{\log d/n}$ suffices to keep the probability going to 0. Notice that the (\ref{arxivsupp:main:lastterm}) can be controlled by:
$$
\frac{n^{1/2}\pi^2 \sqrt{s_{\vb} s}}{8\theta_{\min}\|\vb^{*}\|_2\|\bbeta^{*}\|_2} \|\vb^{*}\|_2\|\bbeta^{*}\|_2\sup_{jk} (\hat \tau_{jk} - \tau_{jk})^2 = O_p\left(\frac{\sqrt{s_{\vb} s}\log d}{n^{1/2}}\right) = o_p(1).
$$
The last equation is implied by our assumption. This concludes the proof.
\end{proof}

\begin{remark} Using the Berry-Esseen theorem for non-identical random variables we can strengthen weak convergence statement to $
\sup_t \left|\PP^* \left(\frac{I_1}{\sqrt{\Delta}}  \leq t\right) - \Phi(t)\right| \leq C_{BE} n^{-1/2} (s_{\vb} s)^{3/2}  = o(1),$
where $C_{BE}$ is an absolute constant. Note that we decomposed our test into $\frac{I_1}{\sqrt{\Delta}} + o_p(1)$, and hence this statement is valid more generally for Corollary \ref{arxivsupp:main:skepticnormality}.
\end{remark}

\begin{proposition} \label{arxivsupp:main:pluginSKEPTIC} Under the same assumptions as in Corollary \ref{arxivsupp:main:skepticnormality}, we have that $\hat \Delta \rightarrow_p \Delta$.
\end{proposition}

\begin{proof}[Proof of Proposition \ref{arxivsupp:main:pluginSKEPTIC}]
Before we go to the main proof, recall the definition of $\bTheta^i$ (\ref{arxivsupp:main:Thetaidef}), where $\Theta^i_{jk} =  \pi\cos \left(\tau_{jk}\frac{\pi}{2}\right) \tau_{jk}^{i}$. Note that in fact $\EE \bTheta^i = 0$, since $\EE \tau_{jk}^{i} = 0$, and thus $\Var(\vb^{*T}\bTheta^i\bbeta^*) = \EE(\vb^{*T}\bTheta^i\bbeta^*)^2$. Similarly one can note the simple identity: $\frac{1}{n} \sum_{i = 1}^n \hat \bTheta^i = 0$. Thus we will in fact focus on showing that $\frac{1}{n} \sum_{i = 1}^n (\hat \vb^T \hat \bTheta^i \hat \bbeta)^2$ is consistent for $\EE(\vb^{*T}\bTheta^i\bbeta^*)^2$.

Consider the following decomposition:
$$
\frac{1}{n} \sum_{i = 1}^n (\hat \vb^T \hat \bTheta^i \hat \bbeta)^2 = \underbrace{\frac{1}{n} \sum_{i = 1}^n [(\hat \vb^T \hat \bTheta^i \hat \bbeta)^2 - (\vb^{*T} \hat \bTheta^i \hat \bbeta)^2]}_{I_1} + \underbrace{\frac{1}{n} \sum_{i = 1}^n [(\vb^{*T} \hat \bTheta^i \hat \bbeta)^2 - (\vb^{*T} \hat \bTheta^i \bbeta^*)^2]}_{I_2}.
$$
Below we show that $I_1$ is asymptotically negligible. We have:
$$
|I_1| =  \bigg|(\hat \vb - \vb^{*})^T   \frac{1}{n} \sum_{i = 1}^n \hat \bTheta^i \hat \bbeta \hat \bbeta^T \hat \bTheta^i (\hat \vb + \vb^{*})^T\bigg| \leq  \|\hat \vb - \vb^*\|_1\|\hat \vb + \vb^*\|_1\|\hat \bbeta\|^2_1 (2\pi)^2,
$$
where we made use of $\|\hat \bTheta^i\|_{\max} \leq 2 \pi$. Thus by Lemma \ref{arxivsupp:main:vdiffkendaltau} 
$$|I_1| = O_p\left(\|\vb^*\|^2_1\|\bbeta^*\|^2_1 s_{\vb} \sqrt{\log d/n}\right) = o_p(1),$$
by assumption. Similarly we obtain $
|I_2| = O_p\left(\|\vb^*\|^2_1\|\bbeta^*\|^2_1 s \sqrt{\log d/n}\right) = o_p(1).$
Next, we inspect the following difference $
I_3 = \frac{1}{n} \sum_{i = 1}^n [(\vb^{*T} \hat \bTheta^i \bbeta^*)^2 - (\vb^{*T} \bTheta^i \bbeta^*)^2].$
Before we bound this term recall that we have the following useful inequality $\|\bTheta^i\|_{\max} \leq 2\pi$. Thus:
\begin{align} \label{arxivsupp:main:I3bound}
|I_3| \leq \|\vb^*\|_1^2 \|\bbeta^*\|_1^2 4\pi \max_{i = 1,\ldots,n}\|\hat \bTheta^i - \bTheta^i \|_{\max}.
\end{align}
To bound the difference $\max_{i = 1,\ldots,n}\|\hat \bTheta^i - \bTheta^i \|_{\max}$ we will use some concentration inequalities. First, since $\cos$ is Lipschitz with constant 1 --- $\left|\cos\left(\frac{\pi}{2} \hat \tau_{jk}\right) - \cos\left(\frac{\pi}{2}  \tau_{jk}\right)\right| \leq \frac{\pi}{2} |\hat \tau_{jk} - \tau_{jk}|$, we have:
\begin{align*}
|\hat \Theta^i_{jk} -  \Theta^i_{jk}| & \leq \pi \left| \cos\left(\frac{\pi}{2} \hat \tau_{jk}\right) - \cos\left(\frac{\pi}{2}  \tau_{jk}\right) \right||\hat \tau^i_{jk}| + \pi \left|\cos\left(\frac{\pi}{2}  \tau_{jk}\right) \right| |\hat \tau^i_{jk} - \tau^i_{jk}| \\
& \leq \pi^2 |\hat \tau_{jk} - \tau_{jk}| + \pi |\hat \tau^i_{jk} - \tau^i_{jk}|.
\end{align*}
where we used the simple observation that $|\hat \tau^i_{jk}| \leq 2$. Next we have:
\begin{align*}
& ~~~ |\hat \tau^i_{jk} - \tau^i_{jk}| \leq |\hat \tau_{jk} - \tau_{jk}| \\
& + \bigg| \underbrace{\frac{1}{n-1} \sum_{i' \neq i} \sign \left((X_{ij} - X_{i'j})(X_{ik} - X_{i'k})\right)}_{\hat \theta^i_{jk}} - \underbrace{\EE \left[\sign \left((X_{ij} - X_{i'j})(X_{ik} - X_{i'k})\right) | \bX_i\right] }_{\theta^i_{jk}}  \bigg|
\end{align*}
This gives us:
\begin{align} \label{arxivsupp:main:thetadiffbound}
|\hat \Theta^i_{jk} -  \Theta^i_{jk}| \leq (\pi^2 + \pi) |\hat \tau_{jk} - \tau_{jk}| + \pi |\hat \theta^i_{jk} - \theta^i_{jk}|.
\end{align}
Next, note since the terms in $\hat \theta^i_{jk}$ are iid conditional on $\bX_i$, and they are in the set $\{-1,1\}$ by Hoeffding's inequality, integrating $\bX_i$ out and the union bound we obtain:
%$$
%\PP(|\hat \theta^i_{jk} - \theta^i_{jk}| > t | \bX_i) \leq 2 \exp\left(- \frac{(n-1)t^2}{2}\right).
%$$
%It follows that the same inequality holds also unconditionally, and thus by the union bound:
%%$$
%\PP(|\hat \theta^i_{jk} - \theta^i_{jk}| > t ) \leq 2 \exp\left(- \frac{(n-1)t^2}{2}\right).
%$$
$$
\PP(\max_{i,j,k}|\hat \theta^i_{jk} - \theta^i_{jk}| > t ) \leq 2nd^2 \exp\left(- \frac{(n-1)t^2}{2}\right).
$$
This implies that selecting $t = 4\sqrt{\log(nd)/n}$, would keep the probability converging to 0. Combining this result with (\ref{arxivsupp:main:hoeffdingustat}) and (\ref{arxivsupp:main:thetadiffbound}) gives us $
\max_{i}\|\hat \bTheta^i -  \bTheta^i\|_{\max} = O_p\left(\sqrt{\log(nd)/n}\right).$
Hence using (\ref{arxivsupp:main:I3bound}), we get $
|I_3| = \|\vb^*\|_1^2 \|\bbeta^*\|_1^2 O_p\left(\sqrt{\log(nd)/n}\right) = o_p(1),$
which follows since $\|\vb^*\|^2_1\|\bbeta^*\|^2_1 \max(s_{\vb},s) \sqrt{\log d/n} = o(1)$, implies $\|\vb^*\|_1^2 \|\bbeta^*\|_1^2 \sqrt{\log(nd)/n} = o(1)$. To see the last implication it is sufficient to observe that $\|\vb^*\|_1 \geq \bOmega_{11} \geq 1, \|\bbeta^*\|_1 \geq \bOmega_{mm} \geq 1$. Finally we assess the difference $
\frac{1}{n}\sum_{i = 1}^n (\vb^{*T} \bTheta^i \bbeta^*)^2 - \EE (\vb^{*T} \bTheta^i \bbeta^*)^2.$
By Markov's inequality in much the same way as in the last part of the proof of Proposition \ref{arxivsupp:main:pluginCLIME}, we can show that the expression above is $o_p(1)$ if $\Var((\vb^{*T} \bTheta \bbeta^*)^2) = o({n})$. Since the elements of $\bTheta$ are bounded by $2\pi$ we have $|\vb^{*T} \bTheta \bbeta^*| \leq \|\vb^*\|_1 \|\bbeta^*\|_1  2\pi$. Hence since $\|\vb^*\|^2_1\|\bbeta^*\|^2_1 \max(s_{\vb},s) \sqrt{\log d/n} = o(1)$, implies $\|\vb^*\|^4_1 \|\bbeta^*\|^4_1 = o(n)$ the proof is complete. 
\end{proof}
\begin{proof}[Proof of Corollary \ref{arxivsupp:main:unifconvTGMCLIME}] Similarly to the proof of Corollary \ref{arxivsupp:main:unifconvCLIME} we simply need to note that our conditions imply the conditions required by Corollary \ref{arxivsupp:main:skepticnormality} and also note that the bounds in the proofs hold uniformly.
\end{proof}

\begin{lemma} \label{arxivsupp:main:kendaltaucovRE} Assume that the minimum eigenvalue  $\lambda_{\min}(\bSigma) > 0$ and $s \sqrt{\log d/n} \leq  (1 - \kappa) \frac{\lambda_{\min}(\bSigma_{\bX})}{(1 + \xi)^2 2.45 \pi}$, where $0 < \kappa < 1$. We then have that $\hat \bS^\tau$ satisfies the RE property with $\operatorname{RE}_{\hat \bS^\tau}(s, \xi) \geq \kappa \lambda_{\min} (\bSigma)$ with probability at least $1 - 1/d$.
\end{lemma}

\begin{proof}[Proof of Lemma \ref{arxivsupp:main:kendaltaucovRE}] Proof is the same as in Lemma \ref{arxivsupp:main:samplecovRE}, but we use Theorem \ref{arxivsupp:main:kendaltaucovthm} instead of Lemma \ref{arxivsupp:main:samplecovpopcov}. We omit the details.
\end{proof}

\begin{definition} Define $\operatorname{RE}_{\kappa}(s,\xi) := \kappa\operatorname{RE}_{\bSigma}(s,\xi) \geq \kappa \lambda_{\min}(\bSigma)$.
\end{definition}

\begin{lemma} \label{arxivsupp:main:vdiffkendaltau} Assume that --- $\lambda_{\min}(\bSigma) > 0$, $s_{\vb} \sqrt{\log d/n} \leq (1 - \kappa) \frac{\lambda_{\min}(\bSigma)}{(1 + 1)^2 2.45 \pi}$, where $0 < \kappa < 1$ and $\lambda' \geq  \|\vb^*\|_1 2.45 \pi \sqrt{\log d/n}$. Then we have that $\|\hat \vb - \vb^*\|_{1} \leq \frac{8 \lambda' s_{\vb}}{ \operatorname{RE}_{\kappa}(s_{\vb}, 1)} $ with probability at least  $1 - 1/d$.
\end{lemma}

\begin{proof}[Proof of Lemma \ref{arxivsupp:main:vdiffkendaltau}] Proof is the same as in Lemma \ref{arxivsupp:main:vdiff}, but we use Theorem \ref{arxivsupp:main:kendaltaucovthm} instead of Lemma \ref{arxivsupp:main:samplecovpopcov} and we use Lemma \ref{arxivsupp:main:kendaltaucovRE} instead of Lemma \ref{arxivsupp:main:samplecovRE}. We omit the details.
\end{proof}

\section{Proofs for the Linear Discriminant Analysis} \label{arxivsupp:main:proofsforLDPapp}

\begin{lemma}\label{arxivsupp:rem:LDP} Under Assumption \ref{var:bound:LDA} we have that $
\alpha V_1 + (1 - \alpha) V_2 \geq V'_{\min} (\|\bbeta^*\|_2^2 \|\vb^*\|_2^2 + \|\vb\|_2^2).$
\end{lemma}

\begin{proof}[Proof of Lemma \ref{arxivsupp:rem:LDP}] The proof follows by an elementary calculation so we omit the details. 
% To see this, first note that:
%$$
%\alpha V_1 + (1 - \alpha) V_2 = \Var(\vb^{*T} \bU^{\otimes 2} \bbeta^*) + \alpha^{-1}\EE(\vb^{*T} \bU)^2 + (1 - \alpha)^{-1}\EE(\vb^{*T} \bU)^2.
%$$
%Since we are assuming that $\vb^{*T}\EE \bU^{\otimes 2}\vb^* \geq \delta \|\vb^*\|^2_2$, we have:
%$$
%\alpha V_1 + (1 - \alpha) V_2 \geq  \Var(\vb^{*T} \bU^{\otimes 2} \bbeta^*)  + \delta(\alpha^{-1} +  (1 - \alpha)^{-1} )\|\vb^*\|^2_2.
%$$
%Therefore since Assumption \ref{var:bound:LDA} holds, we have:
%$$
%\alpha V_1 + (1 - \alpha) V_2 \geq \min(V_{\min}, \delta(\alpha^{-1} +  (1 - \alpha)^{-1} ))(\|\bbeta^*\|_2^2 \|\vb^*\|_2^2 + \|\vb^*\|_2^2).
%$$
\end{proof}

\begin{remark}This also shows that $\Delta \geq \delta(\alpha^{-1} +  (1 - \alpha)^{-1} )\|\vb^*\|^2_2 \geq  \delta(\alpha^{-1} +  (1 - \alpha)^{-1} ) 4K_{\bU}^{-4}> 0$.
\end{remark}

\begin{proof}[Proof of Corollary \ref{IFsparseLDA}]
We verify the conditions of Section \ref{masterthm:sec}. To see Assumption (\ref{consistencyassumpweakcn}), we can use Lemmas \ref{arxivsupp:main:betadiffLDA} and \ref{arxivsupp:main:vdiffLDA} to get $\|\hat\bbeta - \bbeta^*\|_1 = O_p\left((\|\bbeta^*\|_1 \vee 1)s\sqrt{\log d/n}\right)$, $\|\hat\vb - \vb^*\|_1 = O_p\left(\|\vb^*\|_1s_{\vb}\sqrt{\log d/n}\right)$ provided that $\lambda$ and $\lambda'$ are large enough, and we used the fact that $n_1 \asymp n_2$. 

Next we check Assumption \ref{noiseassumpCI}. To see (\ref{betastartassumpCI}), fix a $|\theta - \theta^*| < \epsilon$, for some $\epsilon > 0$. By the triangle inequality:
$$
\bigg\|\hat \bSigma_n\bbeta^*_{\theta} - \bSigma_{}\bbeta_{\theta}^* + (\bar \bX - \bar \bY) - \bmu_1 + \bmu_2\bigg\|_{\infty} \leq \|\hat \bSigma_{n} - \bSigma_{}\|_{\max}(\|\bbeta^*\|_1 + \epsilon) + \|\bar \bX - \bmu_1\|_\infty + \|\bar \bY - \bmu_2\|_\infty.
$$
The RHS is $O_p\bigg((\|\bbeta^*\|_1 \vee 1)\sqrt{\log d/n}\bigg)$, by Lemma \ref{arxivsupp:main:Sigmaclose} and bound (\ref{arxivsupp:main:concXbar}). The same logic shows that $r_2(n) \asymp \|\vb^*\|_1(\|\bbeta^*\|_1 \vee 1)\sqrt{\log d/n}$, which implies (\ref{betastartassumpCIvstar}). Since the Hessian $\Tb$ in (\ref{lambdaprimeasumpCI}) is free of $\bbeta$ we are allowed to set $r_3(n) = \lambda' \asymp \|\vb^*\|_1\sqrt{\log d/n} = o(1)$ (by Lemma \ref{arxivsupp:main:vdiffLDA}). Finally the two expectations in Assumption \ref{noiseassumpCI}, are bounded as we see below:
\begin{align*}
\|\bSigma_{}\bbeta^*_{\theta} - \bSigma_{}\bbeta^*\|_{\infty} = \|\bSigma_{}(\bbeta^*_{\theta} - \bbeta^*)\|_{\infty} \leq \|\bSigma_{*1} \|_{\infty}\epsilon \leq 2 K_{\bU}^2 \epsilon, ~~~~~~~~ \|\vb^{*T}\bSigma_{ -1}\|_{\infty}  = 0.
\end{align*}
By adding up the following two identities:
\begin{align*}
\sqrt{n} O_p\left(\|\vb^*\|_1 \sqrt{\log d/n}\right)O_p\left((\|\bbeta^*\|_1 \vee 1) s\sqrt{\log d/n}\right) & = o_p(1),\\
\sqrt{n} O_p\left( \|\vb^*\|_1s_{\vb}\sqrt{\log d/n}\right) O_p\bigg((\|\bbeta^*\| \vee 1)\sqrt{\log d/n}\bigg) & = o_p(1),
\end{align*}
we get that (\ref{assumpone}) is also valid in this case by assumption.

To verify the consistency of $\tilde \theta$ we check the assumptions in Theorem \ref{consistency:sol}. Clearly the map $\vb^{*T}\bSigma_{} (\bbeta^*_{\theta} - \bbeta^*) = (\theta - \theta^*)$ has a unique $0$ when $\theta = \theta^*$. Moreover, the map $\theta \mapsto \hat \vb^T(\bSigma_{n} \hat\bbeta_{\theta} - (\bar \bX - \bar \bY))$ is continuous as it is linear. In addition, it has a unique zero except in cases when $\hat \vb^T \bSigma_{n,*1} = 0$. However note that $|\hat \vb^T \bSigma_{n,*1} - 1| \leq \lambda'$ by (\ref{lambda:prime:clime}), and hence for small enough values of $\lambda'$ there will exist a unique zero.

Next we verify Assumption \ref{CLTcond} in Lemma \ref{arxivsupp:main:normalityLDA}. Finally we move on to show (\ref{stabtwo}). Observe that (\ref{stabtwo}) is trivial as its LHS $\equiv 0$ in this case. %Furthermore, (\ref{arxivsupp:main:stabone}) is satisfied with $r_6(n) = \lambda'$. 
\end{proof}

\begin{lemma} \label{arxivsupp:main:normalityLDA} Under the conditions of Corollary \ref{IFsparseLDA} we have the following $
\Delta^{-1/2}\sqrt{n}S(\bbeta^*) \rightsquigarrow N(0,1),$
where $\Delta$ is defined as in (\ref{deltasparseLDAdef}).
\end{lemma}

\begin{proof}[Proof of Lemma \ref{arxivsupp:main:normalityLDA}] %We have the following identity:
%\begin{align*}
%n^{1/2} \vb^{*T} (\hat \bSigma_n \bbeta^* - (\bar \bX - \bar \bY)) = \underbrace{n^{1/2} \hat \vb^T (\hat \bSigma_n  \bbeta^* - (\bar \bX - \bar \bY))}_{I_1} + \underbrace{n^{1/2} \hat \vb^T \hat \bSigma_n (\hat \bbeta_0 - \bbeta^*)}_{I_2}.
%\end{align*}
We stat by defining the following quantity:
\begin{align}
\tilde \bSigma_n = \frac{1}{n} \left[\sum_{i = 1}^{n_1} (\bX_i - \bmu_1)^{\otimes 2} + \sum_{i = 1}^{n_2} (\bY_i - \bmu_2)^{\otimes 2}\right]. \label{arxivsupp:main:defsigmatilde}
\end{align}
We have $
n^{1/2} \vb^{*T} (\hat \bSigma_n \bbeta^* - (\bar \bX - \bar \bY)) = \underbrace{n^{1/2} \vb^{*T} (\tilde \bSigma_n  \bbeta^* - (\bar \bX - \bar \bY))}_{I_{1}} + n^{1/2} \underbrace{\vb^{*T} (\hat \bSigma_n - \tilde \bSigma_n) \bbeta^*}_{I_{2}}$.
We proceed with showing that the term $I_{2}$ is small:
$$
|I_{2}| \leq n^{1/2}\|\vb^*\|_1\|\bbeta^*\|_1 \|\hat \bSigma_n - \tilde \bSigma_n\|_{\max} =  \|\vb^*\|_1\|\bbeta^*\|_1 O_p\left(\log d/n^{1/2}\right) = o_p(1),
$$
where we used (\ref{arxivsupp:main:sigmatildesigmahat}) from Lemma \ref{arxivsupp:main:Sigmaclose} (and made usage of the fact that $n_1 \asymp n_2$).
%We can control $I_{13}$ by the following:
%\begin{align*}
%|I_{13}| & \leq n^{1/2} \|\hat \vb - \vb^* \|_1 \|\hat \bSigma_n  \bbeta^* - (\bar \bX - \bar \bY))\|_{\infty} \leq n^{1/2} \|\vb^*\|_1 O_p\left(s_{\vb} \sqrt{\log d/n}\right) (\|\bbeta^*\|_1 \vee 1)O_p\left(\sqrt{\log d/n}\right)\\
%& = o_p(1),
%\end{align*}
%where the next to last inequality follows from Lemma \ref{arxivsupp:main:betadiffLDA} and Lemma \ref{arxivsupp:main:vdiffLDA}. 
%We next deal with $I_2$:
%\begin{align*}
%|I_2| & \leq n^{1/2} \|\hat \vb^T (\hat \bSigma_n)_{-1}\|_{\infty} \|\hat \bbeta_0 - \bbeta^*\|_1 \leq n^{1/2} \lambda' \|\hat \bbeta - \bbeta^*\|_1\\
%& \leq n^{1/2} \|\vb^*\|_1 O_p\left(\sqrt{\log d/n}\right) (\|\bbeta^*\|_1\vee 1) O_p\left(s\sqrt{\log d/n}\right)\\
%& = o_p(1),
%\end{align*}
%where $(\hat \bSigma_n)_{-1}$ means dropping the first column (the one corresponding to the 0 coefficient in $\bbeta^*$ under the null) of $\hat \bSigma_n$, as by definition $\hat \bbeta_0 = (0, \hat \bgamma^T)^T$ and $\bbeta^* = (0, \bgamma^{*T})^T$ under the null. 
%The second inequality in the preceding display follows from Lemma \ref{arxivsupp:main:betadiffLDA} and Lemma \ref{arxivsupp:main:vdiffLDA}.
Next we take a closer look at the term $I_{1}$:
\begin{align*}
I_{1} & = n^{1/2} \vb^{*T} (\tilde \bSigma_n  \bbeta^* - (\bmu_1 - \bmu_2)) + n^{1/2} \vb^{*T}(\bar \bX - \bmu_1 - \bar \bY + \bmu_2)\\
& = n^{1/2} \vb^{*T} \frac{1}{n} \sum_{i = 1}^{n} \left(\bU_i^{\otimes 2}\bbeta^* -  (\bmu_1 - \bmu_2) + \left[\frac{n}{n_1} I(i \leq n_1) - \frac{n}{n_2} I(i > n_1)\right] \bU_i\right).
\end{align*}
Next, by $n_1/n + o_p(1) = \alpha$, it is clear that:
\begin{align*}
I_1 & = n^{-1/2} \vb^{*T} \sum_{i = 1}^{n_1} \left(\bU_i^{\otimes 2}\bbeta^* -  (\bmu_1 - \bmu_2) + \alpha^{-1}  \bU_i\right)\\
& +  n^{-1/2} \vb^{*T} \sum_{i = n_1 + 1}^{n} \left(\bU_i^{\otimes 2}\bbeta^* -  (\bmu_1 - \bmu_2) - (1-\alpha)^{-1}  \bU_i\right)\\
& + \underbrace{\bigg(\frac{n}{n_1} - \alpha^{-1}\bigg)}_{o_p(1)} \underbrace{n^{-1/2} \sum_{i = 1}^{n_1} \vb^{*T}  \bU_i}_{O_p(1)} + \underbrace{\bigg(\frac{n}{n_2} - (1 - \alpha)^{-1}\bigg)}_{o_p(1)} \underbrace{n^{-1/2} \sum_{i = n_1 + 1}^{n} \vb^{*T}  \bU_i}_{O_p(1)},
\end{align*}
where we implicitly used Chebyshev's inequality and the fact that $\Var(\vb^{*T}\bU) \leq 2 \vb^{*T} \bSigma \vb^{*} \leq 2 \delta^{-1}$. 
Next we verify Lyapunov's condition. The sum of variances of the terms above equals:
$$
n_1 V_1 + n_2 V_2 = n (\alpha V_1 + (1 - \alpha) V_2)(1 + o(1)) \geq n V'_{\min}(\|\bbeta^*\|_2^2 \|\vb^*\|_2^2 + \|\vb^*\|_2^2)(1 + o(1)),
$$
by Lemma \ref{arxivsupp:rem:LDP}. Without loss of generality let's assume that $\alpha^{-1} > (1 - \alpha)^{-1}$. It follows then from Lemma \ref{arxivsupp:main:sparseLDAineq}, that for any $k > 2$:
$$
\EE \left|\vb^{*T}\bU_i^{\otimes 2}\bbeta^* -  \vb^{*T}(\bmu_1 - \bmu_2) + \alpha^{-1}  \vb^{*T}\bU_i\right|^k \leq \|\vb^*\|_2^k( C_1(s_{\vb}s)^{k/2}\|\bbeta^*\|^k_2  + C_2\alpha^{-k} s_{\vb}^{k/2}),
$$
and similarly:
$$
\EE \left|\vb^{*T}\bU_i^{\otimes 2}\bbeta^* -  \vb^{*T}(\bmu_1 - \bmu_2) - (1-\alpha)^{-1}  \vb^{*T}\bU_i\right|^k \leq  \|\vb^*\|_2^k( C_1(s_{\vb}s)^{k/2}\|\bbeta^*\|^k_2  + C_2\alpha^{-k} s_{\vb}^{k/2}),
$$
where $C_1$ and $C_2$ are some absolute constants depending on $k$ (see the Lemma for details). Therefore we conclude that the sum in Lyapunov's condition, is bounded by:
$$
\frac{(s_{\vb}s)^{k/2}}{(1 + o(1))n^{k/2 - 1}} \underbrace{\frac{C_1\|\bbeta^*\|^k_2 + \frac{C_2\alpha^{-k}}{s^{k/2}}}{(V'_{\min})^{k/2}(\|\bbeta^*\|^2_2 + 1)^{k/2}}}_{O(1)} = o(1).
$$
This completes the proof.
\end{proof}

\begin{remark}\label{consistent:est:sLDA}
We propose the following consistent estimator of $\Delta$, and prove its consistency in Proposition \ref{varconstestLDP}.
\begin{align*}
\hat \Delta  := \frac{1}{n} \sum_{i = 1}^{n_1} \left(\hat \vb^T (\bX_i - \bar \bX)^{\otimes 2} \hat \bbeta\right)^2 &+ \frac{1}{n} \sum_{i = 1}^{n_1} \left(\frac{n}{n_1}\hat \vb^T (\bX_i - \bar \bX) \right)^2 +  \frac{1}{n} \sum_{i = n_1 + 1}^{n} \left(\hat \vb^T (\bY_i - \bar \bY)^{\otimes 2} \hat \bbeta\right)^2 \\
&+\frac{1}{n} \sum_{i = n_1 + 1}^{n} \left(\frac{n}{n_2}\hat \vb^T (\bY_i - \bar \bY) \right)^2  - (\hat \vb^T (\bar \bX - \bar \bY))^2.
\end{align*}
\end{remark}

\begin{proposition} \label{varconstestLDP} Under the same conditions as in Corollary \ref{IFsparseLDA}, $\max(\|\bmu_1\|_{\infty}, \|\bmu_2\|_{\infty}) = O(1)$, and the following additional assumptions:
\begin{align*}
%\lambda' s_{\vb} \|\vb^*\|_1\|\bbeta^*\|_1\left(\lambda + \|\bmu_1 - \bmu_2\|_{\infty}\right)\log(nd) \|\bmu_1 - \bmu_2\|_{\infty} = o(1),\\
%\|\vb^*\|_1\log(nd)s \lambda \|\bbeta^*\|_1 (1 + s_{\vb}\lambda') = o(1),\\
%\|\bbeta^*\|_1 \left(\|\bmu_1 - \bmu_2\|_{\infty} + \lambda \right)\lambda'(\sqrt{\log(nd)} + \|\bmu_1\|_{\infty} +  \|\bmu_2\|_{\infty}) = o(1),\\
\max(\lambda' s_{\vb}, \lambda s) \|\vb^*\|_1\|\bbeta^*\|_1\sqrt{\log(nd)} &= o(1), ~~~   \Var((\vb^{*T}\bU)^2) = o(n), ~~~  \Var(\vb^{*T}\bU^{\otimes 2}\bbeta^*) = o(n),\\
&\EE(\vb^{*T}\bU^{\otimes 2}\bbeta^*)^2 = O(1),
\end{align*}
we have that $\hat \Delta \rightarrow_p \Delta$.
\end{proposition}

\begin{proof}[Proof of Proposition \ref{varconstestLDP}] %Observe that by (\ref{eqconLDA}) we have that $s_{\vb} \lambda' = o(1)$, and hence using $\max(\|\bmu_1\|_{\infty}, \|\bmu_2\|_{\infty}) = O(1)$, we have the following implications:
%\begin{align*}
%\lambda' s_{\vb} \|\vb^*\|_1\|\bbeta^*\|_1\left(\lambda + \|\bmu_1 - \bmu_2\|_{\infty}\right)\log(nd) \|\bmu_1 - \bmu_2\|_{\infty} = o(1),\\
%\|\vb^*\|_1\log(nd)s \lambda \|\bbeta^*\|_1 (1 + s_{\vb}\lambda') = o(1),\\
%\|\bbeta^*\|_1 \left(\|\bmu_1 - \bmu_2\|_{\infty} + \lambda \right)\lambda'(\sqrt{\log(nd)} + \|\bmu_1\|_{\infty} +  \|\bmu_2\|_{\infty}) = o(1).
%\end{align*}
Note that $\Delta$ can be decomposed as:
\begin{align*}
\Delta = & \alpha \EE(\vb^{*T}\bU^{\otimes 2}\bbeta^*)^2 + \alpha^{-1} \EE(\vb^{*T}\bU)^2 + (1-\alpha) \EE(\vb^{*T}\bU^{\otimes 2}\bbeta^*)^2 \\
& + (1-\alpha)^{-1} \EE(\vb^{*T}\bU)^2 -  (\vb^{*T}(\bmu_1 - \bmu_2))^2.
\end{align*}
%Similarly we can decompose:
%\begin{align*}
%\hat \Delta  & = \underbrace{\frac{1}{n} \sum_{i = 1}^{n_1} \left(\hat \vb^T (\bX_i - \bar \bX)^{\otimes 2} \hat \bbeta + \frac{n}{n_1}\hat \vb^T (\bX_i - \bar \bX) \right)^2}_{I_1} \\
%& +  \underbrace{\frac{1}{n} \sum_{i = n_1 + 1}^{n} \left(\hat \vb^T (\bY_i - \bar \bY)^{\otimes 2} \hat \bbeta - \frac{n}{n_2}\hat \vb^T (\bY_i - \bar \bY) \right)^2}_{I_2} \\
%& - \underbrace{(\hat \vb^T (\bar \bX - \bar \bY))^2}_{I_3}
%\end{align*}
We start from the last term:
$$
\underbrace{(\hat \vb^T (\bar \bX - \bar \bY))^2}_{I} =  \underbrace{[(\hat \vb^T (\bar \bX - \bar \bY))^2 - ( \vb^{*T} (\bar \bX - \bar \bY))^2]}_{I_{1}} + \underbrace{ (\vb^{*T} (\bar \bX - \bar \bY))^2}_{I_{2}}.
$$
We have $
|I_{1}| \leq  \|\hat \vb -  \vb^{*}\|_{1}\|\hat \vb +  \vb^{*}\|_{1} \|(\bar \bX - \bar \bY)^{\otimes 2}\|_{\max}.$
Using Lemma \ref{arxivsupp:main:vdiffLDA}, we know $\|\hat \vb -  \vb^{*}\|_{1} = O_p\left(\|\vb^*\|_1 s_{\vb} \sqrt{\log d/n} \right)$. We can apply the concentration inequality (\ref{arxivsupp:main:concXbar})  provided in Lemma \ref{arxivsupp:main:Sigmaclose} to claim that $
 \|(\bar \bX - \bar \bY)^{\otimes 2}\|_{\max} \leq \|\bmu_1 - \bmu_2\|_{\infty}^2 + \|\bmu_1 - \bmu_2\|_{\infty}O_p\left(\sqrt{\log d/n}\right),$
where we used the triangle inequality $\|\bar \bX - \bar \bY \|_{\infty} \leq \|\bar \bX - \bmu_1\|_{\infty} + \|\bar \bY - \bmu_2\|_{\infty} + \|\bmu_1 - \bmu_2\|_{\infty}$. Finally due to our assumptions we have $
|I_{1}| = \|\bmu_1 - \bmu_2\|_{\infty}^2 O_p\left(\|\vb^*\|^2_1 s_{\vb} \sqrt{\log d/n} \right) = o_p(1).$
Next we tackle $
I_{2} = \underbrace{ (\vb^{*T} (\bar \bX - \bar \bY))^2 - (\vb^{*T} ( \bmu_1 - \bmu_2))^2}_{I_{21}} + \underbrace{(\vb^{*T} ( \bmu_1 -  \bmu_2))^2}_{I_{22}}.$
In a similar fashion to before, applying inequality (\ref{arxivsupp:main:concXbar}), we can get $
|I_{21}| \leq \|\vb^*\|^2_1 O_p\left(\sqrt{\log d/n}\right)\|\bmu_1 - \bmu_2\|_{\infty} = o_p(1)$. %Thus since $I_{322} = \alpha (\vb^{*T} ( \bmu_1 -  \bmu_2))^2 + (\frac{n_1}{n} - \alpha) (\vb^{*T} ( \bmu_1 -  \bmu_2))^2$, and:
%$$
%o\left(\frac{1}{n}\right) (\vb^{*T} ( \bmu_1 -  \bmu_2))^2 \leq o\left(\frac{1}{n}\right) \|\vb^*\|_1^2 \|\bmu_1 - \bmu_2\|_{\infty}^2 = o(1)
%$$
Thus we have shown $
I = (\vb^{*T} ( \bmu_1 -  \bmu_2))^2  + o_p(1).$
To this end define the following shorthand notations:
\begin{align*}
I_X(\vb, \bbeta)& = \frac{1}{n} \sum_{i = 1}^{n_1} \left( \vb^T (\bX_i - \bar \bX)^{\otimes 2}  \bbeta\right)^2, ~~~
I_Y(\vb, \bbeta) = \frac{1}{n} \sum_{i = 1}^{n_2} \left( \vb^T (\bY_i - \bar \bY)^{\otimes 2}  \bbeta\right)^2 
\end{align*}
%Thus in terms of this notation we have $I_{1} = I_{11}(\hat \vb, \hat \bbeta) +  I_{12}(\hat \vb, \hat \bbeta) +  I_{13}(\hat \vb, \hat \bbeta)$. Similarly define $I_{21}, I_{22}, I_{23}$.

Next we show that $I_{X}(\hat \vb, \hat \bbeta) + I_{Y}(\hat \vb, \hat \bbeta)$ is consistent for $\EE(\vb^* \bU^{\otimes 2} \bbeta^*)^2$. We begin with the following bound:
\begin{align*}
\MoveEqLeft \frac{1}{n} \sum_{i = 1}^{n_1} \left( (\hat \vb - \vb^*)^T (\bX_i - \bar \bX)^{\otimes 2}  \hat \bbeta\right)^2 +  \frac{1}{n} \sum_{i = 1}^{n_2} \left( (\hat \vb - \vb^*)^T (\bY_i - \bar \bY)^{\otimes 2}  \hat \bbeta\right)^2 \\
& \leq \|\hat \vb - \vb^*\|^2_1 M \|\hat \bSigma_n \hat \bbeta \|_{\infty} \|\hat \bbeta\|_1,
\end{align*}
where $M = \max\left\{\max_{i = 1,\ldots, n_1} \|(\bX_i - \bar \bX)^{\otimes 2}\|_{\max}, \max_{i = 1,\ldots, n_2} \|(\bY_i - \bar \bY)^{\otimes 2} \|_{\max} \right\}.$ Note that the random variables $\bX_i - \bar \bX$ and $\bY_i - \bar \bY$ are in fact mean 0 sub-Gaussian variables since e.g. $\|\bX_i - \bar \bX\|_{\psi_2} \leq \|\bX_i - \bmu_1\|_{\psi_2} + \|\bar \bX - \bmu_1\|_{\psi_2} \leq 2K_{\bU}$. Thus an application of Lemma \ref{arxivsupp:main:Mbound}, and the fact that $n_1 \asymp n_2 \asymp n$, gives us that $M = O(\log(nd))$. Furthermore we have:
$$
\|\hat \bSigma_n \hat \bbeta\|_{\infty} \leq \lambda + \|\bar \bX - \bmu_1\|_{\infty} + \|\bar \bY - \bmu_2\|_{\infty} + \|\bmu_1 - \bmu_2\|_{\infty} = O_p(1),
$$
by application of (\ref{arxivsupp:main:concXbar}), and the way we select $\lambda$. % we have:
%$$
%\|\hat \bSigma_n \hat \bbeta\|_{\infty} \leq  (\|\bbeta^*\|_1 \vee 1)O_p\left(\sqrt{\log d/n}\right) + \|\bmu_1 - \bmu_2\|_{\infty}.
%$$
 Putting the last several inequalities together with Lemma \ref{arxivsupp:main:betadiffLDA}, Lemma \ref{arxivsupp:main:vdiffLDA} and the triangle inequality, we obtain:
 $$
|\sqrt{ I_X(\hat \vb, \hat \bbeta) + I_Y(\hat \vb, \hat \bbeta)} - \sqrt{ I_X(\vb^*, \hat \bbeta) + I_Y(\vb^*, \hat \bbeta)} | \leq \|\vb\|_1\sqrt{\|\bbeta\|_1} s_{\vb} \sqrt{\log(nd) \log d/n} O_p(1).
 $$
 Applying the same technique we can further show that in fact:
 $$
|\sqrt{ I_X(\hat \vb, \hat \bbeta) + I_Y(\hat \vb, \hat \bbeta)} - \sqrt{ I_X(\vb^*, \bbeta^*) + I_Y(\vb^*, \bbeta^*)} | = o_p(1).
 $$

We proceed to show that $I_X(\vb^*, \bbeta^*) + I_Y(\vb^*, \bbeta^*)$ is consistent for $\EE(\vb^* \bU^{\otimes 2} \bbeta^*)^2$ and since $\EE(\vb^* \bU^{\otimes 2} \bbeta^*)^2 = O(1)$, the latter inequality also shows that $I_X(\hat \vb, \hat \bbeta) + I_Y(\hat \vb, \hat \bbeta)$ is consistent for $\EE(\vb^* \bU^{\otimes 2} \bbeta^*)^2$. Define the following notation $
\tilde M :=  \max_{i = 1,\ldots, n} \|\bU_i^{\otimes 2}\|_{\max}.$
For exactly the same reasons as for $M$ we have $\tilde M = O_p(\log(nd))$. Next we consider the difference:
\begin{align*}
%&I_{11}(\vb^*, \bbeta^*) + I_{12}(\vb^*, \bbeta^*) -  \frac{1}{n} \sum_{i = 1}^{n_1} \left( \vb^{*T} \underbrace{(\bX_i - \bmu_1)(\bX_i - \bmu_1)^T}_{\bU_i^{\otimes 2}}  \bbeta^*\right)^2 -  \frac{1}{n} \sum_{i = n_1 + 1}^{n} \left( \vb^{*T} (\bY_i - \bmu_2)(\bY_i - \bmu_2)^T  \bbeta^*\right)^2\\
&|I_{X}(\vb^*, \bbeta^*) + I_{Y}(\vb^*, \bbeta^*) -  n^{-1}\sum_{i = 1}^{n} \left( \vb^{*T} \bU_i^{\otimes 2}\bbeta^*\right)^2| \\
& \leq \|\vb^*\|_1\|\bbeta^*\|_1 V \left(\frac{1}{n}\sum_{i = 1}^{n_1}|\vb^{*T}(\bX_i - \bar \bX)||(\bX_i - \bar \bX)^T \bbeta^*| +\frac{1}{n}\sum_{i = 1}^{n_2}|\vb^{*T}(\bY_i - \bar \bY)||(\bY_i - \bar \bY)^T \bbeta^*| \right.\\
& ~~~~~~~~~~~~~~~~~~~~~~ \left. +  \frac{1}{n}\sum_{i = 1}^n |\vb^{*T}\bU_i| |\bU_i^T \bbeta^*|\right),
\end{align*}
where $
V =  \max\left\{\max_{i = 1,\ldots, n_1} \|(\bX_i - \bar \bX)^{\otimes 2} - \bU_i^{\otimes 2}\|_{\max}, \max_{i = 1,\ldots, n_2} \|(\bY_i - \bar \bY)^{\otimes 2} -\bU_{i + n_1}^{\otimes 2} \|_{\max} \right\}.$
Note that by the simple inequality $|ab| \leq (a^2 + b^2)/2$, we have, that the expression in the brackets is bounded by:
$$
%(\vb^{*T}\hat \bSigma_n \bbeta^* + \vb^{*T}\tilde \bSigma_n \bbeta^*) \leq \vb^{*T}(\hat \bSigma_n + \tilde \bSigma_n) \vb^*/2 + \bbeta^{*T}(\hat \bSigma_n + \tilde \bSigma_n) \bbeta^*/2
\leq \vb^{*T}(\hat \bSigma_n + \tilde \bSigma_n) \vb^*/2 + \bbeta^{*T}(\hat \bSigma_n + \tilde \bSigma_n) \bbeta^*/2.
$$
We have that $\vb^{*T}\hat \bSigma_n\vb^* \leq \|\vb^*\|_1 \|\vb^{*T}\hat \bSigma_n\|_{\infty} =\|\vb^*\|_1 + \|\vb^*\|^2_1 O_p\left(\sqrt{\log d/n}\right)$. Similarly since by (\ref{arxivsupp:main:sigmatildesigmahat}) $\|\hat \bSigma_n - \tilde \bSigma_n\|_{\max} = O_p\left({\frac{\log d}{n}}\right)$ we have that $\vb^{*T}\tilde \bSigma_n\vb^* \leq \|\vb^*\|_1 + \|\vb^*\|^2_1 O_p\left(\sqrt{\log d/n}\right)$. Similarly one can show that $\bbeta^{*T}\hat \bSigma_n\bbeta^* \leq \|\bbeta^*\|_1 \|\bmu_1 - \bmu_2\|_{\infty}  + \|\bbeta^*\|_1(\|\bbeta^*\|_1 \vee 1)O_p\left(\sqrt{\log d/n}\right)$, and a similar inequality for $\bbeta^{*T}\tilde \bSigma_n\bbeta^*$. We next inspect $V$:
\begin{align*}
\max_{i = 1,\ldots, n_1} \|(\bX_i - \bar \bX)^{\otimes 2} - \bU_i^{\otimes 2}\|_{\max} &  \leq \max_{i = 1,\ldots, n_1} 2 \|\bX_i\|_{\infty} \|\bar \bX - \bmu_1\|_{\infty}\\
& + \|\bar \bX - \bmu_1\|_{\infty}(\|\bar \bX - \bmu_1\|_\infty + 2\|\bmu_1\|_\infty),
\end{align*}
and we can similarly bound the other term in $V$. Note that in Lemma \ref{arxivsupp:main:Mbound} we showed that $\max_{i = 1,\ldots, n_1} \|\bX_i\|_{\infty} = O_p(\sqrt{\log(nd)})$, and as we argue in (\ref{arxivsupp:main:concXbar}), we have $ \|\bar \bX - \bmu_1\|_{\infty} = O_p\left(\sqrt{\log d/n}\right)$, and thus $
V = O_p\left(\sqrt{\log d/n}\right)(\sqrt{\log(nd)} + \|\bmu_1\|_{\infty} +  \|\bmu_2\|_{\infty}).$
Hence under our assumptions, we have:
$$
|I_{X}(\vb^*, \bbeta^*) + I_{Y}(\vb^*, \bbeta^*) -  n^{-1} \sum_{i = 1}^{n} \left( \vb^{*T} \bU_i^{\otimes 2}\bbeta^*\right)^2|  = o_p(1).
$$
Finally we finish this part upon noting that $
\left|\frac{1}{n} \sum_{i = 1}^{n} \left( \vb^{*T} \bU_i^{\otimes 2}\bbeta^*\right)^2  - \EE \left( \vb^{*T} \bU_i^{\otimes 2}\bbeta^*\right)^2\right| = o_p(1)$, and that
under the assumption $\Var((\vb^{*T} \bU^{\otimes 2} \bbeta^*)^2) = o(n)$ by Chebyshev's inequality.

Next we turn our attention to the term $
\frac{n}{n_1} \frac{1}{n_1} \sum_{i = 1}^{n_1} \left(\hat \vb^T (\bX_i - \bar \bX) \right)^2,$
and show it's consistent for $\alpha^{-1} \EE(\vb^{*T} \bU_i)^2$. First note that since $\frac{n}{n_1} = \alpha^{-1} + o(\frac{1}{n})$, and we will show the rest of the expression is $O_p(1)$, we will just focus on the average term. We first show the following difference is small:
\begin{align*}
\left|\frac{1}{n_1} \sum_{i = 1}^{n_1} [\left(\hat \vb^T (\bX_i - \bar \bX) \right)^2 - \left( \vb^{*T} (\bX_i - \bar \bX) \right)^2 ] \right|&=| \left(\hat \vb -  \vb^{*}\right)^T  \hat \bSigma_{\bX} \left(\hat \vb +  \vb^{*}\right)^T|\\
& \leq \|\hat \vb - \vb^*\|_1 (\|\hat \vb - \vb^*\|_1 + 2\|\vb^*\|_1)\|\hat \bSigma_{\bX} \|_{\max}.
\end{align*}
Using the same technique as in the proof of Lemma \ref{arxivsupp:main:Sigmaclose}, one can show that $\|\hat \bSigma_{\bX}\|_{\max} \leq \|\bSigma\|_{\max} + O_p\left(\sqrt{\log d/n}\right)$. By Lemma \ref{arxivsupp:main:vdiffLDA} we have $\|\hat \vb - \vb^*\|_{1} = \|\vb^*\|_1 s_{\vb} O_p\left(\sqrt{\log d/n}\right)$, and hence:
$$
\frac{1}{n_1} \left|\sum_{i = 1}^{n_1} [\left(\hat \vb^T (\bX_i - \bar \bX) \right)^2 - \left( \vb^{*T} (\bX_i - \bar \bX) \right)^2 ] \right| \leq \|\vb^*\|_1^2 s_{\vb} O_p\left(\sqrt{\log d/n}\right) = o_p(1),
$$
by assumption. Next we control:
\begin{align*}
\left|\frac{1}{n_1}\sum_{i = 1}^{n_1} [ \left( \vb^{*T} (\bX_i - \bar \bX) \right)^2 -  \left( \vb^{*T} \bU_i \right)^2]\right| & = \left|\vb^{*T} (\bmu_1 - \bar \bX)\frac{1}{n_1}\sum_{i = 1}^{n_1}(2 \bX_i - \bar \bX - \bmu_1)\vb^{*}\right|\\
& \leq \|\vb^*\|^2_1 \| \bmu_1 - \bar \bX\|^2_{\infty} = \|\vb^*\|_1^2 O_p\left({\frac{\log d}{n}}\right) = o_p(1).% \left(O_p\left(\log(nd)\right) + O_p\left(\sqrt{\log d/n}\right)\right)\\
%& = o_p(1).
\end{align*}
Thus after using Chebyshev's inequality upon observing that $\Var((\vb^{*T}\bU)^2) = o(n)$, we have shown the desired consistency. Similarly we can also show that $\frac{n}{n_2} \frac{1}{n_2} \sum_{i = 1}^{n_2} \left(\hat \vb^T (\bY_i - \bar \bY) \right)^2$ is consistent for $(\alpha - 1)^{-1} \EE(\vb^{*T}\bU_i)^2$. This concludes the proof.

%Finally we show that the cross term is consistent:
%\begin{align*}
%|I_{12}(\hat \vb, \hat \bbeta) + I_{22}(\hat \vb, \hat \bbeta) - I_{12}( \vb^*, \hat \bbeta) - I_{22}( \vb^*, \hat \bbeta)|
%\end{align*}
%\begin{align*}
%& \left|\frac{1}{n_1} \sum_{i = 1}^{n_1} (\hat \vb^T (\bX_i - \bar \bX))^2(\bX_i - \bar \bX)^T \hat \bbeta -  \frac{1}{n_1} \sum_{i = 1}^{n_1} (\vb^{*T} (\bX_i - \bar \bX))^2(\bX_i - \bar \bX)^T \hat \bbeta\right|\\
% &\leq\max_{i = 1,\ldots,n_1}\|\bX_i - \bar \bX\|_{\infty} \|\hat \bbeta\|_{1} \|\hat \vb - \vb^{*T}\|_{1}\|\hat \vb + \vb^{*T}\|_{1}  \| \hat \bSigma_{\bX} \|_{\max}\\
%& = O_p(\sqrt{\log(nd)})\|\bbeta^*\|_1 \|\vb^*\|^2_1 s_{\vb} \sqrt{\log d/n} O_p(\log(nd))\\
%& = o_p(1)
%\end{align*}
%Similarly we deal with:
%\begin{align*}
%& \left|\frac{1}{n_1} \sum_{i = 1}^{n_1} ( \vb^{*T} (\bX_i - \bar \bX))^2(\bX_i - \bar \bX)^T \hat \bbeta -  \frac{1}{n_1} \sum_{i = 1}^{n_1} (\vb^{*T} (\bX_i - \bar \bX))^2(\bX_i - \bar \bX)^T \bbeta^*\right|\\
%& \leq \frac{1}{n_1} \sum_{i = 1}^{n_1} ( \vb^{*T} (\bX_i - \bar \bX))^2 \max_{i = 1,\ldots, n_1}\|(\bX_i - \bar \bX)\|_{\infty}  \|\hat \bbeta - \bbeta^*\|_1\\
%& = \|\vb^*\|_1^2 \|\bbeta^*\|_1 
%\end{align*}
\end{proof}

\begin{lemma} \label{arxivsupp:main:Sigmaclose} 
The following inequality holds $
\|\hat \bSigma_n -  \bSigma \|_{\max} \leq  \tilde{t}_{\bU}(d,n) + t^2_{\bU}(d,n),$
with probability at least $1 - 2 d^{2 - \tilde cA_{\bU}^2} - 2 e d^{1 - cA^2_{\bU}}$, where:
\begin{align}
t_{\bU}(d,n) = A_{\bU} K_{\bU}\sqrt{\log d/\min(n_1, n_2)};~~~\tilde{t}_{\bU}(d,n)=4 A_{\bU} K_{\bU}^2 \sqrt{\log d/n}.\label{arxivsupp:main:t:t:tilde:def}
\end{align}
and $A_{\bU} > 0$ is an arbitrary positive constant, $\bar c$ and $c$ are absolute contents independent of the distribution of $\bU$,  and $K_{\bU}$ is as defined in the main section of the text.
\end{lemma}

\begin{proof}[Proof of Lemma \ref{arxivsupp:main:Sigmaclose}]
We start by showing a concentration bound on $\|\bar \bX - \bmu_1\|_{\infty}$ and $\|\bar \bY - \bmu_2\|_{\infty}$. By proposition 5.10 in \cite{vershynin2010introduction} and the union bound, we have:
\begin{align} \label{arxivsupp:main:concXbar}
\PP(\|\bar \bX - \bmu_1 \|_{\infty} > t) \leq e d \exp\left(-\frac{cn_1 t^2}{K_{\bU}^2}\right).
\end{align}
A similar inequality holds for $\|\bar \bY - \bmu_2 \|_{\infty}$. Select $t_{\bU}(d,n) = A_{\bU} K_{\bU}\sqrt{\log d/\min(n_1, n_2)}$, where $A_{\bU} > 0$ is some large constant. The triangle inequality yields 
$\|\hat \bSigma_n - \bSigma\|_{\max} \leq \|\hat \bSigma_n - \tilde \bSigma_{n}\|_{\max} + \|\tilde \bSigma_{n} - \bSigma\|_{\max},$
where $\tilde \bSigma_n$ is defined as in (\ref{arxivsupp:main:defsigmatilde}). Next, we have that:
\begin{align}
\|\hat \bSigma_n - \tilde \bSigma_{n}\|_{\max} & \leq \frac{n_1}{n} \|(\bar \bX - \bmu_1)^{\otimes 2}\|_{\max} + \frac{n_2}{n} \|(\bar \bY - \bmu_2)^{\otimes 2}\|_{\max} \nonumber\\
& \leq  \frac{n_1}{n} (\|\bar \bX - \bmu_1\|_{\infty})^2 +  \frac{n_2}{n} (\|\bar \bY - \bmu_2\|_{\infty})^2  \leq t^2_{\bU}(d,n).  \label{arxivsupp:main:sigmatildesigmahat}
\end{align}
where the last inequality holds with high probability. Note that by Lemma \ref{arxivsupp:main:samplecovpopcov} we have:
$$
 \|\tilde \bSigma_{n} - \bSigma\|_{\max} \leq 4 A_{\bU} K_{\bU}^2 \sqrt{\log d/n} =: \tilde{t}_{\bU}(d,n),
$$
with probability at least $1 - 2 d^{2 - \bar c A_{\bU}^2}$. Adding the last two inequalities completes the proof.
\end{proof}

\begin{lemma} \label{arxivsupp:main:samplecovRELDA} Assume the same conditions as in Lemma \ref{arxivsupp:main:Sigmaclose}, and assume further that the minimum eigenvalue  $\lambda_{\min}(\bSigma) > 0$ and $s  (\tilde{t}_{\bU}(d,n) + t^2_{\bU}(d,n)) \leq  (1 - \kappa) \frac{\lambda_{\min}(\bSigma)}{(1 + \xi)^2}$, where $0 < \kappa < 1$. We then have that $\hat \bSigma_n$ satisfies the RE property with $\operatorname{RE}_{\hat \bSigma_n}(s, \xi) \geq \kappa \lambda_{\min} (\bSigma)$ with probability at least $1 - 2 d^{2 - \bar c A_{\bU}^2} - 2 e d^{1 - cA^2_{\bU}}$.

\begin{remark} In fact this event happens on the same event as in Lemma \ref{arxivsupp:main:Sigmaclose}.
\end{remark}
\end{lemma}

\begin{proof}[Proof of Lemma \ref{arxivsupp:main:samplecovRELDA}] The proof follows the proof of Lemma \ref{arxivsupp:main:samplecovRE}, but uses Lemma \ref{arxivsupp:main:Sigmaclose} instead of Lemma \ref{arxivsupp:main:samplecovpopcov}, hence we omit it. 
\end{proof}

\begin{lemma} \label{arxivsupp:main:betadiffLDA} Assume that --- $\lambda_{\min}(\bSigma) > 0$, $s(\tilde{t}_{\bU}(d,n) + t^2_{\bU}(d,n)) \leq   (1 - \kappa) \frac{\lambda_{\min}(\bSigma)}{(1 + \xi)^2 }$, where $0 < \kappa < 1$ and $\lambda \geq  \left(\tilde{t}_{\bU}(d,n) + t^2_{\bU}(d,n)\right)\|\bbeta^*\|_1 + 2t_{\bU}(d,n).$
Then we have that $\|\hat \bbeta - \bbeta^*\|_{1} \leq \frac{8 \lambda s}{ \operatorname{RE}(s, 1)} $ with probability at least  $1 - 2 d^{2 - \bar c A_{\bU}^2} - 2 e d^{1 - cA^2_{\bU}}$. (see (\ref{arxivsupp:main:t:t:tilde:def}) for definition of $t_{\bU}$ and $\tilde t_{\bU}$)

\begin{remark} In fact this event happens on the same event as in Lemma \ref{arxivsupp:main:Sigmaclose}.
\end{remark}
\end{lemma}

\begin{proof}[Proof of Lemma \ref{arxivsupp:main:betadiffLDA}] We start by showing the true parameter $\bOmega\bdelta = \bbeta^*$ satisfies the sparse LDA constraint --- $\|\hat \bSigma_n \bbeta^* - (\bar \bX - \bar \bY )\|_{\infty} \leq \lambda$ with probability at least $1 - 2 d^{2 - \bar c A_{\bU}^2} - 2 e d^{1 - cA^2_{\bU}}$. We have that:
\begin{align*}
 \|\hat \bSigma_n \bbeta^* - (\bar \bX - \bar \bY )\|_{\infty} & \leq   \underbrace{\| \bSigma \bbeta^* - (\bmu_1 - \bmu_2)\|_{\infty}}_{0} +  \|\hat \bSigma_n - \bSigma\|_{\max} \|\bbeta^*\|_1\\
& + \|\bar \bX - \bmu_1\|_{\infty} +  \|\bar \bY - \bmu_2\|_{\infty}.
\end{align*}
Collecting the bounds we derived in Lemma \ref{arxivsupp:main:Sigmaclose} we get:
$$
 \|\hat \bSigma_n \bbeta^* - (\bar \bX - \bar \bY )\|_{\infty} \leq  \left(\tilde{t}_{\bU}(d,n) + t^2_{\bU}(d,n)\right)\|\bbeta^*\|_1 + 2t_{\bU}(d,n).
$$
The last inequality implies that if we select $\lambda \geq \left(\tilde{t}_{\bU}(d,n) + t^2_{\bU}(d,n)\right)\|\bbeta^*\|_1 + 2t_{\bU}(d,n)$, it will follow that $\bbeta^*$ satisfies the constraint with probability at least $1 - 2 d^{2 - cA_{\bU}^2} - 2 e d^{1 - cA^2_{\bU}}$.

The rest of the proof is identical to the proof of Lemma \ref{arxivsupp:main:vdiff} but instead of using Lemma \ref{arxivsupp:main:samplecovRE} we use Lemma \ref{arxivsupp:main:samplecovRELDA}. Thus we omit the proof.
\end{proof}

\begin{lemma} \label{arxivsupp:main:vdiffLDA} Assume that --- $\lambda_{\min}(\bSigma) > 0$, $s_{\vb}(\tilde{t}_{\bU}(d,n) + t^2_{\bU}(d,n)) \leq   (1 - \kappa) \frac{\lambda_{\min}(\bSigma)}{(1 + \xi)^2 }$, where $0 < \kappa < 1$ and $\lambda' \geq  \|\vb^*\|_1(\tilde{t}_{\bU}(d,n) + t^2_{\bU}(d,n))$. Then we have that $\|\hat \vb - \vb^*\|_{1} \leq \frac{8 \lambda' s_{\vb}}{ \operatorname{RE}_{\kappa}(s_{\vb}, 1)} $ with probability at least  $1 - 2 d^{2 - \bar c A_{\bU}^2} - 2 e d^{1 - cA^2_{\bU}}$. (see (\ref{arxivsupp:main:t:t:tilde:def}) for definition of $t_{\bU}$ and $\tilde t_{\bU}$)

\begin{remark} In fact this event happens on the same event as in Lemma \ref{arxivsupp:main:Sigmaclose}.
\end{remark}
\end{lemma}

\begin{proof}[Proof of Lemma \ref{arxivsupp:main:vdiffLDA}]  The proof is identical to the one of Lemma \ref{arxivsupp:main:vdiff} but instead of using Lemma \ref{arxivsupp:main:samplecovRE} we use Lemma \ref{arxivsupp:main:samplecovRELDA}, and we use Lemma \ref{arxivsupp:main:Sigmaclose} instead of using Lemma \ref{arxivsupp:main:samplecovpopcov}. We omit the proof.
\end{proof}

\begin{lemma} \label{arxivsupp:main:sparseLDAineq} We have the following inequality:
$$
\EE |\vb^{*T}\bU^{\otimes 2}\bbeta^* -  \vb^{*T}(\bmu_1 - \bmu_2) + c \vb^{*T}\bU|^k \leq 2^{k-1}\|\vb^*\|_2^k( \|\bbeta^*\|^k_2 (s_{\vb}s)^{k/2} (8k K^2_{\bU})^k +|c|^k s_{\vb}^{k/2} (\sqrt{k}K_{\bU})^k).
$$
\end{lemma}

\begin{proof}[Proof of Lemma \ref{arxivsupp:main:sparseLDAineq}] The argument follows applying standard inequalities, and the details are omitted. 
\end{proof}

\section{Proofs for Vector Autoregressive Models}\label{SVAproofs}

Define the following quantities which will be used throughout. Let:
$$
K_d(\bSigma_0, \Ab) := \frac{32 \|\bSigma_0\|_2 \max_j(\Sigma_{0,jj})}{\min_j (\Sigma_{0,jj})(1 - \|\Ab\|_2)}, ~~~  \tilde K_d(\bSigma_0, \Ab) := K_d(\bSigma_0, \Ab) (2 M + 3).
$$
We set 
\begin{align}
\lambda := \tilde K_d(\bSigma_0, \Ab)\sqrt{\log d/T}, ~~~ \lambda' := \frac{K_d(\bSigma_0, \Ab)}{2} \|\bSigma_0^{-1}\|_1 \left(\sqrt{6\log d/T} + 2 \sqrt{1/T}\right).
\end{align}

\begin{lemma}\label{cond:SVA:implied} Assume that $\bSigma_0 \in \cL, \Ab \in \cM(s)$, $\min_j \Psi_{jj} \geq C > 0$ and $\max(s_{\vb}, s)\log d  = o(\sqrt{T})$. Then the following relationships hold:
\begin{align*}
\lambda = o(1), ~~~ \lambda' = o(1), ~~~\sqrt{T} \max(s_{\vb}, s)\|\bSigma_0^{-1}\|_1\lambda'  \lambda & = o(1),
\end{align*}
\begin{align*}
\Delta \geq C' > 0, ~~ \frac{\bbeta^{*T} \bSigma_0 \bbeta^*}{\Psi_{mm}} & = o(T), ~~ \frac{\|\vb^*\|_1^2}{\vb^{*T}\bSigma_0 \vb^*} \frac{\lambda'}{\|\bSigma_0^{-1}\|_1} = o(1),
\end{align*}
where $C'$ is some positive constant. 
\end{lemma}

\begin{proof}[Proof of Lemma \ref{cond:SVA:implied}]
Clearly, since $M = O(1)$, $\|\bSigma_0^{-1}\|_1 = O(1), K_d(\bSigma_0, \Ab) = O(1)$ and $\max(s_{\vb}, s)\log d  = o(\sqrt{T})$, it follows that $\lambda = o(1)$, $\lambda' = o(1)$, $\lambda' M  = o(1)$ and in addition $\sqrt{T} \max(s_{\vb}, s)\|\bSigma_0^{-1}\|_1\lambda'  \lambda = o(1)$.

By the inequality $(\bSigma_0^{-1})_{jj}\Sigma_{0,jj} \geq 1$ it follows that $(\bSigma_0^{-1})_{0,jj} \geq (\max_{j}\Sigma_{0,jj})^{-1} \geq \|\bSigma_0\|_2^{-1} \geq M^{-1}$. Hence $\Delta = \Psi_{mm}\vb^{*T} \bSigma_{0} \vb^* \geq \Psi_{mm} \min_{jj} \Sigma^{-1}_{0, jj} \geq C M^{-1} > 0$.

Next, to show that $\frac{\bbeta^{*T} \bSigma_0 \bbeta^*}{\Psi_{mm}} = o(T)$, it suffices to see that $\bbeta^{*T} \bSigma_0 \bbeta^* = O(1)$. To this end note that $|\bbeta^{*T} \bSigma_0 \bbeta^*| \leq \|\bSigma_{1, *k}\|_{\infty} \|\Ab\|_1 \leq \|\bSigma_{1, *k}\|_{\infty} M$. Next since $\bPsi = \bSigma_0 - \bSigma_1$, we have $\|\bSigma_{1, *k}\|_{\infty} \leq \max_{j} \Sigma_{0, jj} - \min_{j} \Psi_{jj} \leq \|\bSigma_0\|_2$, which shows that $|\bbeta^{*T} \bSigma_0 \bbeta^*| = O(1)$. 

Finally we check $\frac{\|\vb^*\|_1^2}{\vb^{*T}\bSigma_0 \vb^*} \frac{\lambda'}{\|\bSigma_0^{-1}\|_1} = o(1)$. Note that $\|\vb^*\|_1 \leq \|\bSigma_0^{-1}\|_1$, and also that $\vb^{*T}\bSigma_0 \vb^* \geq \min_j \Sigma_{0, jj} \geq \min_j \Psi_{jj} \geq C$, hence it suffices to show that $\|\vb^*\|_1 \lambda' = o(1)$. However, evidently $\|\vb^*\|_1 \leq \|\bSigma_0^{-1}\|_1 = O(1)$ and $\lambda' = o(1)$, which shows what we wanted.
\end{proof}

Next we summarize several results by \cite{han2014direct}, which we use in the later development.

\begin{theorem}[Theorem 4.1. \cite{han2014direct}]  \label{hanthm2014} Suppose that $(\bX_t)_{t = 1}^T$  from a lag 1 vector autoregressive process $(\bX_t)_{t = -\infty}^\infty$. Assume that $\Ab \in \mathcal{M}(s)$. Let $\hat \Ab$ be the optimizer of (\ref{optproblemvecauto}) with the tunning parameter:
$$
\lambda = \tilde K_d(\bSigma_0, \Ab)\sqrt{\log d/T}. %= \frac{32 \|\bSigma_0\|_2 \max_j(\bSigma_{0,jj})}{\min_j (\bSigma_{0,jj})(1 - \|A\|_2)} (2 M + 3)\sqrt{\log d/T}.
$$
For $T \geq 6 \log d + 1$  and $d \geq 8$, we have, with probability at least $1 - 14 d^{-1}$: $
\|\hat \Ab - \Ab\|_1 \leq 4s \|\bSigma_0^{-1}\|_1\lambda.$
In fact on the same event (see Lemmas A.1. and A.2. \citep{han2014direct}), we have:
\begin{align*}
& \|\Sb_0 - \bSigma_0\|_{\max} \leq K_d(\bSigma_0, \Ab)/2\left(\sqrt{6\log d/T} + 2 \sqrt{1/T}\right),\\
& \|\Sb_1 - \bSigma_1\|_{\max} \leq K_d(\bSigma_0, \Ab)\left(\sqrt{3\log d/T} +  2/T\right).
\end{align*}
\end{theorem}

%\begin{remark}\label{SVA:interpret:rem} Under  additional assumptions $M = O(1)$, $\|\bSigma_0^{-1}\|_1 = O(1), K_d(\bSigma_0, \Ab) = O(1)$, $\min_{j} \Psi_{jj} > C'' > 0$ and $\max_j(\Sigma_{0,jj}) < C'''$ for some constants $C'', C''' > 0$, Remark \ref{ext:cond:SVA:implied} in the Supplementary Material show that the conditions of Corollary \ref{IFexpansionSVA} reduce to $\max(s_{\vb}, s)\log d  = o(\sqrt{T})$.	%{\color{red}if you want to keep this form, we need to add a remark to explain the condition} 
%\end{remark}

\begin{proof}[Proof of Corollary \ref{IFexpansionSVA}]
We verify the conditions of Section \ref{masterthm:sec}. To see Assumption (\ref{consistencyassumpweakcn}), note that by Theorem \ref{hanthm2014} we have $\|\hat \bbeta - \bbeta^*\|_1 \leq 4s \|\bSigma_0^{-1}\|_{1} \lambda$.  Next we inspect $\|\hat \vb - \vb^*\|_1 = O_p(s_{\vb} \|\bSigma_0^{-1}\|_{1} \lambda')$ according to Lemma \ref{lemmavdiffSVA}. 
%By Lemma \ref{lemmavdiffSVA} we get $\|[\hat \vb^T \Sb_0]_{-1}\|_{\infty} = O_p(\lambda')$. %Furthermore $\|\Sb_0 \bbeta^* - \Sb_{1,*m}\|_{\infty} \leq \lambda$ with probability at least $1 - 14 d^{-1}$, as is seen from the proof of Theorem \ref{hanthm2014} (see \citep{han2014direct} for details).
Next we check Assumption \ref{noiseassumpCI}. To see (\ref{betastartassumpCI}), fix a $|\theta - \theta^*| < \epsilon$, for some $\epsilon > 0$. By the triangle inequality:
$$
\|\Sb_0\bbeta^*_{\theta} - \bSigma_0\bbeta_{\theta}^* - \Sb_{1,*m} + \bSigma_{1,*m} \|_{\infty} \leq \|\Sb_0 \bbeta^* - \Sb_{1,*m}\|_{\infty} + \|\Sb_{0,*1} - \bSigma_{0,*1}\|_{\infty}\epsilon
$$
The RHS is $O_p(\lambda)$, by Theorem \ref{hanthm2014} and by the fact that $\|\Sb_0 \bbeta^* - \Sb_{1,*m}\|_{\infty} \leq \lambda$ with probability at least $1 - 14 d^{-1}$, as is seen from the proof of Theorem \ref{hanthm2014} (see \cite{han2014direct} for details). The same logic shows that $r_2(n) \asymp \|\vb^*\|_1 \lambda$, which implies (\ref{betastartassumpCIvstar}). Since the Hessian $\Tb$ in (\ref{lambdaprimeasumpCI}) is free of $\bbeta$ we are allowed to set $r_3(n) = \lambda' = o(1)$. Finally the two expectations in Assumption \ref{noiseassumpCI}, are bounded as we see below:
\begin{align*}
\|\bSigma_{0}\bbeta^*_{\theta} - \bSigma_{0}\bbeta^*\|_{\infty} \leq \|\bSigma_{0,*1} \|_{\infty}\epsilon = \Sigma_{0,11}\epsilon, ~~~~~~ \|\vb^{*T}\bSigma_{0,-1}\|_{\infty}  = 0.
\end{align*}
By adding up the following two identities:
\begin{align*}
\sqrt{T} O_p(4s \|\bSigma_0^{-1}\|_{1} \lambda)O_p(\lambda')  = o_p(1), ~~~~~ \sqrt{T} O_p(s_{\vb} \|\bSigma_0^{-1}\|_{1} \lambda')O_p(\lambda)  = o_p(1),
\end{align*}
we get that (\ref{assumpone}) is also valid in this case by assumption.

To verify the consistency of $\tilde \theta$ we check the assumptions in Theorem \ref{consistency:sol}. Clearly the map $\vb^{*T}\bSigma_{0} (\bbeta^*_{\theta} - \bbeta^*) = (\theta - \theta^*)$ has a unique $0$ when $\theta = \theta^*$. Moreover, the map $\theta \mapsto \hat \vb^T(\bSigma_{0} \hat\bbeta_{\theta} - \Sb_1)$ is continuous as it is linear. In addition, it has a unique zero except in cases when $\hat \vb^T \bSigma_{0,*1} = 0$. However note that $|\hat \vb^T \bSigma_{0,*1} - 1| \leq \lambda'$ by (\ref{lambda:prime:clime}), and hence for small enough values of $\lambda'$ there will exist a unique zero.

Next we verify Assumption \ref{CLTcond} in Lemma \ref{weakconv:SVA}. In addition the fact that $\hat \Delta$ is consistent for $\Delta$ is checked in Proposition \ref{consistentestSVA}. Finally we move on to show (\ref{stabtwo}). Observe that (\ref{stabtwo}) is trivial as its LHS $\equiv 0$ in this case. %Furthermore, (\ref{stabone}) is satisfied with $r_6(n) = \lambda'$. 
\end{proof}

\begin{lemma}\label{weakconv:SVA}
Under the conditions of Corollary \ref{IFexpansionSVA}, we have that:
$$
\Delta^{-1/2}T^{1/2} S(\bbeta^*) = \Delta^{-1/2}T^{1/2}\vb^{*T}(\Sb_0\bbeta^* - \Sb_1) \rightsquigarrow N(0,1),
$$
where the definition of $\Delta$ is given in (\ref{Delta:SVA}).
\end{lemma}

\begin{proof}[Proof of Lemma \ref{weakconv:SVA}] First, construct the sequence $\xi_1 = 0$, $\xi_{t + 1} = \frac{\vb^{*T}\bX_t^{\otimes 2}\bbeta^* - \vb^{*T}\bX_t \bX^T_{t + 1} \eb_m^T}{\sqrt{(T - 1)\Delta}}$ for $t = 1,\ldots, T-1$. We start by showing that the difference between the sequence $\sum_{t = 1}^{T} \xi_t$ and $\frac{\sqrt{T - 1} \vb^{*T}(\Sb_0 \bbeta^* - \Sb_{1,*m})}{\sqrt{\Delta}}$ is asymptotically negligible. We have:
$$
\sum_{t = 1}^{T} \xi_t - \frac{\sqrt{T - 1} \vb^{*T}(\Sb_0 \bbeta^* - \Sb_{1,*m})}{\sqrt{\Delta}} = \underbrace{\frac{(\sqrt{T-1} T)^{-1}}{\sqrt{\Delta}}\sum_{t = 1}^{T} \vb^{*T} \bX_t^{\otimes 2} \bbeta^{*}}_{I_1} - \underbrace{\frac{\vb^{*T}\bX_T^{\otimes 2} \bbeta^*}{\sqrt{T-1}\sqrt{\Delta}}}_{I_2}.
$$
By Lemma \ref{lemmavdiffSVA}, $\|\vb^{*T}\Sb_0 - \eb_1\|_{\infty} \leq \lambda'$ with probability not smaller than $1 - 14d^{-1}$. Thus we have:
$$
|I_1| \leq \frac{\lambda'\|\bbeta^*\|_1 + |\beta^*_{1}|}{\sqrt{T-1}\sqrt{\Delta}} \leq \frac{\lambda'M +|\beta^*_{1}|}{\sqrt{T-1}\sqrt{\Delta}} = o(1),
$$
with probability at least $1 - 14d^{-1}$, where we used the fact that $|\beta^*_1| = O(1)$. 

Next observe that $\EE I_2 = 0$, and using Isserlis' theorem and Cauchy-Schwartz we have:
$$
\Var(I_2) = \frac{1}{T-1}\bigg(\frac{\bbeta^{*T} \bSigma_0 \bbeta^*}{\Psi_{mm}} + \frac{(\vb^{*T}\bSigma_0\bbeta^*)^2}{\vb^{*T}\bSigma_0 \vb^* \Psi_{mm}}\bigg) \leq \frac{2}{T-1} \frac{\bbeta^{*T} \bSigma_0 \bbeta^*}{\Psi_{mm}} = o(1).
$$
and hence $I_2 = o_p(1)$. The last shows that, $\sum_{t = 1}^{T} \xi_t - \frac{\sqrt{T} \vb^{*T}(\Sb_0 \bbeta^* - \Sb_{1,*m})}{\sqrt{\Delta}} = o_p(1),$
provided that $\sum_{t = 1}^T \xi_t = O_p(1)$, which we show next. Observe that the sequence $(\xi_t)_{t = 1}^T$ forms a martingale difference sequence with respect to the filtration $\mathcal{F}_t = \sigma(\bX_1, \ldots, \bX_{t})$ for $t = 1,\ldots, T$, as we clearly have $\EE[\xi_t | \mathcal{F}_{t - 1}] = 0$. Furthermore a simple calculation yields that for $t\geq2$ we have $\EE[\xi^2_t | \mathcal{F}_{t-1}] = \frac{(\vb^{*T}\bX_{t - 1})^2}{(T-1)\vb^{*T}\bSigma_0 \vb^*}$. Thus:
\begin{align*}
\bigg|\sum_{t = 1}^{T} \EE[\xi^2_t | \mathcal{F}_{t-1}] - 1\bigg|&= \bigg|\frac{\vb^{*T}}{(T-1)\vb^{*T}\bSigma_0 \vb^*} \sum_{t = 1}^{T-1}\left[\bX_t^{\otimes 2} - \bSigma_0\right]\vb^*\bigg| \\
&\leq \frac{\|\vb^*\|^2_1}{\vb^{*T}\bSigma_0\vb^*} \underbrace{\bigg\|\frac{1}{T-1}\sum_{t = 1}^{T-1}[\bX_t^{\otimes 2} - \bSigma_0]\bigg\|_{\max}}_{I}.
\end{align*}
Using Theorem \ref{hanthm2014}, it is evident that $I \leq K_d(\bSigma_0, \Ab)/2\left(\sqrt{\frac{6\log d}{T-1}} + 2 \sqrt{\frac{1}{T-1}}\right)$ with probability at least $1 - 14 d^{-1}$, and hence the above quantity converges to 0 in probability. 

Having noted these facts, we want to show that $\sum_{t = 1}^T \xi_t$ converges weakly to a $N(0,1)$ with the help of a version of the martingale central limit theorem (MCLT) \citep{hall1980martingale}. Next we show the Lindeberg condition for the MCLT. For $t \geq 2$ and a fixed $\delta > 0$ we have:
$$\EE[\xi_t^2 1(|\xi_t| \geq \delta) | \mathcal{F}_{t - 1}] = \frac{(\vb^{*T}\bX_{t-1})^2 \EE[Z^2 1(|Z| > \delta C)]}{(T - 1)\vb^{*T}\bSigma_0\vb^*},
$$
where $Z \sim N(0,1)$ and $C =  \left\{\frac{(\vb^{*T}\bX_{t-1})^2}{(T - 1)\vb^{*T}\bSigma_0\vb^{*}}\right\}^{-\frac{1}{2}}$. Using the properties of the truncated standard normal distribution we have that $\EE[Z^2 | Z > c] = 1 + \frac{\phi(c)}{\overline{\Phi}(c)}c$, and hence 
$$\EE[Z^2 1(|Z| > c)] = 2 \overline \Phi(c) \left(1 + \frac{\phi(c)}{\overline{\Phi}(c)}c\right) =  2 \overline \Phi(c) + 2\phi(c) c \leq  2\phi(c) (c^{-1} + c),$$ 
where the last inequality follows from a standard tail bound for the normal distribution.

Now notice that by the union bound and a standard bound on the normal cdf we have 
$$\PP(\max_{t = 1,\ldots, T} |\vb^{*T}\bX_t| > u) \leq 2 T \exp(-u^2/(2\vb^{*T}\bSigma_0 \vb^*)).$$ 
Selecting $u = 2\sqrt{\log(T) \vb^{*T}\bSigma_0 \vb^*}$ gives $\max_t |\vb^{*T} \bX_t| \leq  2\sqrt{\log(T) \vb^{*T}\bSigma_0 \vb^*}$ with probability at least $1 - \frac{2}{T}$ . Hence on this event we have:
\begin{align*}
\EE[\xi_t^2 1(|\xi_t| \geq \delta) | \mathcal{F}_{t - 1}] \leq \frac{8\log T \phi(\delta \tilde C) ((\delta \tilde C)^{-1} + \delta \tilde C)}{(T-1)},
\end{align*}
where $\tilde C = \sqrt{\frac{T - 1}{4\log T}}$, and we used the fact that the function $\phi(x)(x^{-1} + x)$ is decreasing.  Summing up over $t$ yields:
$$
\sum_{t = 1}^T \EE[\xi_t^2 1(|\xi_t| \geq \delta) | \mathcal{F}_{t - 1}]  \leq 8\log T \phi(\delta \tilde C) ((\delta \tilde C)^{-1} + \delta \tilde C) \rightarrow 0.
$$
This shows that the Lindeberg condition holds with probability 1. Hence by the MCLT we can claim $
\sum_{t = 1}^{T} \xi_t \rightsquigarrow N(0,1).$
\end{proof}

\begin{proposition}\label{consistentestSVA} Under the assumptions of Corollary \ref{IFexpansionSVA} we have $\hat \Delta \rightarrow_p \Delta$.
\end{proposition}

\begin{proof}[Proof of Proposition \ref{consistentestSVA}] We begin with showing the consistency of $\hat \Psi_{mm} = \Sb_{0, mm} - \hat \bbeta^T \Sb_0 \hat \bbeta$ is consistent for $\Psi_{mm}$. First note that $\bPsi = \bSigma_0 - \Ab^T\bSigma_0 \Ab$, and thus $\Psi_{mm} = \bSigma_{0,mm} - \bbeta^{*T}\bSigma_0\bbeta^*$. Then we have:
\begin{align*}
|\hat \Psi_{mm} - \Psi_{mm}| & \leq |\Sb_{0, mm} - \bSigma_{0, mm}| + |(\hat\bbeta - \bbeta^*)^T\Sb_0 \hat \bbeta| + |\bbeta^{*T}(\Sb_0 \hat \bbeta - \bSigma_0 \bbeta^*)|.
%& \leq \|\Sb_0 - \bSigma_0\|_{\max} + \|\hat \bbeta - \bbeta^* \|_1 [\|\Sb_0 \bbeta^*\|_{\infty} + \|\Sb_0 \hat \bbeta - \Sb_0 \bbeta^*\|_{\infty}]
\end{align*}
Fitstly, by Theorem \ref{hanthm2014}, we have with probability at least $1 - 14 d^{-1}$:
$$ |\Sb_{0, mm} - \bSigma_{0, mm}| \leq  \|\Sb_0 - \bSigma_0\|_{\max} \leq \frac{K_d(\bSigma_0, \Ab)}{2}\left(\sqrt{6\log d/T} + 2 \sqrt{1/T}\right) = \lambda' \|\bSigma_0^{-1}\|^{-1}_1 = o(1).$$
Secondly:
\begin{align*}
|(\hat\bbeta - \bbeta^*)^T\Sb_0 \hat \bbeta| & \leq \|\hat \bbeta - \bbeta^* \|_1 (\|\Sb_0 \bbeta^*\|_{\infty} + \|\Sb_0 \hat \bbeta - \Sb_0 \bbeta^*\|_{\infty})\\
& \leq \|\hat \bbeta - \bbeta^* \|_1 (\|\Sb_0\|_{\max}\|\bbeta^*\|_1 + \|\Sb_0 \hat \bbeta - \Sb_{1,*m}\|_{\infty} + \|\Sb_0 \bbeta^* - \Sb_{1,*m}\|_{\infty}).
\end{align*}
On the event of Theorem \ref{hanthm2014} we further have:
$$
\|\bbeta - \bbeta^* \|_1 \|\Sb_0\|_{\max}\|\bbeta^*\|_1 \leq 4s  \|\bSigma_0^{-1}\|_1\lambda [\|\bSigma_0\|_{\max} +  \|\Sb_0 - \bSigma_0\|_{\max}] M = o(1).
$$
Furthermore within the proof of Theorem \ref{hanthm2014}, it can be seen that on the event of interest we have $\|\Sb_0 \bbeta^* - \Sb_{1,*m}\|_{\infty} \leq \lambda$, and hence:
$$
\|\hat \bbeta - \bbeta^*\|_{1}( \|\Sb_0 \hat \bbeta - \Sb_{1,*m}\|_{\infty} + \|\Sb_0 \bbeta^* - \Sb_{1,*m}\|_{\infty}) \leq 2 \lambda \|\hat \bbeta - \bbeta^*\|_{1} = o_p(1).
$$
Lastly,
\begin{align*}
|\bbeta^{*T}(\Sb_0 \hat \bbeta - \bSigma_0 \bbeta^*)| & \leq |\bbeta^{*T}\Sb_0 (\hat \bbeta - \bbeta^*)| + |\bbeta^{*T}(\Sb_0 - \bSigma_0)\bbeta^*|\\
&  \leq \|\bbeta^*\|_1 2 \lambda + \|\bbeta^*\|_1 [\|\Sb_0 \bbeta^* - \Sb_{1,*m}\|_{\max} +\| \Sb_{1,*m}  - \bSigma_{1,*m}\|_{\max}]\\
& \leq M 3 \lambda  + M K_d(\bSigma_0, \Ab)\left(\sqrt{3\log d/T} +  2/T\right) = o(1),
\end{align*}
where the last two inequalities hold on the event of Theorem \ref{hanthm2014}, and we used the fact that $\|\bbeta^*\|_1 \leq M$ since $\Ab \in \mathcal{M}(s)$.

Next, we show that $\hat \vb^T \Sb_0 \hat \vb \rightarrow_p \vb^{*T} \bSigma_0 \vb^*$. Similarly to before we have:
$$
|\hat \vb^T \Sb_0 \hat \vb - \vb^* \bSigma_0 \vb^*| \leq |(\hat \vb - \vb^*)^T \Sb_0 \hat \vb| + |\vb^{*T}(\Sb_0 \hat \vb - \bSigma_0 \vb^*)|
$$
For the firs termt we have:
\begin{align*}
|(\hat \vb - \vb^*) \Sb_0 \hat \vb| & \leq \|\hat \vb - \vb^*\|_1\|\hat \vb^T \Sb_0  - \eb_1\|_{\infty} + |\eb_1(\hat \vb - \vb^*)| \leq  4 s_{\vb} \|\bSigma_0^{-1}\|_{1} (\lambda')^2 + \|\hat \vb - \vb^*\|_{\infty}\\
& \leq  4 s_{\vb} \|\bSigma_0^{-1}\|_{1} (\lambda')^2 + \|\bSigma_0^{-1}\|_1 2 \lambda' = o(1),
\end{align*}
with the last two inequalities following from Lemma \ref{lemmavdiffSVA} and holding on the event from Theorem \ref{hanthm2014}. Recall that $\eb$ is a unit row vector.

Finally, for the second term we have:
$$
|\vb^{*T}(\Sb_0 \hat \vb - \bSigma_0 \vb^*)| \leq \|\vb^*\|_1 \lambda' \leq \|\bSigma_0^{-1}\|_1 \lambda'= o(1),
$$
and this concludes the proof.
\end{proof}

\begin{lemma} \label{lemmavdiffSVA} Assume the assumptions in Theorem \ref{hanthm2014}. Let $$\lambda' = \|\bSigma_0^{-1}\|_1 \frac{K_d(\bSigma_0, \Ab)}{2}\left(\sqrt{6\log d/T} + 2 \sqrt{1/T}\right).$$ Then on the same event as in Theorem \ref{hanthm2014}, we have $\|\hat \vb - \vb^*\|_{1} \leq 4 s_{\vb} \|\bSigma_0^{-1}\|_{1} \lambda'$.
\end{lemma}

\begin{proof}[Proof of Lemma \ref{lemmavdiffSVA}] We first start by showing that $\vb^*$ satisfies the constraint in the $\hat \vb$ optimization problem with high probability. According to Theorem \ref{hanthm2014}, we have with probability not smaller than $1 - 14 d^{-1}$:
\begin{align*}
\|\vb^{*T}\Sb_0 - \eb_1\|_{\infty} & = \|\vb^{*T}(\Sb_0 - \bSigma_0)\|_{\infty}  \leq \|\vb^*\|_1 \|\Sb_0 - \bSigma_0\|_{\max} \\
&\leq \|\vb^*\|_1 \frac{K_d(\bSigma_0, \Ab)}{2}\left(\sqrt{6\log d/T} + 2 \sqrt{1/T}\right) \leq \lambda'.
\end{align*}
This implies that $\|\hat \vb \|_1 \leq \|\vb^*\|_1 \leq \|\bSigma_0^{-1}\|_{1}$, and hence similarly to (\ref{arxivsupp:main:l1normineq}) in Lemma \ref{arxivsupp:main:vdiff} in the Supplementary Material we can conclude:
\begin{align} \label{vdiffSVA}
\|\hat \vb_{S^c_{\vb}} -  \vb^*_{S^c_{\vb}}\|_1 \leq \|\hat \vb_{S_{\vb}} -  \vb^*_{S_{\vb}}\|_1 
\end{align}

Next we control $\|\hat \vb -  \vb^*\|_{\infty}$. We have:
\begin{align*} \|\hat \vb - \vb^*\|_{\infty} & = \|(\hat \vb^T \bSigma_0  - \eb_1) \bSigma_0^{-1}\|_{\infty} \leq \|\bSigma_0^{-1}\|_{1} (\|\hat \vb^T \Sb_0  - \eb_1\|_{\infty} + \|\hat \vb\|_1 \|\Sb_0 - \bSigma_0\|_{\max}) \\
& \leq  \|\bSigma_0^{-1}\|_{1} 2\lambda'.
\end{align*}
Combining the last bound with (\ref{vdiffSVA}), we get:
$$
\|\hat \vb - \vb^*\|_{1} \leq 4 s_{\vb} \|\bSigma_0^{-1}\|_{1} \lambda',
$$
which is what we wanted to show.
\end{proof}

\end{document}